\documentclass{article}
\usepackage[english]{babel}
\usepackage{amsmath,amssymb,graphicx,stmaryrd,enumerate,hyperref,mathtools,latexsym,makeidx,xcolor,theorem}
\usepackage[tikz]{mdframed}
\makeindex

\catcode`\<=\active \def<{
\fontencoding{T1}\selectfont\symbol{60}\fontencoding{\encodingdefault}}
\catcode`\>=\active \def>{
\fontencoding{T1}\selectfont\symbol{62}\fontencoding{\encodingdefault}}

\newcommand{\assign}{:=}
\newcommand{\glossaryentry}[3]{\item[{#1}\hfill]#2\dotfill#3}
\newcommand{\infixand}{\text{ and }}
\newcommand{\leangle}{\mathrel{\angle}}
\newcommand{\legeangle}{\mathrel{\substack{\leangle\\\anglege}}}
\newcommand{\leqangle}{\mathrel{\angle \llap {\raisebox{-1ex}{$-$}}}}
\newcommand{\mathe}{\mathrm{e}}
\newcommand{\mathpi}{\pi}
\newcommand{\nobracket}{}
\newcommand{\nosymbol}{}
\newcommand{\suchthat}{:}
\newcommand{\tmabbr}[1]{#1}
\newcommand{\tmaffiliation}[1]{\\ #1}
\newcommand{\tmcolor}[2]{{\color{#1}{#2}}}
\newcommand{\tmem}[1]{{\em #1\/}}
\newcommand{\tmemail}[1]{\\ \textit{Email:} \texttt{#1}}
\newcommand{\tmmathbf}[1]{\ensuremath{\boldsymbol{#1}}}
\newcommand{\tmop}[1]{\ensuremath{\operatorname{#1}}}
\newcommand{\tmrsub}[1]{\ensuremath{_{\textrm{#1}}}}
\newcommand{\tmstrong}[1]{\textbf{#1}}
\newcommand{\tmtextbf}[1]{\text{{\bfseries{#1}}}}
\newcommand{\tmtextit}[1]{\text{{\itshape{#1}}}}
\newcommand{\tmtextrm}[1]{\text{{\rmfamily{#1}}}}
\newcommand{\tmtextsc}[1]{\text{{\scshape{#1}}}}

\newcommand{\tmtextup}[1]{\text{{\upshape{#1}}}}
\newcommand{\widespacing}[1]{#1}
\newenvironment{descriptionaligned}{\begin{description} }{\end{description}}
\newenvironment{descriptioncompact}{\begin{description} }{\end{description}}
\newenvironment{enumerateroman}{\begin{enumerate}[i.] }{\end{enumerate}}
\newenvironment{itemizedot}{\begin{itemize} }{\end{itemize}}
\newenvironment{proof}{\noindent\textbf{Proof\ }}{\hspace*{\fill}$\Box$\medskip}
\newenvironment{theglossary}[1]{\begin{list}{}{\setlength{\labelwidth}{6.5em}\setlength{\leftmargin}{7em}\small} }{\end{list}}
\newtheorem{theorem}{Theorem}[section]
\newtheorem{corollary}[theorem]{Corollary}
\newtheorem{definition}[theorem]{Definition}
\newtheorem{lemma}[theorem]{Lemma}
\hypersetup{
    colorlinks = true,
    linkcolor = {blue},
    citecolor = {cyan}
}
\mdfsetup{linecolor=black,linewidth=0.5pt,skipabove=0.5em,skipbelow=0.5em,hidealllines=true,innerleftmargin=0pt,innerrightmargin=0pt,innertopmargin=0pt,innerbottommargin=0pt}
\newtheorem{proposition}[theorem]{Proposition}
{\theorembodyfont{\rmfamily}\newtheorem{remark}[theorem]{Remark}}
\newmdenv[hidealllines=false,innertopmargin=1ex,innerbottommargin=1ex,innerleftmargin=1ex,innerrightmargin=1ex]{tmframed}

\newcommand{\hl}{\ell}

\begin{document}

\title{The hyperserial field of surreal numbers}

\author{
  Vincent Bagayoko
  \tmaffiliation{imj-prg}, France
  \tmemail{bagayoko@imj-prg.fr}
  \and
  Joris van der Hoeven
  \tmaffiliation{CNRS, LIX}, France
  \tmemail{vdhoeven@lix.polytechnique.fr}
}

\maketitle

\begin{abstract}
  For any ordinal $\alpha > 0$, we show how to define a hyperexponential
  $E_{\omega^{\alpha}}$ and a hyperlogarithm $L_{\omega^{\alpha}}$ on the
  class $\mathbf{No}^{>, \succ}$ of positive infinitely large surreal numbers.
  Such functions are archetypes of extremely fast and slowly growing functions
  at infinity. We also show that the surreal numbers form a so-called
  hyperserial field for our definition.
\end{abstract}

\section{Introduction}

The ordered field $\mathbf{No}$ of surreal numbers was
introduced~{\cite{Con76}}. Conway originally used transfinite recursion to
define both the surreal numbers (henceforth called {\tmem{numbers}}), the
ordering on~{\tmstrong{No}}, and the ring operations. Given two sets $L$ and
$R$ of numbers with $L < R$ ({\tmabbr{i.e.}} $x < y$ for all~$x \in L$ and $y
\in R$), there exists a number $\{ L|R \}$ with
\[ \widespacing{L < \{ L|R \} < R,} \]
and all numbers can be obtained in this way. Given $x = \{ x_L |x_R \}$ and~$y
= \{ y_L |y_R \}$, we have
\begin{eqnarray*}
  x + y & \assign & \{ x_L + y, x + y_L |x_R + y, x + y_R \}
\end{eqnarray*}
and similar recursive formulas exist for $- x$, $xy$ and for deciding whether
$x = y$, $x \leqslant y$, and $x < y$. It~is truly remarkable that
$\mathbf{No}$ turns out to be a totally ordered real-closed field for such
``simple'' definitions~{\cite{Con76}}. The bracket $\{ \nosymbol | \nosymbol
\}$ is called the {\tmem{Conway bracket}}. Using the Conway bracket, we obtain
a surreal number in any traditional Dedekind cut, which allows us to embed
$\mathbb{R}$ into~$\mathbf{No}$. In addition, $\mathbf{No}$ also contains all
ordinal numbers
\[ 0 = \{ \nosymbol | \nosymbol \}, \ 1 = \{ 0| \nosymbol \}, \ 2 = \{
   0, 1| \nosymbol \}, \ \ldots, \omega = \{ 0, 1, 2, \ldots |
   \nosymbol \}, \  \omega + 1 = \{ 0, 1, 2, \ldots, \omega | \nosymbol \}
   \ldots \]
so $\mathbf{No}$ is actually a proper class.

An interesting question is which other operations from calculus can be
extended to the surreal numbers. For instance, Gonshor has shown how to extend
the real exponential function to the surreal numbers~{\cite{Gon86}} and the
exponential field $(\mathbf{No}, \exp)$ turns out to be elementarily
equivalent to $(\mathbb{R}, \exp)$~{\cite{EvdD01}}. Berarducci and Mantova
recently defined a~derivation with respect to $\omega$ on the
surreals~{\cite{BM18}}, again with good model-theoretic
properties~{\cite{vdH:bm}}. In collaboration with Mantova, the authors
constructed a surreal solution to the functional equation
\begin{eqnarray*}
  E_{\omega} (x + 1) & = & \exp E_{\omega} x,
\end{eqnarray*}
which is a bijection of $\mathbf{No}^{>, \succ} \assign \{ x \in \mathbf{No}
\suchthat x >\mathbb{R} \}$ onto itself~{\cite{BvdHM:surhyp}}. We call
$E_{\omega}$ a {\tmem{hyperexponential}} and its functional inverse
$L_{\omega}$ a {\tmem{hyperlogarithm}}.

The first goal of this paper is to extend the results
from~{\cite{BvdHM:surhyp}} to the construction of ``higher'' hyperexponentials
$E_{\omega^{\alpha}} : \mathbf{No}^{>, \succ} \longrightarrow \mathbf{No}^{>,
\succ}$ for any ordinal $\alpha > 1$, together with their functional inverses
$L_{\omega^{\alpha}}$. If $\alpha = \beta + 1$ is a successor ordinal, then
$E_{\omega^{\alpha}}$ satisfies the functional equation
\begin{eqnarray*}
  E_{\omega^{\beta + 1}} (x + 1) & = & E_{\omega^{\beta}} (E_{\omega^{\beta +
  1}} (x)) .
\end{eqnarray*}
Our second goal is to show that these hyperexponentials are ``well-behaved''
in the sense that they endow $\mathbf{No}$ with the structure of a
{\tmem{hyperserial field}} in the sense of~{\cite{BvdHK:hyp}}.

\subsection{Motivation and background}

Whereas it is natural to study surreal exponentiation and differentiation, it
may seem more exotic to define and investigate the properties of surreal
hyperexponentials and hyperlogarithms. In fact, the main motivation behind our
work is a conjecture by the second author~{{\cite[{\tmabbr{p.}}~16]{vdH:ln}}}
and a research program that was laid out in {\cite{vdH:icm}} for proving this
conjecture. The ultimate goal is to expose the deep connections between two
types of mathematical infinities: numerical infinities and growth rates at
infinity. Let us briefly recall the rationale behind this connection.

Cantor's ordinal numbers provide us with a way to count beyond all natural
numbers and keep counting beyond the size of any set. However, ordinal
arithmetic is rather poor in the sense that we have no subtraction or division
and that addition and multiplication do not satisfy the usual laws of
arithmetic, such as commutativity. We may regard Conway's surreal numbers as
providing a calculus with Cantor's ordinal numbers which does extend the usual
calculus with real numbers. In this sense, Conway managed to construct the
ultimate framework for computations with numerical infinities.

Another source for computations with infinitely large quantities stems from
the study of growth rates of real functions at infinity. The first major
results towards a systematic asymptotic calculus of this kind are due to Hardy
in~{\cite{H1910,H1912}}, based on earlier ideas by du
Bois-Reymond~{\cite{dBR1870,dBR1875,dBR1877}}. Hardy defined an
{\tmem{$L$-function}} to be a function constructed from~$x$ and the real
numbers $\mathbb{R}$ using the field operations, exponentiation, and
logarithms. He proved that the germs of $L$-functions at infinity form a
totally ordered field. The framework of $L$-functions is suitable for
asymptotic analysis since we have an ordering for comparing the growth at
infinity of any two such functions. This is often rephrased by saying that
$L$-functions have a {\tmem{regular}} growth at infinity.

Hardy also observed~{\cite[p.~22]{H1910}} that ``The only scales of infinity
that are of any practical importance in analysis are those which may be
constructed by means of the logarithmic and exponential functions.'' In other
words, Hardy suggested that the framework of $L$\mbox{-}functions not only
allows for the development of a systematic asymptotic calculus, but that this
framework is also sufficient for all ``practical'' purposes. Alas, there are
several ``holes''. First of all, the framework is not closed under various
useful operations such as functional inversion and integration. Secondly, the
framework does not contain any functions of extremely fast or slow growth at
infinity, like $E_{\omega}$ and $L_{\omega}$, although such functions
naturally appear in the analysis of certain algorithms; for instance, the best
known algorithm for multiplying two polynomials of degree $n$ in $\mathbb{F}_2
[x]$ runs in time $O (n \log n 4^{L_{\omega} n})$; see~{\cite{HarvdH:mult}}.

This raises the question how to construct a truly universal framework for
computations with regular functions at infinity. Our next candidate is the
class of transseries. A~{\tmem{transseries}} is a~formal object that is
constructed from $x$ (with $x \rightarrow \infty$) and the real numbers, using
exponentiation, logarithms, and {\tmem{infinite}} sums. One example of a
transseries is
\[ \mathe^{\mathe^x + \mathe^{x / 2} + \mathe^{x / 3} + \cdots} - 3
   \mathe^{x^2} + 5 (\log x)^{\mathpi} + 42 + x^{- 1} + 2 x^{- 2} + 6 x^{- 3}
   + 24 x^{- 4} + \cdots + \mathe^{- x} . \]
Depending on conditions satisfied by their supports, there are different types
of transseries. The first constructions of fields of transseries are due to
Dahn and Göring~{\cite{DG87}} and~Écalle~{\cite{Ec92}}. More general
constructions were proposed subsequently by the second author and his former
student Schmeling~{\cite{vdH:phd,vdH:gentr,Schm01}}. Clearly, any $L$-function
is a transseries, but the class of transseries is also closed under
integration and functional inversion, contrary to the class of $L$-functions.

However, the class of transseries still does not contain any hyperexponential
or hyperlogarithmic elements like $E_{\omega} x$ or $L_{\omega} x$. In our
quest for a truly universal framework for asymptotic analysis, we are thus
lead to look beyond: a {\tmem{hyperseries}} is a~formal object that is
constructed from~$x$ and the real numbers, using exponentiation, logarithms,
infinite sums, as well as hyperexponentials $E_{\omega^{\alpha}}$ and
hyperlogarithms $L_{\omega^{\alpha}}$ of any strength $\alpha$. The
hyperexponentials $E_{\omega^{\alpha}}$ and the hyperlogarithms
$L_{\omega^{\alpha}}$ are required to satisfy functional equations
\begin{eqnarray}
  E_{\omega^{\alpha + 1}} \circ T_1 & = & E_{\omega^{\alpha}} \circ
  E_{\omega^{\alpha + 1}} \\
  L_{\omega^{\alpha + 1}} \circ L_{\omega^{\alpha}} & = & T_{- 1} \circ
  L_{\omega^{\alpha + 1}}, 
\end{eqnarray}
where $T_s (u) \assign u + s$. For $\gamma = \sum_{i = 1}^p \omega^{\alpha_i}
n_i$ in Cantor normal form with $\alpha_1 < \cdots < \alpha_p$, we also define
\begin{eqnarray}
  L_{\gamma} & = & L_{\omega^{\alpha_1}}^{\circ n_1} \circ \cdots \circ
  L_{\omega^{\alpha_p}}^{\circ n_p}  \label{eq-Cantor-form-composition}
\end{eqnarray}
and we require that
\begin{eqnarray}
  L_{\gamma}' & = & \frac{1}{\prod_{\beta < \gamma} L_{\beta}} . 
\end{eqnarray}
It is non-trivial to construct fields of hyperseries in which these and
several other technical properties (see
section~\ref{subsection-hyperserial-fields} below) are satisfied. This was
first accomplished by Schmeling for hyperexponentials $E_{\omega^n}$ and
hyperlogarithms $L_{\omega^n}$ of finite strength $n \in \mathbb{N}$. The
general case was tackled in~{\cite{vdH:loghyp,BvdHK:hyp}}.

The construction of general hyperseries relies on the definition of an
abstract notion of {\tmem{hyperserial fields}}. Whereas the hyperseries that
we are really after should actually be hyperseries {\tmem{in}} an infinitely
large variable $x$, abstract hyperserial fields potentially contain
hyperseries that can not be written as infinite expressions in $x$. In the
present paper, we define hyperexponentials~$E_{\omega^{\alpha}}$ and
hyperlogarithms~$L_{\omega^{\alpha}}$ on $\mathbf{No}$ for all ordinals
$\alpha$ and show that this provides $\mathbf{No}$ with the structure of an
abstract hyperserial field. Moreover, given any hyperseries $f$ in $x$ can
naturally be evaluated at $x = \omega$ to produce a surreal number $f
(\omega)$. The conjecture from~{{\cite[{\tmabbr{p.}}~16]{vdH:ln}}} states
that, for a sufficiently general notion of ``hyperseries in $x$'',
{\tmem{all}} surreal numbers can actually be obtained in this way. We plan to
prove this and the conjecture in a follow-up paper.

\subsection{General overview and summary of our new contributions}

Our main goal is to define hyperexponentials $E_{\omega^{\alpha}} :
\mathbf{No}^{>, \succ} \longrightarrow \mathbf{No}^{>, \succ}$ for any ordinal
$\alpha > 1$ and to show that $\mathbf{No}$ is a hyperserial field for these
hyperexponentials. Since our construction builds on a lot of previous work,
the paper starts with three sections of reminders.

In section \ref{section-basic-notions}, we recall basic facts about well-based
series and surreal numbers. In particular, we recall that any surreal number
$x \in \mathbf{No}$ can be regarded as a well-based series
\begin{eqnarray*}
  x & = & \sum_{\mathfrak{m} \in \mathbf{Mo}} x_{\mathfrak{m}} \mathfrak{m}
\end{eqnarray*}
with real coefficients $x_{\mathfrak{m}} \in \mathbb{R}$. The corresponding
group of monomials $\mathbf{Mo}$ consists of those positive numbers
$\mathfrak{m} \in \mathbf{No}^{>}$ that are of the form $\mathfrak{m}= \{
\mathbb{R}^{>} L|\mathbb{R}^{>} R \}$ for certain subsets $L$ and $R$
of~$\mathbf{No}$ with $\mathbb{R}^{>} L <\mathbb{R}^{>} R$.

Section~\ref{subsection-surreal-substructures} is devoted to the theory of
surreal substructures from~{\cite{BvdH19}}. One distinctive feature of the
class of surreal numbers is that it comes with a partial, well-founded order
$\sqsubseteq$, which is called the {\tmem{simplicity}}{\index{simplicity}}
relation. The Conway bracket can then be characterized by the fact that, for
any {\tmem{sets}}~$L$ and~$R$ of surreal numbers with $L < R$, there exists a
unique $\sqsubseteq$-minimal number~$\{ L|R \}$\label{autolab1} with~$L < \{
L|R \} < R$. For many interesting subclasses $\mathbf{S}$ of $\mathbf{No}$, it
turns out that $(\mathbf{S}, \leqslant, \sqsubseteq)$ is actually isomorphic
to $(\mathbf{No}, \leqslant, \sqsubseteq)$, where $(\mathbf{S}, \leqslant,
\sqsubseteq)$ stands for the class $\mathbf{S}$ with the restrictions
of~$\leqslant$ and~$\sqsubseteq$ to~$\mathbf{S}$. Such classes $\mathbf{S}$
are called {\tmem{surreal substructures}} of $\mathbf{No}$ and they come with
their own Conway bracket $\{ \nosymbol | \nosymbol \}_{\mathbf{S}}$.

In section~\ref{subsection-hyperserial-fields}, we recall the definition of
hyperserial fields from~{\cite{BvdHK:hyp}} and the main results on how to
construct such fields. One major fact from~{\cite{BvdHK:hyp}} on which we
heavily rely is that the construction of hyperserial fields can be reduced to
the construction of {\tmem{hyperserial skeletons}}. In the context of the
present paper, this means that it suffices to define the hyperlogarithms
$L_{\omega^{\alpha}}$ only for very special, so called {\tmem{$L_{<
\omega^{\alpha}}$\mbox{-}atomic elements}}.

In the case when $\alpha = 0$, the $L_{< 1}$-atomic elements are simply the
monomials in $\mathbf{Mo}$ and the definition of the general logarithm on
$\mathbf{No}^{>}$ indeed reduces to the definition of the logarithm
on~$\mathbf{Mo}$: given $x \in \mathbf{No}^{>}$, we write $x = c\mathfrak{m}
(1 + \varepsilon)$, where $c \in \mathbb{R}$, $\mathfrak{m} \in \mathbf{Mo}$
and $\varepsilon$ is infinitesimal, and we take $\log x \assign \log
\mathfrak{m}+ \log c + \varepsilon - \varepsilon^2 / 2 + \varepsilon^3 / 3 +
\cdots$. This very special case will be considered in more detail in
section~\ref{section-transserial-structure}.

In the case when $\alpha = 1$, the $L_{< \omega}$-atomic elements of
$\mathbf{No}^{>, \succ}$ are those elements $\mathfrak{a} \in \mathbf{No}^{>,
\succ}$ such that $L_n \mathfrak{a}$ is a monomial for every $n \in
\mathbb{N}$. The construction of $L_{\omega}$ on $\mathbf{No}^{>, \succ}$ then
reduces to the construction of $L_{\omega}$ on the class
$\mathbf{Mo}_{\omega}$ of $L_{< \omega}$-atomic numbers. This particular case
was first dealt with in~{\cite{BvdHM:surhyp}} and this paper can be used as an
introduction to the more general results in the present~paper.

For general ordinals $\alpha$, we say that $\mathfrak{a} \in \mathbf{No}^{>,
\succ}$ is $L_{< \omega^{\alpha}}$-atomic if $L_{\beta} \mathfrak{a}$ is a
monomial for every~{$\beta < \alpha$}. The advantage of restricting ourselves
to such numbers $\mathfrak{a}$ when defining hyperlogarithms is that
$L_{\alpha} \mathfrak{a}$ only needs to verify few requirements with respect
to the ordering. This makes it possible to recursively define $L_{\alpha}
\mathfrak{a}$ using a fairly simple formula:
\begin{eqnarray}
  L_{\alpha} \mathfrak{a} & \assign & \{ \mathbb{R}, L_{\alpha} \mathfrak{a}'
  + (L_{< \alpha} \mathfrak{a}')^{- 1} |L_{\alpha} \mathfrak{a}'' - (L_{<
  \alpha} \mathfrak{a})^{- 1}, L_{< \alpha} \mathfrak{a} \}, 
  \label{Lalpha-def}
\end{eqnarray}
where $\mathfrak{a}', \mathfrak{a}''$ range over $L_{< \alpha}$-atomic numbers
with $\mathfrak{a}', \mathfrak{a}'' \sqsubseteq \mathfrak{a}$ and
$\mathfrak{a}' <\mathfrak{a}<\mathfrak{a}''$; see
also~(\ref{eq-rich-hyperlog}).

In section \ref{section-hyperserial-structure}, we prove that this definition
is warranted and that the resulting functions $L_{\alpha}$ satisfy the axioms
of hyperserial skeletons from~{\cite[Section~3]{BvdHK:hyp}}. Our proof
proceeds by induction on $\alpha$ and also relies on the fact that the class
$\mathbf{Mo}_{\alpha}$ of $L_{< \omega^{\alpha}}$-atomic numbers actually
forms a~surreal substructure of $\mathbf{No}$. Our main result is the
following theorem:

\begin{theorem}
  \label{th-confluent-hyperserial-field}The
  definition~\tmtextup{(\ref{Lalpha-def})} gives $\mathbf{No}$ the structure
  of a confluent hyperserial skeleton in the sense of {\cite{BvdHK:hyp}}.
  Consequently, we may uniquely extend $L_{\omega^{\mu}}$ to $\mathbf{No}^{>,
  \succ}$ in a way that gives $\mathbf{No}$ the structure of a confluent
  hyperserial field. Moreover, for each ordinal $\mu$, the extended function
  {$L_{\omega^{\mu}} : \mathbf{No}^{>, \succ} \longrightarrow \mathbf{No}^{>,
  \succ}$} is bijective.
\end{theorem}

Our final section~\ref{section-useful-identities} is devoted to some further
identities relating the hyperexponential and hyperlogarithm functions and the
simplicity relation $\sqsubseteq$ on $\mathbf{No}$. We also prove the
following more symmetric variant of~(\ref{Lalpha-def}):
\begin{eqnarray}
  L_{\alpha} \mathfrak{a} & = & \{ \mathbb{R}, L_{\alpha} \mathfrak{a}' +
  (L_{< \alpha} \mathfrak{a}')^{- 1} |L_{\alpha} \mathfrak{a}'' - (L_{<
  \alpha} \mathfrak{a}'')^{- 1}, L_{< \alpha} \mathfrak{a} \}, 
  \label{Lalpha-eq}
\end{eqnarray}
where $\mathfrak{a}', \mathfrak{a}''$ again range over the $L_{<
\alpha}$-atomic numbers with $\mathfrak{a}', \mathfrak{a}'' \sqsubseteq
\mathfrak{a}$ and $\mathfrak{a}' <\mathfrak{a}<\mathfrak{a}''$. An interesting
open question is whether there exists an easy argument that would allow us to
use~(\ref{Lalpha-eq}) instead of~(\ref{Lalpha-def}) as a definition of
$L_{\alpha} \mathfrak{a}$.

\section{Basic notions}\label{section-basic-notions}

\subsection{Ordered fields of well-based
series}\label{subsection-well-based-series}

\subsubsection{Well-based series}

Let $(\mathfrak{M}, \times, 1, \prec)$ be a (possibly class-sized) linearly
ordered abelian group. We write $\mathbb{S} \assign \mathbb{R}
[[\mathfrak{M}]]$\label{autolab2} for the class of functions $f : \mathfrak{M}
\longrightarrow \mathbb{R}$ whose support\label{autolab3}
\begin{eqnarray*}
  \tmop{supp} f & \assign & \{ \mathfrak{m} \in \mathfrak{M} \suchthat f
  (\mathfrak{m}) \neq 0 \}
\end{eqnarray*}
is a {\tmem{well-based}} set, i.e. a set which is well-ordered in the reverse
order $(\mathfrak{M}, \succ)$.

We see elements $f$ of $\mathbb{S}$ as formal {\tmem{well-based series}} $f =
\sum_{\mathfrak{m}} f_{\mathfrak{m}} \mathfrak{m}$ where for $\mathfrak{m} \in
\mathfrak{M}$, the term $f_{\mathfrak{m}}$ denotes $f (\mathfrak{m}) \in
\mathbb{R}$. If $\tmop{supp} f \neq \varnothing$, then we define
$\mathfrak{d}_f \assign \max \tmop{supp} f \in \mathfrak{M}$\label{autolab4}
to be the {\tmem{dominant monomial}}{\index{dominant monomial}} of $f$. For
$\mathfrak{m} \in \mathfrak{M}$, we let $f_{\succ \mathfrak{m}} \assign
\sum_{\mathfrak{n} \succ \mathfrak{m}} f_{\mathfrak{n}}
\mathfrak{n}$\label{autolab5} and we write $f_{\succ} \assign f_{\succ
1}$\label{autolab6}. For $f, g \in \mathbb{S}$, we sometimes write $f \oplus g
\assign f + g$\label{autolab7} if $\tmop{supp} g \prec f$. We say that a
series $g \in \mathbb{S}$ is a {\tmem{truncation}}{\index{truncation}} of $f$
and we write $g \trianglelefteqslant f$\label{autolab8} if $\tmop{supp} (f -
g) \succ g$. The relation $\trianglelefteqslant$ is a well-founded partial
order on $\mathbb{S}$ with minimum~$0$.

By {\cite{Hahn1907}}, the class $\mathbb{S}$ is an ordered field under the
pointwise sum
\begin{eqnarray*}
  (f + g) & \assign & \sum_{\mathfrak{m}} (f_{\mathfrak{m}} +
  g_{\mathfrak{m}}) \mathfrak{m},
\end{eqnarray*}
the Cauchy product
\begin{eqnarray*}
  fg & \assign & \sum_{\mathfrak{m}} \left(
  \sum_{\mathfrak{u}\mathfrak{v}=\mathfrak{m}} f_{\mathfrak{u}}
  g_{\mathfrak{v}} \right) \mathfrak{m},
\end{eqnarray*}
(where each sum $\sum_{\mathfrak{u}\mathfrak{v}=\mathfrak{m}} f_{\mathfrak{u}}
g_{\mathfrak{v}}$ has finite support), and where the positive cone
$\mathbb{S}^{>} = \{ f \in \mathbb{S} \suchthat f > 0 \}$ is given by
\begin{eqnarray*}
  \mathbb{S}^{>} & \assign & \{ f \in \mathbb{S} \suchthat f \neq 0 \wedge
  f_{\mathfrak{d}_f} > 0 \} .
\end{eqnarray*}
The identification of $\mathfrak{m} \in \mathfrak{M}$ with the formal series
$\sum_{\mathfrak{n}=\mathfrak{m}} 1 \cdot \mathfrak{n} \in \mathbb{S}$ induces
an ordered group embedding $(\mathfrak{M}, \times, \prec) \longrightarrow
(\mathbb{S}^{>}, \times, <)$.

The order on $\mathfrak{M}$ extends into a strict quasi-order
$\prec$\label{autolab9} on $\mathbb{S}$, which is defined by\label{autolab10}
\label{autolab11}
\begin{eqnarray*}
  f \prec g & \Longleftrightarrow & \mathbb{R}^{>}  | f | < | g |\\
  f \preccurlyeq g & \Longleftrightarrow & \exists r \in \mathbb{R}^{>}, | f |
  \leqslant r | g |\\
  f \asymp g & \Longleftrightarrow & f \preccurlyeq g \preccurlyeq f.
\end{eqnarray*}
When $f, g$ are non-zero, we have $f \prec g$ ({\tmabbr{resp.}} $f
\preccurlyeq g$, {\tmabbr{resp.}} $f \asymp g$) if and only if $\mathfrak{d}_f
\prec \mathfrak{d}_g$ ({\tmabbr{resp.}} $\mathfrak{d}_f \preccurlyeq
\mathfrak{d}_g$, {\tmabbr{resp.}} $\mathfrak{d}_f =\mathfrak{d}_g$). We next
define\label{autolab12} \label{autolab13} \label{autolab14}
\begin{eqnarray*}
  \mathbb{S}_{\succ} & \assign & \{ f \in \mathbb{S} \suchthat \tmop{supp} f
  \subseteq \mathfrak{M}^{\succ} \}\\
  \mathbb{S}^{\prec} & \assign & \{ f \in \mathbb{S} \suchthat \tmop{supp} f
  \subseteq \mathfrak{M}^{\prec} \} \hspace{1.2em} = \hspace{1.2em} \{ f \in
  \mathbb{S} \suchthat f \prec 1 \}\\
  \mathbb{S}^{>, \succ} & \assign & \{ f \in \mathbb{S} \suchthat f
  >\mathbb{R} \} = \{ f \in \mathbb{S} \suchthat f \geqslant 0 \wedge f \succ
  1 \} .
\end{eqnarray*}
Series in $\mathbb{S}_{\succ}$, $\mathbb{S}^{\prec}$ and $\mathbb{S}^{>,
\succ}$ are respectively called {\tmem{purely large}}, {\tmem{infinitesimal}},
and {\tmem{positive infinite}}.{\index{purely large
series}}{\index{infinitesimal series}}{\index{positive infinite series}}

\subsubsection{Well-based families}

Let $(f_i)_{i \in I}$ be a family in $\mathbb{S}$, We say that $(f_i)_{i \in
I}$ is {\tmem{well-based}}{\index{well-based family}} if
\begin{enumerateroman}
  \item $\bigcup_{i \in I} \tmop{supp} f_i$ is well-based, and
  
  \item $\{ i \in I \suchthat \mathfrak{m} \in \tmop{supp} f_i \}$ is finite
  for all $\mathfrak{m} \in \mathfrak{M}$.
\end{enumerateroman}
In that case, we may define the sum $\sum_{i \in I} f_i$ of $(f_i)_{i \in I}$
by
\begin{eqnarray*}
  \sum_{i \in I} f_i & \assign & \sum_{\mathfrak{m}} \left( \sum_{i \in I}
  (f_i)_{\mathfrak{m}} \right) \mathfrak{m}.
\end{eqnarray*}
If $\mathbb{U}=\mathbb{R} [[\mathfrak{N}]]$ is another field of well-based
series and $\Psi : \mathbb{S} \longrightarrow \mathbb{U}$ is
$\mathbb{R}$-linear, then we say that~$\Psi$ is {\tmem{strongly linear}} if
for every well-based family $(f_i)_{i \in I}$ in $\mathbb{S}$, the family
$(\Psi (f_i))_{i \in I}$ is well-based, with
\begin{eqnarray*}
  \Psi \left( \sum_{i \in I} f_i \right) & = & \sum_{i \in I} \Psi (f_i) .
\end{eqnarray*}
\subsection{Surreal numbers}\label{subsection-surreal-numbers}

\subsubsection{Surreal numbers and simplicity}

We define $\mathbf{On}$ to be the class of ordinal numbers.\label{autolab15}
Following {\cite{Gon86}}, we define $\mathbf{No}$ as the class of {\tmem{sign
sequences}}{\index{sign sequence}}
\begin{eqnarray*}
  a \hspace{1.2em} = \hspace{1.2em} (a [\beta])_{\beta < \alpha} & \in & \{ -
  1, + 1 \}^{\alpha}
\end{eqnarray*}
of ordinal {\tmem{length}} $\alpha \in \mathbf{On}$. The terms $a [\beta] \in
\{ - 1, + 1 \}$ are called the {\tmem{signs}} of $a$ and we write $l_a$ for
the length of $a$. Given two numbers $a, b \in \mathbf{No}$, we
define\label{autolab16}
\begin{eqnarray*}
  a \sqsubseteq b & \Longleftrightarrow & l_a \leqslant l_b \wedge (\forall
  \beta \in l_a, a [\beta] = b [\beta]) .
\end{eqnarray*}
We call $\sqsubseteq$ the {\tmem{simplicity relation}} on $\mathbf{No}$ and
note that $(\mathbf{No}, \sqsubseteq)$ is well-founded. See
{\cite[Section~2]{BvdH19}} for more details about the interaction between
$\sqsubseteq$ and the structure of ordered field of $\mathbf{No}$.

Recall that the Conway bracket is characterized by the fact that, for any
{\tmem{sets}}~$L$ and~$R$ of surreal numbers with $L < R$, there exists a
unique $\sqsubseteq$-minimal number~$\{ L|R \}$ with~$L < \{ L|R \} < R$.
Conversely, given a number $a \in \mathbf{No}$, we define
\begin{eqnarray*}
  a_L & \assign & \{ x \in \mathbf{No} \suchthat x \sqsubset a, x < a \}\\
  a_R & \assign & \{ x \in \mathbf{No} \suchthat x \sqsupset a, x > a \} .
\end{eqnarray*}
Then $a$ can canonically be written as
\begin{eqnarray*}
  a & = & \{ a_L |a_R \} .
\end{eqnarray*}
\subsubsection{Ordinals as surreal numbers}\label{ordinal-sec}

The structure $(\mathbf{No}, \sqsubseteq)$ contains an isomorphic copy of
$(\mathbf{On}, \in)$ obtained through the identification of each ordinal
$\alpha$ with the constant sequence $(+ 1)_{\beta < \alpha}$ of length
$\alpha$. We will write $\tmmathbf{\nu} \leqslant \mathbf{On}$ to state that
$\tmmathbf{\nu}$ is either an ordinal or the class of ordinals.

For $\gamma \in \mathbf{On}$, we write $\omega^{\gamma}$\label{autolab17} for
the ordinal exponentiation of base $\omega$ at $\gamma$ and we write
\begin{eqnarray*}
  \omega^{\mathbf{On}} & \assign & \{ \omega^{\gamma} \suchthat \gamma \in
  \mathbf{On} \} .
\end{eqnarray*}
If $\mu \in \mathbf{On}$ is a successor ordinal, then we define $\mu_-$ to be
the unique ordinal with $\mu = \mu_- + 1$.\label{autolab18} We also define
$\mu_- \assign \mu$ if $\mu$ is a limit ordinal. Similarly, if $\alpha =
\omega^{\mu}$, then we write $\alpha_{/ \omega} \assign
\omega^{\mu_-}$\label{autolab19}. Recall that every ordinal $\gamma$ has a
unique Cantor normal form{\index{Cantor normal form}}
\begin{eqnarray*}
  \gamma & = & \omega^{\eta_1} n_1 + \cdots + \omega^{\eta_r} n_r,
\end{eqnarray*}
where $r \in \mathbb{N}$, $n_1, \ldots, n_r \in \mathbb{N}^{> 0}$ and $\eta_1,
\ldots, \eta_r \in \mathbf{On}$ with $\eta_1 > \cdots > \eta_r$.

\subsubsection{Surreal numbers as well-based series}

We define $\mathbf{Mo}$ to be the class of positive numbers $\mathfrak{m} \in
\mathbf{No}^{>}$ that are of the form $\mathfrak{m}= \{ \mathbb{R}^{>}
L|\mathbb{R}^{>} R \}$ for certain subsets $L$ and $R$ of~$\mathbf{No}$ with
$\mathbb{R}^{>} L <\mathbb{R}^{>} R$. Numbers in $\mathbf{Mo}$ are called
{\tmem{monomials}}. It turns out {\cite[Theorem~21]{Con76}} that the monomials
form a subgroup of $(\mathbf{No}^{>}, \times, <)$ and that there is a natural
isomorphism between $\mathbf{No}$ and the ordered field $\mathbb{R}
[[\mathbf{Mo}]]$. We will identify those two fields and thus see $\mathbf{No}$
as a field of well-based series. The ordinal $\omega$, seen as a surreal
number, is the simplest element, or $\sqsubseteq$-minimum, of the class
$\mathbf{No}^{>, \succ}$.

\section{Surreal substructures}\label{subsection-surreal-substructures}

\subsection{Surreal substructures}\label{subsubsection-surreal-substructures}

In {\cite{BvdH19}}, we introduced the notion of {\tmem{surreal
substructures}}{\index{surreal substructure}}. A surreal substructure is a
subclass $\mathbf{S}$ of $\mathbf{No}$ such that $(\mathbf{No}, \leqslant,
\sqsubseteq)$ and $(\mathbf{S}, \leqslant, \sqsubseteq)$ are isomorphic. The
isomorphism $\mathbf{No} \longrightarrow \mathbf{S}$ is unique and denoted
by~$\Xi_{\mathbf{S}}$. Many important subclasses of $\mathbf{No}$ that are
relevant to the study of hyperserial properties of $\mathbf{No}$ are surreal
substructures. In particular, it is known that the following classes are
surreal substructures:
\begin{itemizedot}
  \item The classes $\mathbf{No}^{>}$, $\mathbf{No}^{>, \succ}$ and
  $\mathbf{No}^{\prec}$ of positive, positive infinite and infinitesimal
  numbers.{\nopagebreak}
  
  \item The classes $\mathbf{Mo}$ and $\mathbf{Mo}^{\succ}$ of monomials and
  infinite monomials.
  
  \item The classes $\mathbf{No}_{\succ}$ and $\mathbf{No}_{\succ}^{>}$ of
  purely infinite and positive purely infinite numbers.{\nopagebreak}
  
  \item The class $\mathbf{Mo}_{\omega}$ of log-atomic numbers.
\end{itemizedot}
If $\mathbf{U}, \mathbf{V}$ are surreal substructures, then the class
$\mathbf{U} \mathbin{\Yleft} \mathbf{V} \assign \Xi_{\mathbf{U}}
\mathbf{V}$\label{autolab20} is a surreal substructure with~$\Xi_{\mathbf{U}
\mathbin{\Yleft} \mathbf{V}} = \Xi_{\mathbf{U}} \circ \Xi_{\mathbf{V}}$.

\subsection{Cuts}

Given a subclass $\mathbf{X}$ of $\mathbf{No}$ and $a \in \mathbf{X}$, we will
write
\[ \left. a_L^{\mathbf{X}} \; \assign \; \{ b \in \mathbf{X} \suchthat b < a
   \wedge b \sqsubseteq a \} \right. \infixand \left. a_R^{\mathbf{X}} \;
   \assign \; \{ b \in \mathbf{X} \suchthat b > a \wedge b \sqsubseteq a \}
   \right., \]
so that $a_L \assign a_L^{\mathbf{No}}$ and $a_R \assign a_R^{\mathbf{No}}$.
We also write $a_{\sqsubset}^{\mathbf{X}} \assign a_L^{\mathbf{X}} \cup
a_R^{\mathbf{X}}$ and $a_{\sqsubset} \assign a_{\sqsubset}^{\mathbf{No}}$.

If $\mathbf{X}$ is a subclass of $\mathbf{No}$ and $L, R$ are subsets of
$\mathbf{X}$ with $L < S$, then the class
\begin{eqnarray*}
  (L|R)_{\mathbf{X}} & \assign & \{ a \in \mathbf{X} \suchthat (\forall l \in
  L, l < a) \wedge (\forall r \in R, a < r) \}
\end{eqnarray*}
is called a {\tmem{cut}} in $\mathbf{X}$. If $(L|R)_{\mathbf{X}}$ contains a
unique simplest element, then we denote this element by~$\{ L|R
\}_{\mathbf{X}}$ and say that $(L, R)$ is a {\tmem{cut
representation}}{\index{cut representation}} (of $\{ L|R \}_{\mathbf{X}}$) in
$\mathbf{X}$. These notations naturally extend to the case when $\mathbf{L}$
and $\mathbf{R}$ are subclasses of $\mathbf{X}$ with $\mathbf{L}<\mathbf{R}$.

A surreal substructure $\mathbf{S}$ may be characterized as a subclass of
$\mathbf{No}$ such that for all cut representations $(L, R)$ in $\mathbf{S}$,
the cut $(L|R)_{\mathbf{S}}$ has a unique simplest element
{\cite[Proposition~4.7]{BvdH19}}.

Let $\mathbf{S}$ be a surreal substructure. Note that we have $a = \{
a_L^{\mathbf{S}} |a_R^{\mathbf{S}} \}$ for all $a \in \mathbf{S}$. Let $a \in
\mathbf{S}$ and let~{$(L, R)$} be a cut representation of $a$ in $\mathbf{S}$.
Then $(L, R)$ is {\tmem{cofinal with respect to}}{\index{cofinal with respect
to}} $(a_L^{\mathbf{S}}, a_R^{\mathbf{S}})$ in the sense that $L$ has no
strict upper bound in $a_L^{\mathbf{S}}$ and $R$ has no strict lower bound in
$a_R^{\mathbf{S}}$ {\cite[Proposition~4.11(b)]{BvdH19}}.

Given numbers $a, b \in \mathbf{No}$ with $a \leqslant b$, the number $c
\assign \{ a_L |b_R \}$ is the unique $\sqsubseteq$-maximal number with $c
\sqsubseteq a, b$. We have $a \leqslant c \leqslant b$. Let $\mathbf{S}$ be a
surreal substructure. Considering the isomorphism $\Xi_{\mathbf{S}} :
(\mathbf{No}, \leqslant, \sqsubseteq) \longrightarrow (\mathbf{S}, \leqslant,
\sqsubseteq)$, we see that for all $a, b \in \mathbf{S}$ with $a \leqslant b$,
there is a unique $\sqsubseteq$-maximal element $c$ of $\mathbf{S}$ with $c
\sqsubseteq a, b$, and we have $a \leqslant c \leqslant b$. In what follows,
we will use this basic fact several times without further mention.

\subsection{Cut equations}

Let $\mathbf{X} \subseteq \mathbf{No}$ be a subclass, let $\mathbf{T}$ be a
surreal substructure and $F : \mathbf{X} \longrightarrow \mathbf{T}$ be a
function. Let $\lambda, \rho$ be functions defined for cut representations in
$\mathbf{X}$ and such that $\lambda (L, R), \rho (L, R)$ are subsets of
$\mathbf{T}$ whenever $(L, R)$ is a cut representation in $\mathbf{X}$. We say
that $(\lambda, \rho)$ is a {\tmem{cut equation}}{\index{cut equation}} for
$F$ if for all $a \in \mathbf{X}$, we have
\[ \lambda (a_L^{\mathbf{X}}, a_R^{\mathbf{X}}) \; < \; \rho
   (a_L^{\mathbf{X}}, a_R^{\mathbf{X}}), \qquad F (a) \; = \; \{ \lambda
   (a_L^{\mathbf{X}}, a_R^{\mathbf{X}}) | \rho (a_L^{\mathbf{X}},
   a_R^{\mathbf{X}}) \}_{\mathbf{T}} . \]
Elements in $\lambda (a_L^{\mathbf{X}}, a_R^{\mathbf{X}})$ (resp. $\rho
(a_L^{\mathbf{X}}, a_R^{\mathbf{X}})$) are called {\tmem{left}}
({\tmabbr{resp.}} {\tmem{right}}) {\tmem{options}}{\index{left option, right
option}} of this cut equation at~$a$.

We say that the cut equation is {\tmem{uniform}}{\index{uniform equation}} if
we have
\[ \lambda (L, R) \; < \; \rho (L, R), \quad F (\{ L|R \}_{\mathbf{X}}) \; =
   \; \{ \lambda (L, R) | \rho (L, R) \}_{\mathbf{T}} \]
whenever $(L, R)$ is a cut representation in $\mathbf{X}$. For instance, given
$r \in \mathbb{R}$, consider the translation $T_r : \mathbf{No}
\longrightarrow \mathbf{No} ; a \longmapsto a + r$ on $\mathbf{No}$. By
{\cite[Theorem~3.2]{Gon86}}, we have the following uniform cut equation for
$T_r$ on $\mathbf{No}$:
\begin{equation}
  \forall a \in \mathbf{No}, \quad a + r = \{ a_L + r, a + r_L |a + r_R, a_R +
  r \} . \label{eq-uniform-sum} \text{}
\end{equation}
We will need the following result from~{\cite{BvdH19}}:

\begin{proposition}
  \label{prop-nearly-extensive}\tmtextup{{\cite[Proposition~4.36]{BvdH19}}}
  Let $\mathbf{S}, \mathbf{T}$ be surreal substructures. Let $\Lambda$ be a
  function from $\mathbf{S}$ to the class of subsets of $\mathbf{T}$ such that
  for $x, y \in \mathbf{S}$ with $x < y$, the set $\Lambda (y)$ is cofinal
  with respect to $\Lambda (x)$. For $x \in \mathbf{S}$, let
  $\tmmathbf{\Lambda} [x]$ denote the class of elements $u$ of $\mathbf{S}$
  such that $\Lambda (x)$ and $\Lambda (u)$ are mutually cofinal. Let $\{
  \lambda | \rho \}_{\mathbf{T}}$ be a cut equation on $\mathbf{S}$ that is
  extensive in the sense that
  \[ \forall x, y \in \mathbf{S}, \quad (x \sqsubseteq y \Longrightarrow
     (\lambda (x_L^{\mathbf{S}}, x_R^{\mathbf{S}}) \subseteq \lambda
     (y_L^{\mathbf{S}}, y_R^{\mathbf{S}}) \wedge \rho (x_L^{\mathbf{S}},
     x_R^{\mathbf{S}}) \subseteq \rho (y_L^{\mathbf{S}}, y_R^{\mathbf{S}}))) .
  \]
  Let $F : \mathbf{S} \longrightarrow \mathbf{T}$ be strictly increasing with
  cut equation
  \[ \forall x \in \mathbf{S}, \quad F (x) = \{ \Lambda (x), \lambda
     (x_L^{\mathbf{S}}, x_R^{\mathbf{S}}) | \rho (x_L^{\mathbf{S}},
     x_R^{\mathbf{S}}) \}_{\mathbf{T}} . \]
  Then $F$ induces an embedding $(\tmmathbf{\Lambda} [x], \leqslant,
  \sqsubseteq) \longrightarrow (\mathbf{T}, \leqslant, \sqsubseteq)$ for each
  element $x$ of $\mathbf{S}$.
\end{proposition}

\subsection{Convex partitions}

One natural way to obtain surreal substructures is \tmtextit{via} convex
partitions. If $\mathbf{S}$ is a surreal substructure, then a {\tmem{convex
partition}}{\index{convex partition}} of $\mathbf{S}$ is a partition
$\tmmathbf{\Pi}$ of $\mathbf{S}$ whose members are convex subclasses of
$\mathbf{S}$ for the order $\leqslant$. We may then consider the class
$\mathbf{Smp}_{\tmmathbf{\Pi}}$\label{autolab21} of simplest elements (i.e.
$\sqsubseteq$-minima) in each member of~$\tmmathbf{\Pi}$. Those elements are
said $\tmmathbf{\Pi}${\tmem{-simple}}{\index{$\tmmathbf{\Pi}$-simple
element}}. For $a \in \mathbf{S}$, we let $\tmmathbf{\Pi} [a]$ denote the
unique member of $\tmmathbf{\Pi}$ containing~$a$. By
{\cite[Proposition~4.16]{BvdH19}}, the class $\tmmathbf{\Pi} [a]$ contains a
unique $\tmmathbf{\Pi}$-simple element, which we denote by
$\pi_{\tmmathbf{\Pi}} (a)$\label{autolab22}. The function
$\pi_{\tmmathbf{\Pi}}$ is a surjective non-decreasing function $\mathbf{S}
\longrightarrow \mathbf{Smp}_{\tmmathbf{\Pi}}$ with $\pi_{\tmmathbf{\Pi}}
\circ \pi_{\tmmathbf{\Pi}} = \pi_{\tmmathbf{\Pi}}$.

Given $a, b \in \mathbf{Smp}_{\tmmathbf{\Pi}}$, note that we have $a < b$ if
and only if $\tmmathbf{\Pi} [a] <\tmmathbf{\Pi} [b]$. For $\mathbf{X}
\subseteq \mathbf{No}$, we write $\tmmathbf{\Pi} [\mathbf{X}] = \bigcup_{a \in
\mathbf{X}} \tmmathbf{\Pi} [a]$. We have the following criterion to
characterize elements of $\mathbf{Smp}_{\tmmathbf{\Pi}}$.

\begin{proposition}
  \label{prop-simplicity-condition}{\tmem{{\cite[Lemma~6.5]{BvdH19}}}} An
  element $a$ of $\mathbf{S}$ is $\tmmathbf{\Pi}$-simple if and only if there
  is a cut representation $(L, R)$ of $a$ in $\mathbf{S}$ with $\tmmathbf{\Pi}
  [L] < a <\tmmathbf{\Pi} [R]$. Equivalently $a \in \mathbf{S}$ is
  $\tmmathbf{\Pi}$-simple if and only if~$\tmmathbf{\Pi} [a_L^{\mathbf{S}}] <
  a <\tmmathbf{\Pi} [a_R^{\mathbf{S}}]$.
\end{proposition}

We say that $\tmmathbf{\Pi}$ is {\tmem{thin}}{\index{thin convex partition}}
if each member of $\tmmathbf{\Pi}$ has a cofinal and coinitial subset. We then
have:

\begin{proposition}
  \label{prop-thin-substructure}{\tmem{{\cite[Theorem~6.7 and
  Proposition~6.8]{BvdH19}}}} If $\tmmathbf{\Pi}$ is thin, then the class
  $\mathbf{Smp}_{\tmmathbf{\Pi}}$ is a~surreal substructure and
  $\Xi_{\mathbf{Smp}_{\tmmathbf{\Pi}}}$ has the following uniform cut
  equation:
  \[ \forall z \in \mathbf{No}, \quad \; \Xi_{\mathbf{Smp}_{\tmmathbf{\Pi}}} z
     = \{ \tmmathbf{\Pi} [\Xi_{\mathbf{Smp}_{\tmmathbf{\Pi}}} z_L]
     |\tmmathbf{\Pi} [\Xi_{\mathbf{Smp}_{\tmmathbf{\Pi}}} z_R] \}_{\mathbf{S}}
     . \]
\end{proposition}

\subsection{Function groups}

A special type of thin convex partitions is that of partitions induced by
function groups acting on surreal substructures. A {\tmem{function
group}}{\index{function group}} $\mathcal{G}$ on a surreal substructure
$\mathbf{S}$ is a set-sized group of strictly increasing bijections
$\mathbf{S} \longrightarrow \mathbf{S}$ under functional composition. We see
elements $f, g$ of $\mathcal{G}$ as actions on $\mathbf{S}$ and we sometimes
write $fg$ and $f a$ instead of $f \circ g$ and~$f (a)$, where $a \in
\mathbf{S}$.

For such a function group $\mathcal{G}$, the collection
$\tmmathbf{\Pi}_{\mathcal{G}}$ of classes\label{autolab23}
\begin{eqnarray*}
  \mathcal{G} [a] & \assign & \{ b \in \mathbf{S} \suchthat \exists f, g \in
  \mathcal{G}, f a \leqslant b \leqslant g a \}
\end{eqnarray*}
with $a \in \mathbf{S}$ is a thin convex partition of $\mathbf{S}$. We write
$\mathbf{Smp}_{\mathcal{G}} \assign
\mathbf{Smp}_{\tmmathbf{\Pi}_{\mathcal{G}}}$. We have the uniform cut equation
\begin{equation}
  \forall z \in \mathbf{No}, \quad \; \Xi_{\mathbf{Smp}_{\mathcal{G}}} z = \{
  \mathcal{G} \Xi_{\mathbf{Smp}_{\mathcal{G}}} z_L |\mathcal{G}
  \Xi_{\mathbf{Smp}_{\mathcal{G}}} z_R \}_{\mathbf{S}} . \label{eq-uniform-G}
\end{equation}
Consider sets $X, Y$ of strictly increasing bijections $\mathbf{S}
\longrightarrow \mathbf{S}$, then we say that $Y$ is {\tmem{pointwise
cofinal}} with respect to $X$, and we write $X \leqangle Y$\label{autolab24},
if we have $\forall f \in X, \forall a \in \mathbf{S}, \exists g \in Y, f a
\leqslant g a$. We also define\label{autolab25}
\begin{eqnarray*}
  \langle X \rangle & \assign & \{ f_0 \circ f_1 \circ \cdots \circ f_n
  \suchthat n \in \mathbb{N}, f_0, \ldots, f_n \in X \cup X^{- 1} \} .
\end{eqnarray*}
It is easy to see that $\langle X \rangle$ is a function group on $\mathbf{S}$
and that we have $\langle X \rangle \leqangle \langle Y \rangle$ if $X
\leqangle Y$ or $X^{- 1} \leqangle Y^{- 1}$. The relation $\langle X \rangle
\leqangle \langle Y \rangle$ trivially implies $\mathbf{Smp}_{\langle Y
\rangle} \subseteq \mathbf{Smp}_{\langle X \rangle}$. If $X \leqangle Y$ and
$Y \leqangle X$, then we say that $X$ and $Y$ are {\tmem{mutually pointwise
cofinal}} and we write $X \legeangle Y$\label{autolab26}. We then have
$\mathbf{Smp}_{\langle X \rangle} = \mathbf{Smp}_{\langle Y \rangle}$.

We write $X \leqslant Y$ {\tmabbr{(resp.}} $X < Y$) if we have $\forall a \in
\mathbf{S}, \forall f \in X, \forall g \in Y, f a \leqslant g a$
{\tmabbr{(resp.}} $\forall a \in \mathbf{S}, \forall f \in X, \forall g \in Y,
f a < g a$). We also write $f < Y$ and $X < g$ instead of $\{ f \} < Y$ and $X
< \{ g \}$.

Given a function group $\mathcal{G}$ on $\mathbf{S}$, the relation defined by
$f < g \Longleftrightarrow \{ f \} < \{ g \}$ is a partial order
on~$\mathcal{G}$. We will frequently rely on the basic fact that
$(\mathcal{G}, <)$ is {\tmem{partially bi-ordered}} in the sense that
\[ \forall f, g, h \in \mathcal{G}, \quad \tmop{id}_{\mathbf{S}} < g
   \Longleftrightarrow f h < f g h. \]

\subsection{Remarkable function groups}

Each of the examples of surreal substructures from
Subsection~\ref{subsubsection-surreal-substructures} can be regarded as the
classes $\mathbf{Smp}_{\mathcal{G}}$ for actions of the following function
groups $\mathcal{G}$ acting on $\mathbf{No}$, $\mathbf{No}^{>}$ or
$\mathbf{No}^{>, \succ}$. For $c \in \mathbb{R}$ and $r \in \mathbb{R}^{>}$,
we define\label{autolab27} \label{autolab28} \label{autolab29}
\begin{eqnarray*}
  T_r & \assign & a \longmapsto a + c \hspace{2.2em} \text{acting on
  $\mathbf{No}$ or $\mathbf{No}^{>, \succ}$.}\\
  H_c & \assign & a \longmapsto ra \hspace{3em} \text{ acting on
  $\mathbf{No}^{>}$ or $\mathbf{No}^{>, \succ}$.}\\
  P_c & \assign & a \longmapsto a^r \hspace{3em} \text{ \,acting on
  $\mathbf{No}^{>}$ or $\mathbf{No}^{>, \succ}$.}
\end{eqnarray*}
Now consider\label{autolab30} \label{autolab31} \label{autolab32}
\label{autolab33} \label{autolab34}
\begin{eqnarray*}
  \mathcal{T} & \assign & \{ T_c \suchthat c \in \mathbb{R} \},\\
  \mathcal{H} & \assign & \{ H_r \suchthat r \in \mathbb{R}^{>} \},\\
  \mathcal{P} & \assign & \{ P_r \suchthat r \in \mathbb{R}^{>} \},\\
  \mathcal{E}' & \assign & \langle E_n H_r L_n : n \in \mathbb{N}, r \in
  \mathbb{R}^{>} \rangle, \hspace{1.0em} \tmop{and}\\
  \mathcal{E}^{\ast} & \assign & \{ E_n, L_n \suchthat n \in \mathbb{N} \} .
\end{eqnarray*}
Then we have the following list of correspondences $\mathcal{G} \longmapsto
\mathbf{Smp}_{\mathcal{G}}$:
\begin{itemizedot}
  \item The action of $\mathcal{T}$ on $\mathbf{No}$ ({\tmabbr{resp.}}
  $\mathbf{No}^{>, \succ}$) yields $\mathbf{No}_{\succ}$ ({\tmabbr{resp.}}
  $\mathbf{No}_{\succ}^{>}$), e.g. $\mathbf{Smp}_{\mathcal{T}} =
  \mathbf{No}_{\succ}$.
  
  \item The action of $\mathcal{H}$ on $\mathbf{No}^{>}$ ({\tmabbr{resp.}}
  $\mathbf{No}^{>, \succ}$) yields $\mathbf{Mo}$ ({\tmabbr{resp.}}
  $\mathbf{Mo}^{\succ}$).
  
  \item The action of $\mathcal{P}$ on $\mathbf{No}^{>, \succ}$ yields
  $\mathbf{Mo} \mathbin{\Yleft} \mathbf{Mo} = E_1  \mathbf{Mo}^{\succ}$.
  
  \item The action of $\mathcal{E}'$ on $\mathbf{No}^{>, \succ}$ yields
  $\mathbf{Mo}_{\omega}$.
  
  \item The action of $\mathcal{E}^{\ast}$ on $\mathbf{No}^{>, \succ}$ yields
  $\mathbf{K} \assign \mathbf{Mo}_{\omega} \mathbin{\Yleft}
  \mathbf{No}_{\succ}$ (which will coincide with $E_{\omega} 
  \mathbf{No}_{\succ}^{>}$).
\end{itemizedot}
Generalizations of those function groups will allow us to define certain
surreal substructures related to the hyperlogarithms and hyperexponentials on
$\mathbf{No}$.

\section{Hyperserial fields}\label{subsection-hyperserial-fields}

In this section, we briefly recall the definition of hyperserial fields
from~{\cite{BvdHK:hyp}} and how to construct such fields from their
hyperserial skeletons.

\subsection{Logarithmic hyperseries}

Let $x$ be a formal, infinitely large indeterminate. The field
$\mathbb{L}$\label{autolab35} of {\tmem{logarithmic
hyperseries}}{\index{logarithmic hyperseries}} of {\cite{vdH:loghyp}} is the
smallest field of well-based series that contains all ordinal real power
products of the hyperlogarithms $L_{\alpha} x$ with $\alpha \in \mathbf{On}$.
It is naturally equipped with a derivation $\partial : \mathbb{L}
\longrightarrow \mathbb{L}$ and composition law $\circ : \mathbb{L} \times
\mathbb{L}^{>, \succ} \longrightarrow \mathbb{L}$.

\paragraph{Definition}Let $\alpha$ be an ordinal. For each~$\gamma < \alpha$,
we introduce the formal hyperlogarithm {$\hl_{\gamma} \assign L_{\gamma} x$}
and define $\mathfrak{L}_{< \alpha}$\label{autolab36} to be the group of
formal power products {$\mathfrak{l}= \prod_{\gamma < \alpha}
\hl_{\gamma}^{\mathfrak{l}_{\gamma}}$} with $\mathfrak{l}_{\gamma} \in
\mathbb{R}$. This group comes with a monomial ordering~$\succ$ that is defined
by
\[ \mathfrak{l} \succ 1 \Longleftrightarrow \mathfrak{l}_{\beta} > 0
   \text{\quad for $\beta = \min \{ \gamma < \alpha \suchthat
   \mathfrak{l}_{\gamma} \neq 0 \}$} . \]
We define $\mathbb{L}_{< \alpha}$ to be the ordered field of well-based series
$\mathbb{L}_{< \alpha} \assign \mathbb{R} [[\mathfrak{L}_{<
\alpha}]]$\label{autolab37}. If $\alpha, \beta$ are ordinals with $\beta <
\alpha$, then we define $\mathfrak{L}_{[\beta, \alpha)}$ to be the subgroup of
$\mathfrak{L}_{< \alpha}$ of monomials $\mathfrak{l}$ with
$\mathfrak{l}_{\gamma} = 0$ whenever $\gamma < \beta$. As in
{\cite{vdH:loghyp}}, we write
\begin{eqnarray*}
  \mathbb{L}_{[\beta, \alpha)} & \assign & \mathbb{R} [[\mathfrak{L}_{[\beta,
  \alpha)}]],\\
  \mathfrak{L} & \assign & \bigcup_{\alpha \in \mathbf{On}} \mathfrak{L}_{<
  \alpha},\\
  \mathbb{L} & \assign & \mathbb{R} [[\mathfrak{L}]] .
\end{eqnarray*}
We have natural inclusions $\mathfrak{L}_{[\beta, \alpha)} \subseteq
\mathfrak{L}_{< \alpha} \subset \mathfrak{L}$, hence natural inclusions
$\mathbb{L}_{[\beta, \alpha)} \subseteq \mathbb{L}_{< \alpha} \subset
\mathbb{L}$.

\paragraph{Derivation on $\mathbb{L}_{< \alpha}$}The field $\mathbb{L}_{<
\alpha}$ is equipped with a derivation $\partial : \mathbb{L}_{< \alpha}
\longrightarrow \mathbb{L}_{< \alpha}$ which satisfies the Leibniz rule and
which is strongly linear. Write $\hl_{\gamma}^{\dag} \assign \prod_{\iota
\leqslant \gamma} \hl_{\iota}^{- 1} \in \mathfrak{L}_{< \alpha}$ for all
$\gamma < \alpha$. The derivative of a logarithmic hypermonomial $\mathfrak{l}
\in \mathfrak{L}_{< \alpha}$ is defined by
\begin{eqnarray*}
  \partial \mathfrak{l} & \assign & \left( \sum_{\gamma < \alpha}
  \mathfrak{l}_{\gamma}  \hl_{\gamma}^{\dag} \right) \mathfrak{l}.
\end{eqnarray*}
So $\partial \hl_{\gamma} = \frac{1}{\prod_{\iota < \gamma} \hl_{\iota}}$ for
all $\gamma < \alpha$. For $f \in \mathbb{L}_{< \alpha}$ and $k \in
\mathbb{N}$, we will sometimes write $f^{(k)} \assign \partial^k f$.

\paragraph{Composition on $\mathbb{L}_{< \alpha}$}Assume that $\alpha =
\omega^{\nu}$ for a certain ordinal $\nu$. Then the field $\mathbb{L}_{<
\alpha}$ is equipped with a composition $\circ : \mathbb{L}_{< \alpha} \times
\mathbb{L}_{< \alpha}^{>, \succ} \longrightarrow \mathbb{L}_{< \alpha}$ that
satisfies in particular:
\begin{itemizedot}
  \item For $g \in \mathbb{L}_{< \alpha}^{>, \succ}$, the map $\mathbb{L}_{<
  \alpha} \longrightarrow \mathbb{L}_{< \alpha} ; f \longmapsto f \circ g$ is
  a strongly linear embedding {\cite[Lemma~6.6]{vdH:loghyp}}.
  
  \item For $f \in \mathbb{L}_{< \alpha}$ and $g, h \in \mathbb{L}_{<
  \alpha}^{>, \succ}$, we have $g \circ h \in \mathbb{L}_{< \alpha}^{>,
  \succ}$ and $f \circ (g \circ h) = (f \circ g) \circ h$
  {\cite[Proposition~7.14]{vdH:loghyp}}.
  
  \item For $g \in \mathbb{L}_{< \alpha}^{>, \succ}$ and successor ordinals
  $\mu < \nu$, we have $\hl_{\omega^{\mu}} \circ \hl_{\omega^{\mu_-}} =
  \hl_{\omega^{\mu}} - 1$ {\cite[Lemma~5.6]{vdH:loghyp}}.
\end{itemizedot}
The same properties hold for the composition $\circ : \mathbb{L} \times
\mathbb{L}^{>, \succ} \longrightarrow \mathbb{L}$ if $\alpha$ is replaced by
$\mathbf{On}$. For $\gamma < \alpha$, the map $\mathbb{L}_{< \alpha}
\longrightarrow \mathbb{L}_{< \alpha} ; f \longmapsto f \circ \hl_{\gamma}$ is
injective, with image $\mathbb{L}_{[\gamma, \alpha)}$
{\cite[Lemma~5.11]{vdH:loghyp}}. For $g \in \mathbb{L}_{[\gamma, \alpha)}$, we
define $g^{\mathord{\uparrow} \gamma}$\label{autolab38} to be the unique
series in $\mathbb{L}_{< \alpha}$ with~$g^{\mathord{\uparrow} \gamma} \circ
\hl_{\gamma} = g$.

\subsection{Hyperserial fields}

Let $\mathfrak{M}$ be an ordered group. A {\tmem{real powering operation}} on
$\mathfrak{M}$ is a law
\[ \mathbb{R} \times \mathfrak{M} \longrightarrow \mathfrak{M}; (r,
   \mathfrak{m}) \longmapsto \mathfrak{m}^r \]
of ordered $\mathbb{R}$-vector space on $\mathfrak{M}$. Let
$\mathbb{T}=\mathbb{R} [[\mathfrak{M}]]$ be a field of well-based series with
$\mathfrak{M} \neq 1$, let $\tmmathbf{\nu} \leqslant \mathbf{On}$, and let
$\circ : \mathbb{L} \times \mathbb{T}^{>, \succ} \longrightarrow \mathbb{T}$
be a function. For $\tmmathbf{\mu} \leqslant \tmmathbf{\nu}$ , we define
$\mathfrak{M}_{\omega^{\tmmathbf{\mu}}}$ to be the class of series $s \in
\mathbb{T}^{>, \succ}$ with $\forall \gamma < \omega^{\tmmathbf{\mu}},
\hl_{\gamma} \circ s \in \mathfrak{M}^{\succ}$. We say that $(\mathbb{T},
\circ)$ is a {\tmem{hyperserial field}}{\index{hyperserial field}} if

\begin{descriptionaligned}
  \item[HF1] \label{HF1}$\mathbb{L} \longrightarrow \mathbb{T}; f \longmapsto
  f \circ s$ is a strongly linear morphism of ordered rings for each $s \in
  \mathbb{T}^{>, \succ}$.
  
  \item[HF2] $f \circ (g \circ s) = (f \circ g) \circ s$ for all $f \in
  \mathbb{L}$, $g \in \mathbb{L}^{>, \succ}$, and $s \in \mathbb{T}^{>,
  \succ}$.
  
  \item[HF3] $f \circ (t + \delta) = \sum_{k \in \mathbb{N}} \frac{f^{(k)}
  \circ t}{k!} \delta^k$ for all $f \in \mathbb{L}$, $t \in \mathbb{T}^{>,
  \succ}$, and $\delta \in \mathbb{T}$ with $\delta \prec t$.
  
  \item[HF4] $\hl_{\omega^{\mu}}^{\mathord{\uparrow} \gamma} \circ s <
  \hl_{\omega^{\mu}}^{\mathord{\uparrow} \gamma} \circ t$ for all ordinals
  $\mu$, $\gamma < \omega^{\mu}$, and $s, t \in \mathbb{T}^{>, \succ}$ with $s
  < t$.
  
  \item[HF5] The map $\mathbb{R}^{>} \times \mathfrak{M}^{\succ} \!\!
  \rightarrow \mathfrak{M}; \left( r, \!\! \mathfrak{m} \right) \mapsto
  \mathfrak{m}^r \assign \hl_0^r \circ \mathfrak{m}$ extends to a real
  powering operation on~$\mathfrak{M}$.
  
  \item[HF6] $\hl_1 \circ (st) = \hl_1 \circ s + \hl_1 \circ t$ for all $s, t
  \in \mathbb{T}^{>, \succ}$.{\nopagebreak}
  
  \item[HF7] \label{HF7}$\tmop{supp} \hl_1 \circ \mathfrak{m} \succ 1$ for all
  $\mathfrak{m} \in \mathfrak{M}^{\succ}$; {\nopagebreak}
  
  {\noindent}$\tmop{supp} \hl_{\omega^{\mu}} \circ \mathfrak{a} \succ \left(
  \hl_{\gamma} \circ \mathfrak{a} \right)^{- 1}$ for all $1 \leqslant \mu
  <\tmmathbf{\nu}$, $\gamma < \omega^{\mu}$ and $\mathfrak{a} \in
  \mathfrak{M}_{\omega^{\mu}}$.
\end{descriptionaligned}

For each $\mu \in \mathbf{On}$, we define the function $L_{\omega^{\mu}} :
\mathfrak{M}_{\omega^{\mu}} \longrightarrow \mathbb{T}; \mathfrak{a}
\longmapsto \hl_{\omega^{\mu}} \circ \mathfrak{a}$. The
{\tmem{skeleton}}{\index{skeleton of a hyperserial field}} of~$(\mathbb{T},
\circ)$ is defined to be the structure $(\mathbb{T}, (L_{\omega^{\mu}})_{\mu
\in \mathbf{On}})$ equipped with the real power operation
from~{\tmstrong{HF5}}.

We say that $(\mathbb{T}, \circ)$ is {\tmem{confluent}}{\index{confluent
hyperserial field}} if for all $\mu \in \mathbf{On}$ with $\mu \leqslant
\tmmathbf{\nu}$, we have
\[ \forall s \in \mathbb{T}^{>, \succ}, \exists \mathfrak{a} \in
   \mathfrak{M}_{\omega^{\mu}}, \exists \gamma < \omega^{\mu}, \quad
   \hl_{\gamma} \circ s \asymp \hl_{\gamma} \circ \mathfrak{a}. \]
In particular $(\mathbb{L}, \circ)$ is a confluent hyperserial field.

\subsection{Hyperserial skeletons}

It turns out that each hyperlogarithm $L_{\omega^{\mu}}$ on a hyperserial
field $\mathbb{T}$ can uniquely be reconstructed from its restriction to the
subset of $L_{< \omega^{\mu}}$-atomic hyperseries (here we say that $f \in
\mathbb{T}^{>, \succ}$ is $L_{< \omega^{\mu}}$\mbox{-}atomic if $L_{\gamma} f
\in \mathfrak{M}$ for all $\gamma < \omega^{\mu}$). One of the main ideas
behind~{\cite{vdH:loghyp}} is to turn this fact into a way to
{\tmem{construct}} hyperserial fields. This leads to the definition of a
hyperserial skeleton as a~field $\mathbb{T}$ with partially defined
hyperlogarithms $L_{\omega^{\mu}}$, which satisfy suitable counterparts of the
above axioms {\tmstrong{\ref{HF1}}} until {\tmstrong{\ref{HF7}}}.

More precisely, let $\mathbb{T}=\mathbb{R} [[\mathfrak{M}]]$ be a field of
well-based series and fix $\tmmathbf{\nu} \in \mathbf{On}^{>} \cup \{
\mathbf{On} \}$. A{\tmem{~hyperserial skeleton}}{\index{hyperserial skeleton}}
on $\mathbb{T}$ of {\tmem{force}} $\tmmathbf{\nu}$ consists of a family of
partial functions $L_{\omega^{\mu}}$\label{autolab39} for $\mu
<\tmmathbf{\nu}$, called {\tmem{(hyper)logarithms}}, which satisfy a list of
axioms that we will describe now.

First of all, the domains $\mathfrak{M}_{\omega^{\mu}} \assign \tmop{dom}
L_{\omega^{\mu}}$\label{autolab40} on which the partial
functions~$L_{\omega^{\mu}}$ are defined should satisfy the following axioms:

\begin{tmframed}
  Domains of definition:
  \begin{description}
    \item[\tmem{$\mathbf{DD}_0$}] $\tmop{dom} L_1 =\mathfrak{M}^{\succ}$;
    
    \item[\tmem{$\mathbf{DD}_{\mu}$}] \label{DD-mu}$\tmop{dom}
    L_{\omega^{\mu}} = \bigcap_{\eta < \mu} \tmop{dom} L_{\omega^{\eta}}$, if
    $\mu$ is a non-zero limit ordinal;
    
    \item[\tmem{$\mathbf{DD}_{\mu}$}] $\tmop{dom} L_{\omega^{\mu}} = \left\{ s
    \in \mathbb{T}: L_{\omega^{\mu_-}}^{\circ n} (s) \in \tmop{dom} \;
    L_{\omega^{\mu_-}} \text{ for all $n$} \right\}$, if $\mu$ is a successor
    ordinal.
  \end{description}
\end{tmframed}

It will be convenient to also define the class
$\mathfrak{M}_{\omega^{\tmmathbf{\nu}}}$ by
\begin{align*}
     \mathfrak{M}_{\omega^{\tmmathbf{\nu}}} & \assign & \left\{ s \in
     \mathbb{T}: L_{\omega^{\tmmathbf{\nu}_-}}^{\circ n} (s) \in
     \mathfrak{M}_{\omega^{\tmmathbf{\nu}_-}} \text{ for all $n$} \right\} &&
     \text{if $\tmmathbf{\nu}$ is a successor ordinal}\\
     \mathfrak{M}_{\omega^{\tmmathbf{\nu}}} & \assign & \bigcap_{\mu
     <\tmmathbf{\nu}} \mathfrak{M}_{\omega^{\mu}} && \text{if
     $\tmmathbf{\nu}$ is a non-zero limit ordinal.}
   \end{align*}
Consider an ordinal $\gamma < \omega^{\tmmathbf{\nu}}$ written in Cantor
normal form $\gamma = \sum_{i = 1}^r \omega^{\eta_i} n_i$ where $\eta_1 >
\eta_2 > \cdots > \eta_r$ and $n_1, \ldots, n_r < \omega$. We let $L_{\gamma}$
denote the partial function
\begin{eqnarray}
  L_{\gamma} & = & L_{\omega^{\eta_1}}^{\circ n_1} \circ \cdots \circ
  L_{\omega^{\eta_r}}^{\circ n_r} .  \label{eq-general-hyperlog-Cantor}
\end{eqnarray}
It follows from the definition that for all $\tmmathbf{\mu} \leqslant
\tmmathbf{\nu}$, the class $\mathfrak{M}_{\omega^{\tmmathbf{\mu}}}$ consists
of those series $s \in \mathbb{T}^{>, \succ}$ for which $s \in \tmop{dom}
L_{\gamma}$ and $L_{\gamma} s \in \mathfrak{M}^{\succ}$ for all $\gamma <
\omega^{\tmmathbf{\mu}}$. We call such series $L_{<
\omega^{\tmmathbf{\mu}}}$-{\tmem{atomic}}{\index{$L_{< \beta}$-atomic
series}}.

Secondly, the hyperlogarithms $L_{\omega^{\mu}}$ with $\mu <\tmmathbf{\nu}$
should satisfy the following axioms:

\begin{tmframed}
  {\tmstrong{Axioms for the logarithm}}{\smallskip}
  
  {\noindent}Functional equation:
  
  \begin{descriptioncompact}
    \item[\tmem{$\mathbf{FE}_0$}] \label{functional-eq-0}$\forall
    \mathfrak{m}, \mathfrak{n} \in \mathfrak{M}_1, L_1
    (\mathfrak{m}\mathfrak{n}) = L_1 \mathfrak{m}+ L_1 \mathfrak{n}$.
  \end{descriptioncompact}
  
  Asymptotics:
  
  \begin{descriptioncompact}
    \item[$\mathbf{A}_0$] \label{asymptotics-0}$\forall r \in \mathbb{R}^{>},
    \forall \mathfrak{m} \in \mathfrak{M}_1, L_1 \mathfrak{m} \prec
    \mathfrak{m}$.
  \end{descriptioncompact}
  
  Monotonicity:
  
  \begin{descriptioncompact}
    \item[\tmem{$\mathbf{M}_0$}] \label{monotonicity-0}$\forall \mathfrak{m},
    \mathfrak{n} \in \mathfrak{M}_1, \mathfrak{m} \prec \mathfrak{n}
    \Longrightarrow L_1 \mathfrak{m}< L_1 \mathfrak{n}$.
  \end{descriptioncompact}
  
  Regularity:
  
  \begin{descriptioncompact}
    \item[$\mathbf{R}_0$] \label{regularity-0}$\forall \mathfrak{m} \in
    \mathfrak{M}_1, \tmop{supp} L_1 \mathfrak{m} \succ 1$.
  \end{descriptioncompact}
  
  Surjective logarithm:
  
  \begin{descriptioncompact}
    \item[$\mathbf{SL}$] \label{surjective-log}$\forall \varphi \in
    \mathbb{T}_{\succ}^{>}, \exists \mathfrak{m} \in \mathfrak{M}_1, \varphi =
    L_1 \mathfrak{m}$.
  \end{descriptioncompact}
\end{tmframed}

\begin{tmframed}
  {\tmstrong{Axioms for the hyperlogarithms}} (for each $\mu \in \mathbf{On}$
  with $0 < \mu <\tmmathbf{\nu}$ and $\beta \assign \omega^{\mu}$){\smallskip}
  
  {\noindent}Functional equation:\label{autolab41}
  
  \begin{descriptioncompact}
    \item[\tmem{$\mathbf{FE}_{\mu}$}] \label{functional-eq}$\forall
    \mathfrak{a} \in \mathfrak{M}_{\beta}, L_{\beta} L_{\beta_{/ \omega}}
    \mathfrak{a}= L_{\beta} \mathfrak{a}- 1$ if $\mu$ is a successor ordinal.
  \end{descriptioncompact}
  
  Asymptotics:\label{autolab42}
  
  \begin{descriptioncompact}
    \item[$\mathbf{A}_{\mu}$] \label{asymptotics}$\forall \gamma < \beta,
    \forall \mathfrak{a} \in \mathfrak{M}_{\beta}, L_{\beta} \mathfrak{a}<
    L_{\gamma} \mathfrak{a}$.
  \end{descriptioncompact}
  
  Monotonicity:\label{autolab43}
  
  \begin{descriptioncompact}
    \item[\tmem{$\mathbf{M}_{\mu}$}] \label{monotonicity}$\forall
    \mathfrak{a}, \mathfrak{b} \in \mathfrak{M}_{\beta}, \forall \gamma <
    \beta, \mathfrak{a} \prec \mathfrak{b} \Longrightarrow L_{\beta}
    \mathfrak{a}+ (L_{\gamma} \mathfrak{a})^{- 1} < L_{\beta} \mathfrak{b}-
    (L_{\gamma} \mathfrak{b})^{- 1}$.
  \end{descriptioncompact}
  
  Regularity:\label{autolab44}
  
  \begin{descriptioncompact}
    \item[$\mathbf{R}_{\mu}$] \label{regularity}$\forall \mathfrak{a} \in
    \mathfrak{M}_{\beta}, \forall \gamma < \beta, \tmop{supp} L_{\beta}
    \mathfrak{a} \succ (L_{\gamma} \mathfrak{a})^{- 1}$.
  \end{descriptioncompact}
\end{tmframed}

Finally, for $\mu \leqslant \tmmathbf{\nu}$ with $\mu \in \mathbf{On}$, we
also need the following axiom

\begin{tmframed}
  Infinite products:\label{autolab45}
  
  \begin{descriptioncompact}
    \item[\tmem{$\mathbf{P}_{\mu}$}] \label{infinite-products}$\forall
    \mathfrak{a} \in \mathfrak{M}_{\beta}, \forall \mathfrak{l} \in
    \mathfrak{L}_{< \beta}^{\succ}, \sum_{\gamma < \beta}
    \mathfrak{l}_{\gamma} L_{\gamma + 1} \mathfrak{a} \in L_1
    \mathfrak{M}^{\succ}$.
  \end{descriptioncompact}
\end{tmframed}

Note that \hyperref[surjective-log]{\textbf{SL}} and \hyperref[regularity-0]{$\mathbf{R_0}$} together imply $L_1
\mathfrak{M}^{\succ} =\mathbb{T}_{\succ}^{>}$, whence $\mathbf{P}_{\mu}$
automatically holds. This will in particular be the case for $\mathbf{No}$
(see Section~\ref{section-transserial-structure}).

In summary, we have:

\begin{definition}
  \emph{\cite[Definition~3.3]{BvdHK:hyp}} Given $\tmmathbf{\nu} \in
  \mathbf{On}^{>} \cup \{ \mathbf{On} \}$, we say that the structure $(\mathbb{T},
  (L_{\omega^{\mu}})_{\mu <\tmmathbf{\nu}})$ is a
  {\textit{{\tmstrong{hyperserial skeleton}}{\index{hyperserial field
  skeleton}}}} of force $\tmmathbf{\nu}$ if it satisfies \hyperref[DD-mu]{$\mathbf{DD_{\mu}}$},
  \hyperref[functional-eq]{$\mathbf{FE_{\mu}}$}, \hyperref[asymptotics]{$\mathbf{A_{\mu}}$}, \hyperref[monotonicity]{$\mathbf{M_{\mu}}$}, and
  \hyperref[regularity]{$\mathbf{R_{\mu}}$} for all~$\mu <\tmmathbf{\nu}$, as well as
  \hyperref[infinite-products]{$\mathbf{P_{\mu}}$} for all ordinals~{$\mu \leqslant \tmmathbf{\nu}$}.
\end{definition}

Assume that $\mathbb{T}$ is a hyperserial skeleton of force $\tmmathbf{\nu}$.
The partial logarithm $L_1 : \mathfrak{M}_1 \longrightarrow \mathbb{T}$
extends naturally into a strictly increasing morphism $(\mathbb{T}^{>},
\times, <) \longrightarrow (\mathbb{T}, +, <)$, which we call the
{{\tmem{logarithm}}} and denote by $L_1$ or~$\log$
{\cite[Section~4.1]{BvdHK:hyp}}. If $\mathbb{T}$ satisfies
\hyperref[surjective-log]{$\mathbf{SL}$}, then this extended logarithm is actually an isomorphism
{\cite[Proposition~2.3.8]{Schm01}}. In that case, for any $s \in
\mathbb{T}^{>}$ and $r \in \mathbb{R}$, we define {$s^r \assign \exp (r \log
s) \in \mathbb{T}^{>}$}.

\subsection{Confluence}

\begin{definition}
  \label{def-confluence}{\tmem{{\cite[Definition~3.5]{BvdHK:hyp}}}} Given a
  hyperserial skeleton $\mathbb{T}=\mathbb{R} [[\mathfrak{M}]]$ of force $\nu
  \in \mathbf{On}^{>}$ and $\mu < \nu$, we inductively define the notion of
  $\mu$-{\tmstrong{{\tmem{confluence}}}}{\index{$\mu$-confluence}} in
  conjunction with the definition of functions $\mathfrak{d}_{\omega^{\mu}} :
  \mathbb{T}^{>, \succ} \longrightarrow \mathfrak{M}_{\omega^{\mu}}$, as
  follows.
  \begin{itemizedot}
    \item The field $\mathbb{T}$ is said $0$-confluent if $\mathfrak{M}$ is
    non-trivial. The function $\mathfrak{d}_1$ maps every positive infinite
    series $s \in \mathbb{T}^{>, \succ}$ onto its dominant monomial
    $\mathfrak{d}_s$. For each $s \in \mathbb{T}^{>, \succ}$, we write
    \begin{eqnarray*}
      \mathcal{E}_1 [s] & \assign & \{ t \in \mathbb{T}^{>, \succ} \suchthat t
      \asymp s \} .
    \end{eqnarray*}
  \end{itemizedot}
  Let $\mu \leqslant \tmmathbf{\nu}$ be such that $\mathbb{T}$ is
  $\eta$-confluent for all $\eta < \mu$ and let $s \in \mathbb{T}^{>, \succ}$.
  \begin{itemizedot}
    \item If $\mu$ is a successor ordinal, then we write
    $\mathcal{E}_{\omega^{\mu}} [s]$ for the class of series $t$ with
    \begin{eqnarray*}
      \left( L_{\omega^{\mu_-}} \circ \mathfrak{d}_{\omega^{\mu_-}}
      \right)^{\circ n} (s) & \asymp & \left( L_{\omega^{\mu_-}} \circ
      \mathfrak{d}_{\omega^{\mu_-}} \right)^{\circ n} (t)
    \end{eqnarray*}
    for a certain $n \in \mathbb{N}$.
    
    \item If $\mu$ is a limit ordinal, then we write
    $\mathcal{E}_{\omega^{\mu}} [s]$ for the class of series $t$ with
    \begin{eqnarray*}
      L_{\omega^{\eta}} \mathfrak{d}_{\omega^{\eta}} (s) & \asymp &
      L_{\omega^{\eta}} \mathfrak{d}_{\omega^{\eta}} (t)
    \end{eqnarray*}
    for a certain $\eta < \mu$.
  \end{itemizedot}
  We say that $\mathbb{T}$ is {\tmem{{\tmstrong{$\mu$-confluent}}}} if each
  class $\mathcal{E}_{\omega^{\mu}} [s]$\label{autolab46} contains a $L_{<
  \omega^{\mu}}$-atomic element; we then define~$\mathfrak{d}_{\omega^{\mu}}
  (s)$\label{autolab47} to be this element.
\end{definition}

This inductive definition is sound. Indeed, if $\mu \leqslant \nu + 1$ and
$\mathbb{T}$ is $\eta$-confluent for all $\eta < \mu$, then the functions
$\mathfrak{d}_{\omega^{\eta}} : \mathbb{T}^{>, \succ} \longrightarrow
\mathfrak{M}_{\omega^{\eta}}$ with $\eta < \mu$ are well-defined and
non-decreasing. Thus, for~$\eta < \mu$, the collection of
$\mathcal{E}_{\omega^{\eta}} [s]$ with $s \in \mathbb{T}^{>, \succ}$ forms a
partition of $\mathbb{T}^{>, \succ}$ into convex subclasses.

We say that $\mathbb{T}$ is {\tmem{confluent}} if it is $\nu$-confluent. If
$\mathbb{T}$ has force $\mathbf{On}$, then we say that $\mathbb{T}$ is
$\mathbf{On}$\mbox{-}confluent, or {\tmem{confluent}}, if $(\mathbb{T},
(L_{\omega^{\eta}})_{\eta < \mu})$ is $\mu$-confluent for all $\mu \in
\mathbf{On}$.

\subsection{Correspondence between fields and
skeletons}\label{subsubsection-field-skeleton}

\begin{proposition}
  \label{th-skeleton-field}{\tmem{{\cite[Theorem~1.1]{BvdHK:hyp}}}} If
  $(\mathbb{T}, (L_{\omega^{\mu}})_{\mu \in \mathbf{On}})$ is a confluent
  hyperserial skeleton, then there is a unique function $\circ : \mathbb{L}
  \times \mathbb{T}^{>, \succ} \longrightarrow \mathbb{T}$ with
  \[ \forall \mu \in \mathbf{On}, \forall \mathfrak{a} \in
     \mathfrak{M}_{\omega^{\mu}}, \quad \hl_{\omega^{\mu}} \circ \mathfrak{a}=
     L_{\omega^{\mu}} \mathfrak{a} \]
  such that $(\mathbb{T}, \circ)$ is a confluent hyperserial field.
\end{proposition}

Assume now that $\mathbb{T}$ is only a hyperserial skeleton of force
$\tmmathbf{\nu} \in \mathbf{On}^{>} \cup \{ \mathbf{On} \}$ and that $\mu$ is
an ordinal with $0 < \mu <\tmmathbf{\nu}$ such that $(\mathbb{T},
(L_{\omega^{\eta}})_{\eta < \mu})$ is $\mu$-confluent. Let $\beta \assign
\omega^{\mu}$. By {\cite[Definition~4.11 and Lemma~4.12]{BvdHK:hyp}}, the
partial function $L_{\beta}$ naturally extends into a function $\mathbb{T}^{>,
\succ} \longrightarrow \mathbb{T}^{>, \succ}$ that we still denote by
$L_{\beta}$. This extended function is strictly increasing, by`
{\cite[Corollary~4.17]{BvdHK:hyp}}. If $\mu$ is a~successor ordinal, then it
satisfies the functional equation
\begin{equation}
  \forall s \in \mathbb{T}^{>, \succ}, \quad L_{\beta} L_{\beta_{/ \omega}} s
  = L_{\beta} s - 1 \label{eq-functional-total},
\end{equation}
by {\cite[Proposition~4.13]{BvdHK:hyp}}. For $\gamma < \beta$, we have a
strictly increasing function $L_{\gamma} : \mathbb{T}^{>, \succ}
\longrightarrow \mathbb{T}^{>, \succ}$ obtained as a composition of functions
$L_{\omega^{\eta}}$ with $\eta < \mu$, as in
(\ref{eq-general-hyperlog-Cantor}). By {\cite[Proposition~4.7]{BvdHK:hyp}}, we
have
\begin{eqnarray*}
  \mathcal{E}_{\beta} [s] & = & \{ t \in \mathbb{T}^{>, \succ} \suchthat
  \exists \gamma < \beta, L_{\gamma} t \asymp L_{\gamma} s \} .
\end{eqnarray*}
\subsection{Hyperexponentiation}\label{subsubsection-hyperexponentiation}

In a traditional transseries field $\mathbb{T}$, the transmonomials are
characterized by the fact that, for any $f \in \mathbb{T}^{>}$, we have
\begin{eqnarray}
  f \in \mathfrak{M} & \Longleftrightarrow & \tmop{supp} \log f \succ 1. 
  \label{mon-charac}
\end{eqnarray}
In particular, the logarithm $\log : \mathbb{T}^{>} \longrightarrow
\mathbb{T}$ is surjective as soon as $\exp \varphi$ is defined for all
{$\varphi \in \mathbb{T}$} with $\tmop{supp} \varphi \succ 1$. In hyperserial
fields, similar properties hold for $L_{< \omega^{\eta}}$-atomic elements with
respect to the hyperexponential $E_{\omega^{\eta}}$, as we will recall now.

Given $\tmmathbf{\nu} \in \mathbf{On}^{>} \cup \{ \mathbf{On} \}$, let
$\mathbb{T}$ be a confluent hyperserial skeleton $\mathbb{T}$ of force
$\tmmathbf{\nu}$. By {\cite[Theorem~4.1]{BvdHK:hyp}}, we have a composition
$\circ : \mathbb{L}_{< \omega^{\tmmathbf{\nu}}} \times \mathbb{T}^{>, \succ}
\longrightarrow \mathbb{T}$. Given $\eta <\tmmathbf{\nu}$, the extended
function {$L_{\omega^{\eta}} : \mathbb{T}^{>, \succ} \longrightarrow
\mathbb{T}^{>, \succ}$} is strictly increasing and hence injective.
Consequently, $L_{\omega^{\eta}}$ has a partially defined functional inverse
that we denote by $E_{\omega^{\eta}}$\label{autolab48}.

The characterization~(\ref{mon-charac}) generalizes as follows:

\begin{definition}
  \emph{{\cite[Definition~7.10]{BvdHK:hyp}}} We say that $\varphi \in
  \mathbb{T}^{>, \succ}$ is {\tmem{$\textbf{1}${\tmstrong{-truncated}}}} if
  \begin{eqnarray*}
    \tmop{supp} \varphi & \succ & 1.
  \end{eqnarray*}
  Given $0 < \eta < \nu$, we say that a series $\varphi \in \mathbb{T}^{>,
  \succ}$ is
  {\tmstrong{{\tmem{$\omega^{\eta}$-truncated}}}}{\index{$\beta$-truncated
  series}} if
  \[ \forall \mathfrak{m} \in \tmop{supp} \varphi, \quad \mathfrak{m} \prec 1
     \Longrightarrow \left( \forall \gamma < \omega^{\eta}, \varphi <
     \hl_{\omega^{\eta}}^{\mathord{\uparrow} \gamma} \circ \mathfrak{m}^{- 1}
     \right) . \]
  For any $\beta = \omega^{\eta} < \omega^{\tmmathbf{\nu}}$, we write
  $\mathbb{T}_{\succ, \beta}$\label{autolab49} for the class of
  $\beta$-truncated series in $\mathbb{T}$.
\end{definition}

\begin{proposition}
  {\tmem{{\cite[Corollary~7.21]{BvdHK:hyp}}}} For $f \in \mathbb{T}^{>,
  \succ}$ and $\beta = \omega^{\eta} < \omega^{\tmmathbf{\nu}}$, we have
  \begin{eqnarray*}
    f \in \mathfrak{M}_{\beta} & \Longleftrightarrow & L_{\beta} f \in
    \mathbb{T}_{\succ, \beta} .
  \end{eqnarray*}
\end{proposition}

In general, we have~$\mathbb{T}_{\succ, \beta} +\mathbb{R}^{\geqslant}
\subseteq \mathbb{T}_{\succ, \beta}$. Whenever $\eta$ is a successor ordinal,
we even have
\begin{eqnarray}
  \mathbb{T}_{\succ, \beta} +\mathbb{R} & = & \mathbb{T}_{\succ, \beta} 
  \label{eq-succ-truncated}
\end{eqnarray}
Let $\varphi$ be a series such that $E_{\beta} \varphi$ is defined. By
{\cite[Lemma~7.14]{BvdHK:hyp}}, the series $\varphi$ is $\beta$-truncated if
and only if
\[ \forall \gamma < \beta, \quad \tmop{supp} \varphi \succ (L_{\gamma}
   E_{\beta} \varphi)^{- 1} . \]
For $\mu <\tmmathbf{\nu}$, the axiom \hyperref[regularity]{$\mathbf{R_0}$} is therefore equivalent
to the inclusion $L_{\omega^{\mu}} \mathfrak{M}_{\omega^{\mu}} \subseteq
\mathbb{T}_{\succ, \omega^{\mu}}$. For $s \in \mathbb{T}^{>, \succ}$, there is
a unique $\trianglelefteqslant$-maximal truncation $\sharp_{\beta} (s)$ of $s$
which is $\beta$-truncated. By {\cite[Propositions~6.16 and~6.17]{BvdHK:hyp}},
the classes\label{autolab50}
\begin{eqnarray}
  \mathcal{L}_{\beta} [s] & \assign & \left\{ t \in s +\mathbb{T}^{\prec}
  \suchthat t = s \text{, or $\exists \gamma < \beta, \left( t <
  \hl_{\beta}^{\mathord{\uparrow} \gamma} \circ | s - t |^{- 1} \right)$}
  \right\}  \label{eq-L-class}
\end{eqnarray}
with$\hspace{1.4em} s \in \mathbb{T}^{>, \succ}$ form a partition of
$\mathbb{T}^{>, \succ}$ into convex subclasses. Moreover, the series
$\sharp_{\beta} (s)$\label{autolab51} is both the unique $\beta$-truncated
element and the $\trianglelefteqslant$-minimum of $\mathcal{L}_{\beta} [s]$.
If $E_{\beta} s$ is defined, then we have the following simplified definition
{\cite[Proposition~7.19]{BvdHK:hyp}} of the class $\mathcal{L}_{\beta} [s]$:
\begin{eqnarray}
  \mathcal{L}_{\beta} [s] & \assign & \left\{ t \in \mathbb{T}^{>, \succ}
  \suchthat \exists \gamma < \beta, t - s \prec \frac{1}{L_{\gamma} E_{\beta}
  s} \right\} .  \label{L-when-E}
\end{eqnarray}
The following shows that the existence of $E_{\beta}$ on $\mathbb{T}^{>,
\succ}$ is essentially equivalent to its existence on $\mathbb{T}_{\succ,
\beta}$.

\begin{proposition}
  \label{th-bijective-hyperlogarithm}{\tmem{{\cite[Corollary~7.24]{BvdHK:hyp}}}}
  Let $\tmmathbf{\mu} \leqslant \tmmathbf{\nu}$ and assume that for $\eta
  <\tmmathbf{\mu}$, the function $E_{\omega^{\eta}}$ is defined
  on~$\mathbb{T}_{\succ, \omega^{\eta}}$. Then each hyperlogarithm
  $L_{\omega^{\eta}}$ for $\eta <\tmmathbf{\mu}$ is bijective.
\end{proposition}

If Proposition~\ref{th-bijective-hyperlogarithm} holds, then we say that
$\mathbb{T}$ is a (confluent) {\tmem{hyperserial field of force}}
$(\tmmathbf{\nu}, \tmmathbf{\mu})${\index{hyperserial field of force
$(\tmmathbf{\nu}, \tmmathbf{\mu})$}}. Since every function $L_{\gamma}, \gamma
< \omega^{\tmmathbf{\mu}}$ is then a strictly increasing bijection
$\mathbb{T}^{>, \succ} \longrightarrow \mathbb{T}^{>, \succ}$, we obtain
\begin{eqnarray}
  \mathcal{E}_{\lambda} [s] & = & \{ t \in \mathbb{T}^{>, \succ} \suchthat
  \exists \gamma < \lambda, \exists r_0, r_1 \in \mathbb{R}^{>}, E_{\gamma}
  (r_0 L_{\gamma} s) < t < E_{\gamma} (r_1 L_{\gamma} s) \}, 
  \label{exp-class-char}
\end{eqnarray}
for each ordinal $\lambda = \omega^{\iota}$ with $\iota \leqslant
\tmmathbf{\mu}$. By {\cite[Corollary~7.23]{BvdHK:hyp}}, for all $s \in
\mathbb{T}^{>, \succ}$, we have
\begin{eqnarray}
  E_{\beta} (\sharp_{\beta} (s)) & = & \mathfrak{d}_{\beta} (E_{\beta} s) . 
  \label{eq-hyperexp-proj}
\end{eqnarray}
\section{The transseries field
$\mathbf{No}$}\label{section-transserial-structure}

Recall that $\mathbf{No}$ is identified with the ordered field of well-based
series $\mathbb{R} [[\mathbf{Mo}]]$. In this section, we describe, in the
first level $\nu = 1$ of our hierarchy, the properties of $\mathbf{No}$
equipped with the Kruskal-Gonshor logarithm.

\subsection{Surreal exponentiation}

In {\cite[Chapter~10]{Gon86}}, Gonshor defines the exponential function $\exp
: \mathbf{No} \longrightarrow \mathbf{No}^{>}$, relying on partial Taylor sums
of the real exponential function. For $a \in \mathbf{No}$ and $n \in
\mathbb{N}$, write
\begin{eqnarray*}
  {}[a]_n & \assign & \sum_{k \leqslant n} \frac{a^k}{k!} .
\end{eqnarray*}
We then have the recursive definition for all numbers $a \in \mathbf{No}$:
\[\exp a \assign \left\{ \exp (a_L)  [a -
   a_L]_{\mathbb{N}}, \exp (a_R)  [a - a_R]_{2\mathbb{N}+ 1} | \frac{\exp
   a_R}{[a_R - a]_{\mathbb{N}}}, \frac{\exp a_L}{[a_L - a]_{2\mathbb{N}+ 1}}
   \right\} . \]
We will sometimes write $\mathe^a$ instead of $\exp a$. The function $\exp :
(\mathbf{No}, +, <) \longrightarrow (\mathbf{No}^{>}, \times, <)$ is an isomorphism {\cite[Corollary~10.1, Corollary~10.3]{Gon86}} which
satisfies:
\begin{itemizedot}
  \item $\exp$ coincides with the natural exponential on $\mathbb{R} \subseteq
  \mathbf{No}$ {\cite[Theorem~10.2]{Gon86}}.
  
  \item $\mathe^{\mathbf{No}_{\succ}} = \mathbf{Mo}$
  {\cite[Theorems~10.7,~10.8 and~10.9]{Gon86}}.
\end{itemizedot}
We define $\log : \mathbf{No}^{>} \longrightarrow \mathbf{No}$ to be the
functional inverse of $\exp$, and we set $L_1 \assign \log
\mathrel{\upharpoonleft} \mathbf{Mo}^{\succ}$. Given an ordinal $\alpha$, we
understand that $\omega^{\alpha}$ still stands for the $\alpha$-th ordinal
power of $\omega$ from section~\ref{ordinal-sec} and warn the reader that
$\omega^{\alpha}$ does not necessarily coincide with $\mathe^{\alpha \log
\omega}$.

Together, the above facts imply that $L_1$ satisfies the axioms
\hyperref[functional-eq-0]{$\mathbf{FE_0}$}, \hyperref[asymptotics-0]{$\mathbf{A_0}$}, \hyperref[monotonicity-0]{$\mathbf{M_0}$},
\hyperref[regularity-0]{$\mathbf{R_0}$} and \hyperref[surjective-log]{$\mathbf{SL}$}. Therefore, $(\mathbf{No}, L_1)$
is a hyperserial skeleton of force $1$. The extension of $L_1$ to
$\mathbf{No}^{>}$ from section~\ref{subsubsection-field-skeleton} coincides
with~$\log$. It was shown in {\cite{EvdD01}} that $(\mathbf{No}, +, \times, <,
\exp)$ is an elementary extension of~{$(\mathbb{R}, +, \times, <, \exp)$}. See
{\cite{MM17,Ber20,BKMM}} for more details on $\exp$ and $\log$.

\

\subsection{$\mathbf{No}$ as a transseries field}

Berarducci and Mantova identified the class $\mathbf{Mo}_{\omega}$ of
$\log$-atomic numbers as $\mathbf{Mo}_{\omega} = \mathbf{Smp}_{\mathcal{E}}$
{\cite[Corollary~5.17]{BM18}} and showed that $(\mathbf{No}, L_1)$ is
$1$-confluent {\cite[Corollary~5.11]{BM18}}. Thus $(\mathbf{No}, L_1)$ is a
confluent hyperserial skeleton of force $(1, 1)$. Thanks to
{\cite[Theorem~1.1]{BvdHK:hyp}}, it is therefore equipped with a composition
law $\mathbb{L}_{< \omega} \times \mathbf{No}^{>, \succ} \longrightarrow
\mathbf{No}$. See {\cite{Schm01,BM19}} for further details on extensions of
this composition law to exponential extensions of $\mathbb{L}_{< \omega}$.

Berarducci and Mantova also proved {\cite[Theorem 8.10]{BM18}} that
$\mathbf{No}$ is a field of transseries in the sense of
{\cite{vdH:phd,Schm01}}. In other words $(\mathbf{No}, L_1)$ satisfies the axiom
$\textbf{T4}$\label{autolab52} of {\cite[Definition~2.2.1]{Schm01}}. We plan
to prove in subsequent work that $(\mathbf{No}, (L_{\omega^{\mu}})_{\mu \in
\mathbf{On}})$ satisfies a generalized version of $\textbf{T4}$.

\section{Hyperserial structure on
$\mathbf{No}$}\label{section-hyperserial-structure}

We have seen in section~\ref{section-transserial-structure} that
$(\mathbf{No}, L_1)$ is a confluent hyperserial skeleton of force $(\nu, \nu)$
for~{$\nu = 1$}. The aim of this section is to extend this result to any
ordinal $\nu$. More precisely, we will define a sequence
$(L_{\omega^{\mu}})_{\mu \in \mathbf{On}}$ of partial functions on
$\mathbf{No}$ such that for each ordinal $\nu$, the structure~{$(\mathbf{No},
(L_{\omega^{\mu}})_{\mu < \nu})$} is a confluent hyperserial skeleton of force
$(\nu, \nu)$, and $L_1$ coincides with Gonshor's logarithm.

\subsection{Remarkable group actions on
$\mathbf{No}$}\label{rem-subclasses}\label{subsection-remarkable-group-actions}

Assume for the moment that we can define $L_{\gamma}$ and $E_{\gamma}$ as
bijective strictly increasing functions on~$\mathbf{No}^{>, \succ}$ for all
ordinals $\gamma$. This is the case already for $\gamma < \omega$. Let us
introduce several useful groups that act on $\mathbf{No}$, as well as several
remarkable subclasses of $\mathbf{No}$.

Given an ordinal $\nu$, we write $\alpha = \omega^{\nu}$ and we consider the
function groups\label{autolab53} \label{autolab54}
\begin{eqnarray*}
  \mathcal{E}_{\alpha}' & = & \langle E_{\gamma} H_r L_{\gamma} \suchthat
  \gamma < \alpha \wedge r \in \mathbb{R}^{>} \rangle\\
  \mathcal{E}_{\alpha}^{\ast} & = & \langle E_{\gamma}, P_r \suchthat \gamma <
  \alpha \wedge r \in \mathbb{R}^{>} \rangle .
\end{eqnarray*}
where $E_{\gamma}$, $H_s, P_s$ and $L_{\gamma}$ act on $\mathbf{No}^{>,
\succ}$. We also define\label{autolab55} \label{autolab56}
\begin{eqnarray*}
  \mathcal{L}_{\alpha}' & = & L_{\alpha} \mathcal{E}_{\alpha}' E_{\alpha}\\
  \mathcal{L}_{\alpha}^{\ast} & = & L_{\alpha} \mathcal{E}_{\alpha}^{\ast}
  E_{\alpha} .
\end{eqnarray*}
We write $L_{< \lambda} \assign \{ L_{\gamma} \suchthat \gamma < \lambda \}$
and $E_{< \lambda} \assign \{ E_{\gamma} : \gamma < \lambda \}$ for each
$\lambda \leqslant \alpha$. In the case when $\alpha = 1$, note that
\begin{eqnarray*}
  \mathcal{E}_1' & = & \mathcal{H}\\
  \mathcal{E}_1^{\ast} & = & \mathcal{P}\\
  \mathcal{L}_1' & = & \mathcal{T}\\
  \mathcal{L}_1^{\ast} & = & \mathcal{H}.
\end{eqnarray*}
By Proposition~\ref{prop-thin-substructure} and the fact the set-sized
function groups $\mathcal{E}_{\alpha}'$, $\mathcal{E}_{\alpha}^{\ast}$,
$\mathcal{L}_{\alpha}'$, and $\mathcal{L}_{\alpha}^{\ast}$ induce thin
partitions of $\mathbf{No}^{>, \succ}$, we may define the following surreal
substructures\label{autolab57} \label{autolab58} \label{autolab59}
\label{autolab60}
\begin{eqnarray*}
  \mathbf{Mo}_{\alpha}' & \assign & \mathbf{Smp}_{\mathcal{E}_{\alpha}'}\\
  \mathbf{Mo}_{\alpha}^{\ast} & \assign &
  \mathbf{Smp}_{\mathcal{E}_{\alpha}^{\ast}}\\
  \mathbf{Tr}_{\alpha} & \assign & \mathbf{Smp}_{\mathcal{L}_{\alpha}'}\\
  \mathbf{Tr}_{\alpha}^{\ast} & \assign &
  \mathbf{Smp}_{\mathcal{L}_{\alpha}^{\ast}} .
\end{eqnarray*}
Here we note that $\mathbf{Mo}_1'$ corresponds to the class
$\mathbf{Mo}^{\succ} = \mathbf{Mo}_1$ of infinite monomials in $\mathbf{No}$
and~$\pi_{\mathcal{E}_1}'$ maps positive infinite numbers to their dominant
monomial. Similarly, $\mathbf{Tr}_1$ coincides with~$\textbf{}
\mathbf{No}_{\succ}^{>}$ and $\pi_{\mathcal{L}_1'}$ maps $a \in
\mathbf{No}^{>, \succ}$ to $a_{\succ}$. In
sections~\ref{section-hyperserial-structure}
and~\ref{section-useful-identities}, we will prove the following identities, where $r$ ranges among real numbers.
\begin{align*}
  \mathbf{Mo}_{\alpha}' & =  \mathbf{Mo}_{\alpha} &&
  \text{[Proposition~\ref{prop-I3}]}\\
  \pi_{\mathcal{E}_{\alpha}'} & =  \mathfrak{d}_{\alpha} &&
  \text{[Proposition~\ref{prop-I3}]}\\
  \mathbf{Tr}_{\alpha} & = \mathbf{No}_{\succ, \alpha} = L_{\alpha} 
  \mathbf{Mo}_{\alpha} &&
  \text{[Proposition~\ref{identification-prop}]}\\
  \pi_{\mathcal{L}_{\alpha}'} & = \sharp_{\alpha} &&
  \text{[Proposition~\ref{identification-prop}]}\\
  \mathbf{Tr}_{\alpha}^{\ast} & = \mathbf{Tr}_{\alpha} \quad
  \text{if $\nu$ is a limit
  ordinal} && \text{[Lemma~\ref{lem-limit-hypermon}]}\\
  \mathbf{Tr}_{\alpha}^{\ast} & = \mathbf{No}_{\succ}^{>} \quad
  \text{if $\nu$ is a successor
  ordinal} && \text{[Lemma~\ref{lem-Tr'-successor-identity}]}\\
  \Xi_{\mathbf{No}_{\succ, \alpha}} T_r & = T_r
  \Xi_{\mathbf{No}_{\succ, \alpha}} \quad \text{if $\nu$ is a successor
  ordinal} && \text{[Lemma~\ref{lem-Tr'-translation-invariant}]}\\
  \Xi_{\mathbf{Mo}_{\alpha}} T_r & = E_{\alpha}
  T_r L_{\alpha} \Xi_{\mathbf{Mo}_{\alpha}} \quad \text{if $\nu$ is a
  successor ordinal} && \text{[Proposition~\ref{prop-hyperexp-hypermon-identity}]}\\
  \mathbf{Mo}_{\alpha}^{\ast} & = \mathbf{Mo}_{\alpha} \mathbin{\Yleft}
  \mathbf{No}_{\succ} \quad \text{if $\nu$ is a successor ordinal} && \text{[Proposition~\ref{prop-hyperexp-J-right-factor}]}\\
  \mathbf{Mo}_{\alpha}^{\ast} & = E_{\alpha}  \mathbf{Tr}_{\alpha}^{\ast} .
  && \text{[Proposition~\ref{prop-hyperexp-star-relation}]}
\end{align*}
The first and third identities imply in particular that the classes
$\mathbf{Mo}_{\alpha}'$ and $\mathbf{No}_{\succ, \alpha}$, as defined in
section~\ref{subsection-hyperserial-fields} when seeing $\mathbf{No}$ as a
hyperserial field, are in fact surreal substructures.

\subsection{Inductive setting}\label{subsection-inductive-setting}

For the definition of the partial hyperlogarithm $L_{\omega^{\mu}}$, we will
proceed by induction on $\mu$. Let $\mu$ be an ordinal. Until the end of this
section we make the following induction hypotheses:

\begin{tmframed}
  {\tmstrong{\begin{center}
    Induction hypotheses
  \end{center}}}
  
   \begin{descriptionaligned}
    \item[\tmem{$\mathbf{I}_{1, \mu}$}] \label{ind-ax1}For $\eta < \mu$, the
    partial hyperlogarithm $L_{\omega^{\eta}}$ is defined; we have $L_1 = \log
    \upharpoonleft \mathbf{Mo}^{\succ}$ and~{$(\mathbf{No},
    (L_{\omega^{\eta}})_{\eta < \mu})$} is a confluent hyperserial skeleton of
    force $(\mu, \mu)$.
    
    \item[\tmem{$\mathbf{I}_{2, \mu}$}] \label{ind-ax2}For $r, s \in
    \mathbb{R}$ with $1 < s$ and for $\gamma, \rho < \omega^{\mu}$ with
    $\gamma < \rho$, we have
    \[ \forall a \in \mathbf{No}^{>, \succ}, \quad E_{\gamma}  (rL_{\gamma} a)
       < E_{\rho} (sL_{\rho} a) . \]
    \item[\tmem{$\mathbf{I}_{3, \mu}$}] \label{ind-ax3}For $\eta \leqslant
    \mu$, the class $\mathbf{Mo}_{\omega^{\eta}}'$ is that of $L_{<
    \omega^{\eta}}$-atomic surreal numbers, i.e. $\mathbf{Mo}_{\omega^{\eta}}'
    = \mathbf{Mo}_{\omega^{\eta}}$.
  \end{descriptionaligned}
\end{tmframed}

These induction hypotheses require a few additional explanations. Assuming
that \hyperref[ind-ax1]{$\mathbf{I_{1,\mu}}$} holds, the partial functions $L_{\omega^{\eta}}$ with $\eta
< \mu$ extend into strictly increasing bijections $L_{\omega^{\eta}} :
\mathbf{No}^{>, \succ} \longrightarrow \mathbf{No}^{>, \succ}$, by the results
from section~\ref{subsection-hyperserial-fields}.
Using~(\ref{eq-Cantor-form-composition}), this allows us to define a strictly
increasing bijection $L_{\gamma} : \mathbf{No}^{>, \succ} \longrightarrow
\mathbf{No}^{>, \succ}$ for any $\gamma < \mu$ and we denote by $E_{\gamma}$
its functional inverse. In particular, this ensures that the hypotheses
\hyperref[ind-ax2]{$\mathbf{I_{2,\mu}}$} and~\hyperref[ind-ax3]{$\mathbf{I_{3,\mu}}$} make sense.

\begin{remark}
  In addition to the above induction hypotheses, we will implicitly assume
  that our hyperlogarithms $L_{\omega^{\eta}}$ for $\eta < \mu$ are always
  defined by (\ref{eq-poor-hyperlog}) below. In particular, our construction
  of $L_{\omega^{\mu}}$ is {\tmem{not}} relative to any potential construction
  of the preceding hyperlogarithms~$L_{\omega^{\eta}}$ with~$\eta < \mu$ that
  would satify the induction hypotheses \hyperref[ind-ax1]{$\mathbf{I_{1,\mu}}$}, \hyperref[ind-ax2]{$\mathbf{I_{2,\mu}}$},
  and~\hyperref[ind-ax3]{$\mathbf{I_{3,\mu}}$}. Instead, we define {\tmem{one}} specific family of
  functions $(L_{\omega^{\eta}})_{\eta \in \mathbf{On}}$ that satisfy our
  requirements, as well as the additional identities listed in
  subsection~\ref{subsection-remarkable-group-actions}.
\end{remark}

\begin{proposition}
  \label{prop-initial-case}The axioms  $\mathbf{I}_{1, 1}$, $\mathbf{I}_{2,
  1}$ and $\mathbf{I}_{3, 1}$ hold for $(\mathbf{No}, L_1)$.
\end{proposition}

\begin{proof}
  Section~\ref{section-transserial-structure} shows that $\mathbf{I}_{1, 1}$
  holds. Consider $r, s \in \mathbb{R}^{>}$ with $s > 1$. On $\mathbf{No}^{>,
  \succ}$, we have $T_{\log r} < H_s$, hence $H_r = E_1 T_{\log r} L_1 < E_1
  H_s L_1$. It follows that we have $E_n H_r L_n < E_{n + 1} H_s L_{n + 1}$ on
  $\mathbf{No}^{>, \succ}$ for all $n \in \mathbb{N}$. This implies that
  $\mathbf{I}_{2, 1}$ holds. Finally, $\mathbf{I}_{3, 1}$ is valid because of
  the relation~$\mathbf{Mo}_{\omega} = \mathbf{Smp}_{\mathcal{E}}$.
\end{proof}

\begin{proposition}
  \label{prop-limit-case}Let $\nu$ be a limit ordinal and assume that
  \hyperref[ind-ax1]{$\mathbf{I_{1,\mu}}$}, \hyperref[ind-ax2]{$\mathbf{I_{2,\mu}}$}, and~\hyperref[ind-ax3]{$\mathbf{I_{3,\mu}}$} hold for all $\mu < \nu$.
  Then $\mathbf{I}_{1, \nu}$, $\mathbf{I}_{2, \nu}$, and $\mathbf{I}_{3, \nu}$
  hold.
\end{proposition}

\begin{proof}
  The statement $\mathbf{I}_{2, \nu}$ follows immediately by induction.
  Towards $\mathbf{I}_{3, \nu}$, note that we have $\mathbf{Mo}_{\alpha} =
  \bigcap_{\eta < \nu} \mathbf{Mo}_{\omega^{\eta}} = \bigcap_{\eta < \nu}
  \mathbf{Mo}_{\omega^{\eta}}'$ by $\mathbf{I}_{1, \eta}$ (and thus
  $\mathbf{DD}_{\eta}$) and $\mathbf{I}_{3, \eta}$ for all $\eta < \nu$. By
  {\cite[Proposition~6.28]{BvdH19}}, we have $\mathbf{Mo}_{\alpha}' =
  \bigcap_{\eta < \nu} \mathbf{Mo}_{\omega^{\eta}}' = \mathbf{Mo}_{\alpha}$.
  So $\mathbf{I}_{3, \nu}$ holds.
  
  By $\mathbf{I}_{1, \eta}$ for all $\eta < \nu$, we need only justify that
  $(\mathbf{No}, (L_{\omega^{\eta}})_{\eta < \nu})$ is $\nu$-confluent to
  deduce that $\mathbf{I}_{1, \nu}$ holds. For $a \in \mathbf{No}^{>, \succ}$,
  by $\mathbf{I}_{2, \nu}$, there are a $\mathfrak{a} \in
  \mathbf{Mo}_{\alpha}' = \mathbf{Mo}_{\alpha}$ and a $\beta \assign
  \omega^{\eta} < \alpha$ with $E_{\beta} \left( 1 / 2 L_{\beta} a \right)
  \leqslant \mathfrak{a} \leqslant E_{\beta} (2 L_{\beta} a)$. We deduce that
  $L_{\beta} a \asymp L_{\beta} \mathfrak{a}$, thus $\mathfrak{a} \in
  \mathcal{E}_{\beta} [a]$. This concludes the proof.
\end{proof}

From now on, we assume that \hyperref[ind-ax1]{$\mathbf{I_{1,\mu}}$}, \hyperref[ind-ax2]{$\mathbf{I_{2,\mu}}$}, and~\hyperref[ind-ax3]{$\mathbf{I_{3,\mu}}$}
are satisfied for $\mu \geqslant 1$ and we define
\begin{eqnarray*}
  \nu & \assign & \mu + 1\\
  \alpha & \assign & \omega^{\nu}\\
  \beta & \assign & \omega^{\mu} .
\end{eqnarray*}
The remainder of the section is dedicated to the definition of $L_{\beta}$ and
the proof of the inductive hypotheses {\tmstrong{I}}\tmrsub{$1, \nu$},
{\tmstrong{I}}\tmrsub{$2, \nu$}, and~{\tmstrong{I}}\tmrsub{$3, \nu$} for
$\nu$. In combination with Propositions~\ref{prop-initial-case}
and~\ref{subsection-defining-hyperlogarithms}, this will complete our
induction and the proof of Theorem~\ref{th-confluent-hyperserial-field}.

\subsection{Defining the
hyperlogarithm}\label{subsection-defining-hyperlogarithms}

Recall that we have $\mathbf{Mo}_{\beta}' = \mathbf{Mo}_{\beta}$ by
\hyperref[ind-ax3]{$\mathbf{I_{3,\mu}}$}. In particular $\mathbf{Mo}_{\beta}$ is a surreal substructure.
Consider a~$\eta < \nu$. The skeleton $(\mathbf{No},
(L_{\omega^{\iota}})_{\iota < \eta})$ is a confluent hyperserial skeleton of
force $(\eta, \eta)$ by \hyperref[ind-ax1]{$\mathbf{I_{1,\mu}}$}. So for $a \in \mathbf{No}^{>, \succ}$,
(\ref{exp-class-char}) and \hyperref[ind-ax2]{$\mathbf{I_{2,\mu}}$} yield $\mathcal{E}_{\omega^{\eta}}
[a] =\mathcal{E}_{\omega^{\eta}}' [a]$.

In view of \hyperref[asymptotics]{$\mathbf{A_{\mu}}$} and \hyperref[monotonicity]{$\mathbf{M_{\mu}}$}, the simplest way to
define $L_{\beta}$ is {\tmem{via}} the cut equation:
\begin{equation}
  \forall \mathfrak{a} \in \mathbf{Mo}_{\beta}, \quad L_{\beta} \mathfrak{a}
  \assign \left\{ \mathbb{R}, L_{\beta} \mathfrak{a}' + \frac{1}{L_{< \beta}
  \mathfrak{a}'} : \mathfrak{a}' \in
  \mathfrak{a}_L^{\smash{\mathbf{Mo}_{\beta}}} | L_{\beta}
  \mathfrak{a}_R^{\smash{\mathbf{Mo}_{\beta}}} - \frac{1}{L_{< \beta}
  \mathfrak{a}}, L_{< \beta} \mathfrak{a} \right\} . \label{eq-poor-hyperlog}
\end{equation}
Note the asymmetry between left and right options $L_{\beta} \mathfrak{a}' +
(L_{< \beta} \mathfrak{a}')^{- 1}$ and $L_{\beta} \mathfrak{a}'' - (L_{<
\beta} \mathfrak{a})^{- 1}$ (instead of $L_{\beta} \mathfrak{a}'' - (L_{<
\beta} \mathfrak{a}'')^{- 1}$) for generic $\mathfrak{a}' \in
\mathfrak{a}_L^{\smash{\mathbf{Mo}_{\beta}}}$ and $\mathfrak{a}'' \in
\mathfrak{a}_R^{\smash{\mathbf{Mo}_{\beta}}}$. In Corollary~\ref{cor-true-def}
below, we will derive a more symmetric but equivalent cut equation
for~$L_{\beta}$, as promised in the introduction. For now, we prove that
(\ref{eq-poor-hyperlog}) is warranted and that \hyperref[asymptotics]{$\mathbf{A_{\mu}}$}, \hyperref[monotonicity]{$\mathbf{M_{\mu}}$}, and \hyperref[regularity]{$\mathbf{R_{\mu}}$} hold.

\begin{proposition}
  \label{prop-good-hyperlog-def}The function $L_{\beta}$ is well-defined on
  $\mathbf{Mo}_{\beta}$ and, for $\mathfrak{a} \in \mathbf{Mo}_{\beta}$, we
  have
  
  $\mathbf{H}_{\mathfrak{a}} : \quad \left( \forall \mathfrak{a}' \in
  \mathfrak{a}_L^{\smash{\mathbf{Mo}_{\beta}}}, L_{\beta} \mathfrak{a}' +
  \frac{1}{L_{< \beta} \mathfrak{a}'} < L_{\beta} \mathfrak{a}- \frac{1}{L_{<
  \beta} \mathfrak{a}} \right)$ \text{ and } 
  
  $\left( \forall \mathfrak{a}'' \in
  \mathfrak{a}_R^{\smash{\mathbf{Mo}_{\beta}}}, L_{\beta} \mathfrak{a}+
  \frac{1}{L_{< \beta} \mathfrak{a}} < L_{\beta} \mathfrak{a}'' -
  \frac{1}{L_{< \beta} \mathfrak{a}''} \right)$.
\end{proposition}

\begin{proof}
  We prove this by induction on $(\mathbf{Mo}_{\beta}, \sqsubseteq)$. Let
  $\mathfrak{a} \in \mathbf{Mo}_{\beta}$ such that $\mathbf{H}_{\mathfrak{b}}$
  holds for all~$\mathfrak{b} \in
  \mathfrak{a}_{\sqsubset}^{\smash{\mathbf{Mo}_{\beta}}}$. Let $\mathfrak{a}'
  \in \mathfrak{a}_L^{\smash{\mathbf{Mo}_{\beta}}}$ and $\mathfrak{a}'' \in
  \mathfrak{a}_R^{\smash{\mathbf{Mo}_{\beta}}}$. We have $\mathfrak{a}' \in
  (\mathfrak{a}'')_L^{\smash{\mathbf{Mo}_{\beta}}}$ or $\mathfrak{a}'' \in
  (\mathfrak{a}')_R^{\smash{\mathbf{Mo}_{\beta}}}$, so
  $\mathbf{H}_{\mathfrak{a}'}$ or $\mathbf{H}_{\mathfrak{a}''}$ yields
  \begin{eqnarray*}
    L_{\beta} \mathfrak{a}' + \frac{1}{L_{< \beta} \mathfrak{a}'} & < &
    L_{\beta} \mathfrak{a}'' - \frac{1}{L_{< \beta} \mathfrak{a}''} .
  \end{eqnarray*}
  For $\gamma < \beta$, we have $\hl_{\gamma + 1} \prec \frac{1}{2} 
  \hl_{\gamma}$ and $\frac{1}{L_{\gamma} \mathfrak{a}'} \succ
  \frac{1}{L_{\gamma} \mathfrak{a}''}, \frac{1}{L_{\gamma} \mathfrak{a}}$,
  whence
  \[ \widespacing{\frac{1}{L_{\gamma + 1} \mathfrak{a}'} \succ
     \frac{2}{L_{\gamma} \mathfrak{a}'} > \frac{1}{L_{\gamma} \mathfrak{a}'} +
     \frac{1}{L_{\gamma} \mathfrak{a}''} + \frac{1}{L_{\gamma} \mathfrak{a}},}
  \]
  for all $\gamma < \beta$. Hence,
  \begin{eqnarray*}
    L_{\beta} \mathfrak{a}' + \frac{1}{L_{< \beta} \mathfrak{a}'} & < &
    L_{\beta} \mathfrak{a}'' - \frac{1}{L_{< \beta} \mathfrak{a}} .
  \end{eqnarray*}
  We clearly have $L_{\beta} \mathfrak{a}'' - \frac{1}{L_{< \beta}
  \mathfrak{a}} \asymp L_{\beta} \mathfrak{a}'' >\mathbb{R}$. Finally,
  \begin{eqnarray*}
    L_{\beta} \mathfrak{a}' + \frac{1}{L_{< \beta} \mathfrak{a}'} & \asymp &
    L_{\beta} \mathfrak{a}' \hspace{1.2em} \prec \hspace{1.2em} L_{< \beta}
    \mathfrak{a}',
  \end{eqnarray*}
  so $L_{\beta} \mathfrak{a}' + \frac{1}{L_{< \beta} \mathfrak{a}'} < L_{<
  \beta} \mathfrak{a}$. This shows that $L_{\beta} \mathfrak{a}$ is defined
  and
  \[ \widespacing{L_{\beta} \mathfrak{a}' + \frac{1}{L_{< \beta}
     \mathfrak{a}'} < L_{\beta} \mathfrak{a}< L_{\beta} \mathfrak{a}'' -
     \frac{1}{L_{< \beta} \mathfrak{a}} .} \]
  Since $\mathfrak{a}' <\mathfrak{a}<\mathfrak{a}''$, it follows that
  \[ \widespacing{L_{\beta} \mathfrak{a}' + \frac{1}{L_{< \beta}
     \mathfrak{a}'} < L_{\beta} \mathfrak{a} \pm \frac{1}{L_{< \beta}
     \mathfrak{a}} < L_{\beta} \mathfrak{a}'' - \frac{1}{L_{< \beta}
     \mathfrak{a}} .} \]
  By induction, this proves $\mathbf{H}_{\mathfrak{a}}$ for all $\mathfrak{a}
  \in \mathbf{Mo}_{\beta}$.
\end{proof}

\begin{proposition}
  \label{prop-monotonicity}The axiom \hyperref[monotonicity]{$\mathbf{M_{\mu}}$} holds.{\nopagebreak}
\end{proposition}

\begin{proof}
  Let $\mathfrak{a}, \mathfrak{b} \in \mathbf{Mo}_{\beta}$ with $\mathfrak{a}
  \prec \mathfrak{b}$. Since $\mathbf{Mo}_{\beta}$ is a surreal substructure,
  there is a $\mathfrak{c} \in \mathbf{Mo}_{\beta}$ with $\mathfrak{c}
  \sqsubseteq \mathfrak{a}, \mathfrak{b}$ and~$\mathfrak{a} \leqslant
  \mathfrak{c} \leqslant \mathfrak{b}$. If $\mathfrak{a}<\mathfrak{c}$, then
  we have $L_{\beta} \mathfrak{a}+ (L_{< \beta} \mathfrak{a})^{- 1} <
  L_{\beta} \mathfrak{c}- (L_{< \beta} \mathfrak{c})^{- 1}$ by
  $\mathbf{H}_{\mathfrak{a}}$. If $\mathfrak{c}<\mathfrak{b}$, then we have
  $L_{\beta} \mathfrak{c}+ (L_{< \beta} \mathfrak{c})^{- 1} < L_{\beta}
  \mathfrak{b}- (L_{< \beta} \mathfrak{b})^{- 1}$ by
  $\mathbf{H}_{\mathfrak{b}}$. We cannot have both $\mathfrak{a}=\mathfrak{c}$
  and $\mathfrak{c}=\mathfrak{b}$, so this proves that~$L_{\beta}
  \mathfrak{a}+ (L_{< \beta} \mathfrak{a})^{- 1} < L_{\beta} \mathfrak{b}-
  (L_{< \beta} \mathfrak{b})^{- 1}$. Therefore \hyperref[monotonicity]{$\mathbf{M_{\mu}}$} holds.
\end{proof}

\begin{proposition}
  The axiom \hyperref[asymptotics]{$\mathbf{A_{\mu}}$} holds.
\end{proposition}

\begin{proof}
  The rightmost options in (\ref{eq-poor-hyperlog}) directly yield
  \hyperref[asymptotics]{$\mathbf{A_{\mu}}$}.
\end{proof}

\begin{proposition}
  \label{prop-regularity}The axiom \hyperref[regularity]{$\mathbf{R_{\mu}}$} holds.
\end{proposition}

\begin{proof}
  Let $\mathfrak{a} \in \mathbf{Mo}_{\beta}$ and write $\varphi \assign
  L_{\beta} \mathfrak{a}$. Let $\mathfrak{m} \in \tmop{supp} \varphi$ with
  $\mathfrak{m} \prec 1$. We have $\varphi < L_{< \beta} \mathfrak{a}$ and
  $\varphi_{\succ \mathfrak{m}} \asymp \varphi$ so $\varphi_{\succ
  \mathfrak{m}} < L_{< \beta} \mathfrak{a}$. Moreover $\varphi_{\succ
  \mathfrak{m}}$ is positive infinite. The number $\varphi_{\succ
  \mathfrak{m}}$ is strictly simpler than $\varphi$, so $\varphi_{\succ
  \mathfrak{m}}$ does not lie in the cut which defines $L_{\beta}
  \mathfrak{a}$ in (\ref{eq-poor-hyperlog}). Therefore, there is an
  {$\mathfrak{a}' \in \mathfrak{a}_L^{\smash{\mathbf{Mo}_{\beta}}}$} or an
  $\mathfrak{a}'' \in \mathfrak{a}_R^{\smash{\mathbf{Mo}_{\beta}}}$ and an
  ordinal $\gamma < \beta$ with $\varphi_{\succ \mathfrak{m}} \leqslant
  L_{\beta} \mathfrak{a}' + (L_{\gamma} \mathfrak{a}')^{- 1}$ or
  $\varphi_{\succ \mathfrak{m}} \geqslant L_{\beta} \mathfrak{a}'' -
  (L_{\gamma} \mathfrak{a})^{- 1}$. Consider the first case. We have
  $L_{\beta} \mathfrak{a}' + (L_{< \beta} \mathfrak{a}')^{- 1} < \varphi
  \leqslant \varphi_{\succ \mathfrak{m}} + \varphi_{\mathfrak{m}}
  \mathfrak{m}+ \delta$ for a certain $\delta \prec \mathfrak{m}$. So
  $\varphi_{\mathfrak{m}} > 0$ and
  \begin{eqnarray*}
    \frac{1}{L_{< \beta} \mathfrak{a}'} & < & \frac{1}{L_{\gamma}
    \mathfrak{a}'} + \varphi_{\mathfrak{m}} \mathfrak{m}.
  \end{eqnarray*}
  For $\rho < \beta$ with $\gamma < \rho$, we have $(L_{\rho}
  \mathfrak{a}')^{- 1} \succ (L_{\gamma} \mathfrak{a}')^{- 1}$ so $(L_{\rho}
  \mathfrak{a}')^{- 1} - (L_{\gamma} \mathfrak{a}')^{- 1} \asymp (L_{\rho}
  \mathfrak{a}')^{- 1}$. We deduce that $(L_{\rho} \mathfrak{a}')^{- 1}
  \preccurlyeq \mathfrak{m}$ for all such $\rho$. It follows that $(L_{\rho}
  \mathfrak{a})^{- 1} \preccurlyeq \mathfrak{m}$ for all $\rho < \beta$. In
  the second case, we directly get $\mathfrak{m} \succ (L_{\gamma}
  \mathfrak{a})^{- 1}$. This proves that we always have $\mathfrak{m} \succ
  (L_{< \beta} \mathfrak{a})^{- 1}$. In other words $\tmop{supp} \varphi \succ
  (L_{< \beta} \mathfrak{a})^{- 1}$, whence \hyperref[regularity]{$\mathbf{R_{\mu}}$} holds.
\end{proof}

\begin{proposition}
  \label{cor-hyperlog-uniform}If $\mu$ is a successor ordinal, then the cut
  equation \tmtextup{(\ref{eq-poor-hyperlog})} is uniform.
\end{proposition}

\begin{proof}
  Let $(\mathfrak{L}_{\mathfrak{a}}, \mathfrak{R}_{\mathfrak{a}})$ be a cut
  representation in $\mathbf{Mo}_{\beta}$ and set $\mathfrak{a} \assign \{
  \mathfrak{L}_{\mathfrak{a}} |\mathfrak{R}_{\mathfrak{a}}
  \}_{\mathbf{Mo}_{\beta}}$. For $\mathfrak{l} \in
  \mathfrak{L}_{\mathfrak{a}}$, we have $L_{\beta} \mathfrak{l}< L_{\beta}
  \mathfrak{a} \prec L_{< \beta} \mathfrak{a}$ so $L_{\beta} \mathfrak{l}<
  L_{< \beta} \mathfrak{a}$. For $\mathfrak{r} \in
  \mathfrak{R}_{\mathfrak{a}}$, we have $L_{\beta} \mathfrak{l}+ (L_{< \beta}
  \mathfrak{l})^{- 1} < L_{\beta} \mathfrak{r}$ by \hyperref[monotonicity]{$\mathbf{M_{\mu}}$}. Since
  $\mathfrak{l} \prec \mathfrak{a}$, it follows that \[L_{\beta} \mathfrak{l}+
  (L_{< \beta} \mathfrak{l})^{- 1} < L_{\beta} \mathfrak{r}- (L_{< \beta}
  \mathfrak{a})^{- 1}.\] We may thus define the number
  \begin{eqnarray*}
    \varphi & \assign & \left\{ \mathbb{R}, L_{\beta} \mathfrak{l}+
    \frac{1}{L_{< \beta} \mathfrak{l}} \suchthat \mathfrak{l} \in
    \mathfrak{L}_{\mathfrak{a}} | L_{\beta} \mathfrak{R}_{\mathfrak{a}} -
    \frac{1}{L_{< \beta} \mathfrak{a}}, L_{< \beta} \mathfrak{a} \right\} .
  \end{eqnarray*}
  In order to show that (\ref{eq-poor-hyperlog}) is uniform, we need to prove
  that $L_{\beta} \mathfrak{a}= \varphi$, for any choice of the cut
  representation $(\mathfrak{L}_{\mathfrak{a}}, \mathfrak{R}_{\mathfrak{a}})$.
  We will do so by proving that $L_{\beta} \mathfrak{a} \sqsubseteq \varphi$
  and $\varphi \sqsubseteq L_{\beta} \mathfrak{a}$.
  
  Recall that $(\mathfrak{L}_{\mathfrak{a}}, \mathfrak{R}_{\mathfrak{a}})$ is
  cofinal with respect to $\left( \mathfrak{a}_L^{\smash{\mathbf{Mo}_{\beta}}}
  |\mathfrak{a}_R^{\smash{\mathbf{Mo}_{\beta}}} \right)$ and that $L_{\beta}$
  is strictly increasing. Consequently, we have
  \begin{eqnarray*}
    \varphi & < & L_{\beta} \mathfrak{a}_R^{\smash{\mathbf{Mo}_{\beta}}} -
    (L_{< \beta} \mathfrak{a})^{- 1} .
  \end{eqnarray*}
  Given $\mathfrak{a}' \in \mathfrak{a}_L^{\smash{\mathbf{Mo}_{\beta}}}$,
  there is an $\mathfrak{l} \in \mathfrak{L}_{\mathfrak{a}}$ with
  $\mathfrak{a}' \leqslant \mathfrak{l}$. In view of \hyperref[monotonicity]{$\mathbf{M_{\mu}}$}, we have
  \[L_{\beta} \mathfrak{a}' + (L_{\gamma} \mathfrak{a}')^{- 1} \leqslant
  L_{\beta} \mathfrak{l}+ (L_{\gamma} \mathfrak{l})^{- 1}\] for all $\gamma <
  \beta$, so $\varphi > \left\{ L_{\beta} \mathfrak{a}' + (L_{< \beta}
  \mathfrak{a}')^{- 1} \suchthat \mathfrak{a}' \in
  \mathfrak{a}_L^{\smash{\mathbf{Mo}_{\beta}}} \right\}$. This proves that
  $\varphi$ lies in the cut defining $L_{\beta} \mathfrak{a}$ as per
  (\ref{eq-poor-hyperlog}), whence $L_{\beta} \mathfrak{a} \sqsubseteq
  \varphi$.
  
  Conversely, in order to prove that $\varphi \sqsubseteq L_{\beta}
  \mathfrak{a}$, it suffices to show that $L_{\beta} \mathfrak{a}$ lies in the
  cut
  \[ \left( L_{\beta} \mathfrak{l}+ \frac{1}{L_{< \beta} \mathfrak{l}}
     \suchthat \mathfrak{l} \in \mathfrak{L}_{\mathfrak{a}} | L_{\beta}
     \mathfrak{R}_{\mathfrak{a}} - \frac{1}{L_{< \beta} \mathfrak{a}} \right)
     . \]
  Let $\mathfrak{l} \in \mathfrak{L}_{\mathfrak{a}}$ and let $\mathfrak{b} \in
  \mathbf{Mo}_{\beta}$ be $\sqsubseteq$-maximal with $\mathfrak{b} \sqsubseteq
  \mathfrak{l}, \mathfrak{a}$. We have $\mathfrak{l} \leqslant \mathfrak{b}
  \leqslant \mathfrak{a}$, whence $L_{\beta} \mathfrak{b} \leqslant L_{\beta}
  \mathfrak{a}$, by~\hyperref[monotonicity]{$\mathbf{M_{\mu}}$}. If $\mathfrak{b} \sqsubset
  \mathfrak{l}$, then $\mathfrak{b} \in
  \mathfrak{l}_R^{\smash{\mathbf{Mo}_{\beta}}}$, so
  $\mathbf{H}_{\mathfrak{l}}$ yields $L_{\beta} \mathfrak{l}+ (L_{< \beta}
  \mathfrak{l})^{- 1} < L_{\beta} \mathfrak{b}$ and $L_{\beta} \mathfrak{l}+
  (L_{< \beta} \mathfrak{l})^{- 1} < L_{\beta} \mathfrak{a}$. Otherwise
  $\mathfrak{l}=\mathfrak{b} \in
  \mathfrak{a}_L^{\smash{\mathbf{Mo}_{\beta}}}$, so
  $\mathbf{H}_{\mathfrak{a}}$ yields $L_{\beta} \mathfrak{l}+ (L_{< \beta}
  \mathfrak{l})^{- 1} < L_{\beta} \mathfrak{a}$. This proves that {$\{
  L_{\beta} \mathfrak{l}+ (L_{< \beta} \mathfrak{l})^{- 1} \suchthat
  \mathfrak{l} \in \mathfrak{L}_{\mathfrak{a}} \} < L_{\beta} \mathfrak{a}$}.
  
  Let $\mathfrak{r} \in \mathfrak{R}_{\mathfrak{a}}$ and let $\mathfrak{c} \in
  \mathbf{Mo}_{\beta}$ be $\sqsubseteq$-maximal with $\mathfrak{c} \sqsubseteq
  \mathfrak{r}, \mathfrak{a}$. As above, if $\mathfrak{c} \sqsubset
  \mathfrak{a}$, then $\mathfrak{c} \in
  \mathfrak{a}_R^{\smash{\mathbf{Mo}_{\beta}}}$ so $\mathbf{H}_{\mathfrak{a}}$
  yields $L_{\beta} \mathfrak{a}< L_{\beta} \mathfrak{c}- (L_{< \beta}
  \mathfrak{a})^{- 1}$, whence \[L_{\beta} \mathfrak{a}< L_{\beta}
  \mathfrak{r}- (L_{< \beta} \mathfrak{a})^{- 1}.\] Otherwise
  $\mathfrak{a}=\mathfrak{c} \in \mathfrak{r}_L^{\smash{\mathbf{Mo}_{\beta}}}$.
  Therefore $\mathbf{H}_{\mathfrak{r}}$ yields $L_{\beta} \mathfrak{r}> L_{\beta}
  \mathfrak{a}+ (L_{< \beta} \mathfrak{a})^{- 1}$. We deduce that $L_{\beta}
  \mathfrak{a}< L_{\beta} \mathfrak{R}_{\mathfrak{a}} - (L_{< \beta}
  \mathfrak{a})^{- 1}$ and we conclude by induction.
\end{proof}

\subsection{Functional equation}\label{subsection-functional-equation}

In this subsection we derive \hyperref[functional-eq]{$\mathbf{FE_{\mu}}$}, under the assumption that
$\mu$ is a successor ordinal. We start with the following inequality. Write $\rho \assign \beta_{/ \omega}$

\begin{lemma}
  \label{lem-not-too-strong-hyperexp}If $\mu > 1$, then we have $E_{< \rho} < E_{\rho} H_2 L_{\rho}$ on
  $\mathbf{No}^{>, \succ}$.
\end{lemma}

\begin{proof}
 For $\gamma < \rho$, there are $\eta < \mu_-$ and $n < \omega$
  with $\gamma < \omega^{\eta} n$. We have
  \[ \widespacing{E_{\gamma} < E_{\omega^{\eta} n} = E_{\omega^{\eta + 1}} T_n
     L_{\omega^{\eta + 1}} < E_{\omega^{\eta + 1}} H_2 L_{\omega^{\eta + 1}}}
  \]
  on $\mathbf{No}^{>, \succ}$ by (\ref{eq-functional-total}). Note that $\eta
  + 1 \leqslant \mu_- < \mu$, so \hyperref[ind-ax2]{$\mathbf{I_{2,\mu}}$} yields
  \begin{eqnarray*}
    E_{\omega^{\eta + 1}} H_2 L_{\omega^{\eta + 1}} & \leqslant & E_{\rho} H_2 L_{\rho},
  \end{eqnarray*}
  whence~$E_{\gamma} < E_{\rho} H_2 L_{\rho}$.
\end{proof}

Let $\mathfrak{a} \in \mathbf{Mo}_{\beta}$. Since $\mathbf{Mo}_{\beta}$ is a
surreal substructure, we may consider the $L_{< \beta}$-atomic number
\begin{eqnarray*}
  \mathfrak{b} & \assign & \left\{ L_{\rho}
  \mathfrak{a}_L^{\smash{\mathbf{Mo}_{\beta}}} |L_{\rho}
  \mathfrak{a}_R^{\smash{\mathbf{Mo}_{\beta}}}, \mathfrak{a}
  \right\}_{\mathbf{Mo}_{\beta}} .
\end{eqnarray*}
We claim that $\mathfrak{b}= L_{\rho} \mathfrak{a}$. Assume that
$\mu = 1$ and write $\mathfrak{a}= \Xi_{\mathbf{Mo}_{\omega}} a$. We have
\begin{align*}
  \log \mathfrak{a} & = \Xi_{\mathbf{Mo}_{\omega}} (a - 1) &&
  \text{(by {\cite[Proposition~2.5]{vdH:bm}})}\\
  & = \Xi_{\mathbf{Mo}_{\omega}}  \{ a_L - 1| a_R - 1, a \} &&
  \text{(by (\ref{eq-uniform-sum}))}\\
  & = \{ \Xi_{\mathbf{Mo}_{\omega}} (a_L - 1) | \Xi_{\mathbf{Mo}_{\omega}}
  (a_R - 1), \Xi_{\mathbf{Mo}_{\omega}} a \}_{\mathbf{Mo}_{\omega}}\\
  & = \{ \log \Xi_{\mathbf{Mo}_{\omega}} a_L | \log
  \Xi_{\mathbf{Mo}_{\omega}} a_R, \Xi_{\mathbf{Mo}_{\omega}} a
  \}_{\mathbf{Mo}_{\omega}} && \text{(by
  {\cite[Proposition~2.5]{vdH:bm}})}\\
  & = \{ \log \mathfrak{a}_L^{\mathbf{Mo}_{\omega}} | \log
  \mathfrak{a}_R^{\mathbf{Mo}_{\omega}}, \mathfrak{a}
  \}_{\mathbf{Mo}_{\omega}}\\
  & = \mathfrak{b}.
\end{align*}
Assume now that $\mu > 1$. The function $L_{\rho}$ is strictly
increasing with $L_{\rho} < \tmop{id}_{\mathbf{No}^{>, \succ}}$.
Therefore
\begin{eqnarray*}
  L_{\rho} \mathfrak{a} & \in & \left( L_{\rho}
  \mathfrak{a}_L^{\smash{\mathbf{Mo}_{\beta}}} |L_{\rho}
  \mathfrak{a}_R^{\smash{\mathbf{Mo}_{\beta}}}, \mathfrak{a}
  \right)_{\mathbf{Mo}_{\beta}},
\end{eqnarray*}
so~$\mathfrak{b} \sqsubseteq L_{\rho} \mathfrak{a}$. Since
$\mathfrak{a} \in \mathbf{Mo}_{\beta}$, the cut equation
(\ref{eq-poor-hyperlog}) for $\mu_-$ yields
\begin{eqnarray}
  L_{\rho} \mathfrak{a} = \left\{ \mathbb{R}, L_{\rho} \mathfrak{a}' + \frac{1}{L_{< \beta} \mathfrak{a}'} \suchthat
  \mathfrak{a}' \in \mathfrak{a}_L^{\smash{\mathbf{Mo}_{\rho}}}
  |L_{\rho} \mathfrak{a}_R^{\smash{\mathbf{Mo}_{\rho}}} - \frac{1}{L_{< \beta} \mathfrak{a}}, L_{< \rho}
  \mathfrak{a} \right\}.  \label{eq-subhyperlog}
\end{eqnarray}
Given $\mathfrak{a}' \in \mathfrak{a}_L^{\smash{\mathbf{Mo}_{rho}}}$, we have $\mathfrak{d}_{\beta} (\mathfrak{a}') \in
\mathfrak{a}_L^{\smash{\mathbf{Mo}_{\beta}}}$ and $\mathfrak{a}' \in
\mathcal{E}_{\beta} [\mathfrak{d}_{\beta} (\mathfrak{a}')]$. We deduce that
\[ \widespacing{L_{\rho} \mathfrak{a}' \in L_{\rho}
   \mathcal{E}_{\beta} [\mathfrak{d}_{\beta} (\mathfrak{a}')]
   =\mathcal{E}_{\beta} [L_{\rho} \mathfrak{d}_{\beta}
   (\mathfrak{a}')] .} \]
Moreover, by definition, we have
\[ \widespacing{\mathfrak{b}>\mathcal{E}_{\beta}' [L_{\rho}
   \mathfrak{d}_{\beta} (\mathfrak{a}')] =\mathcal{E}_{\beta} [L_{\rho} \mathfrak{d}_{\beta} (\mathfrak{a}')],} \]
so $\mathfrak{b} \succ L_{\rho} \mathfrak{a}'$. Symmetric
arguments yield~$\mathfrak{b} \prec L_{\rho}
\mathfrak{a}_R^{\smash{\mathbf{Mo}_{\rho}}}$.
Lemma~\ref{lem-not-too-strong-hyperexp} implies that {$L_{< \rho}
\mathfrak{a} \subseteq \mathcal{E}_{\beta} [\mathfrak{a}]$}, whence
$\mathfrak{d}_{\beta} (L_{< \rho} \mathfrak{a}) = \{ \mathfrak{a}
\}$. We get $\mathfrak{b}<\mathcal{E}_{\beta} \mathfrak{d}_{\beta} (L_{<
\rho} \mathfrak{a})$, whence $\mathfrak{b}< L_{< \rho}
\mathfrak{a}$. Thus $\mathfrak{b}$ lies in the cut defining $L_{\rho} \mathfrak{a}$ in (\ref{eq-subhyperlog}), so~$L_{\rho}
\mathfrak{a} \sqsubseteq \mathfrak{b}$. This proves our claim that
\begin{equation}
  \forall \mathfrak{a} \in \mathbf{Mo}_{\beta}, \quad L_{\rho}
  \mathfrak{a}= \left\{ L_{\rho}
  \mathfrak{a}_L^{\smash{\mathbf{Mo}_{\beta}}} |L_{\rho}
  \mathfrak{a}_R^{\smash{\mathbf{Mo}_{\beta}}}, \mathfrak{a}
  \right\}_{\mathbf{Mo}_{\beta}} . \label{eq-weaker-hyperlog}
\end{equation}
We now derive \hyperref[functional-eq]{$\mathbf{FE_{\mu}}$}.

\begin{proposition}
  \label{prop-functional-equation}For $\mathfrak{a} \in \mathbf{Mo}_{\beta}$,
  we have $L_{\beta} L_{\rho} \mathfrak{a}= L_{\beta}
  \mathfrak{a}- 1$.
\end{proposition}

\begin{proof}
  We prove this by induction on $(\mathbf{Mo}_{\beta}, \sqsubseteq)$. Let
  $\mathfrak{a} \in \mathbf{Mo}_{\beta}$ be such that the result holds
  on~$\mathfrak{a}_{\sqsubset}^{\smash{\mathbf{Mo}_{\beta}}}$. By
  (\ref{eq-weaker-hyperlog}), we have
  \begin{eqnarray*}
    L_{\beta_{/ \omega}} \mathfrak{a} & = & \left\{ L_{\rho}
    \mathfrak{a}_L^{\smash{\mathbf{Mo}_{\beta}}} |L_{\rho}
    \mathfrak{a}_R^{\smash{\mathbf{Mo}_{\beta}}}, \mathfrak{a}
    \right\}_{\mathbf{Mo}_{\beta}} .
  \end{eqnarray*}
  Let $\mathfrak{a}'$ and $\mathfrak{a}''$ range in
  $\mathfrak{a}_L^{\smash{\mathbf{Mo}_{\beta}}}$ and
  $\mathfrak{a}_R^{\smash{\mathbf{Mo}_{\beta}}}$ respectively.
  Proposition~\ref{cor-hyperlog-uniform} and our induction hypothesis yield:
  \begin{eqnarray*}
    L_{\beta} L_{\rho} \mathfrak{a} & = \left\{ \mathbb{R},
    L_{\beta} L_{\rho} \mathfrak{a}' + \frac{1}{L_{< \beta}
    L_{\rho} \mathfrak{a}'} | L_{\beta} L_{\rho}
    \mathfrak{a}'' - \frac{1}{L_{< \beta} L_{\rho} \mathfrak{a}},
    L_{\beta} \mathfrak{a}- \frac{1}{L_{< \beta} \mathfrak{a}}, L_{< \beta}
    L_{\rho} \mathfrak{a} \right\}\\
    & = \left\{ \mathbb{R}, L_{\beta} \mathfrak{a}' - 1 + \frac{1}{L_{<
    \beta} \mathfrak{a}'} | L_{\beta} \mathfrak{a}'' - 1 - \frac{1}{L_{<
    \beta} \mathfrak{a}}, L_{\beta} \mathfrak{a}- \frac{1}{L_{< \beta}
    \mathfrak{a}}, L_{< \beta} \mathfrak{a} \right\} .
  \end{eqnarray*}
  On the other hand, we have
  \begin{eqnarray*}
    L_{\beta} \mathfrak{a}- 1 & = & \left\{ \mathbb{R}- 1, L_{\beta}
    \mathfrak{a}' + \frac{1}{L_{< \beta} \mathfrak{a}'} - 1 | L_{\beta}
    \mathfrak{a}'' - \frac{1}{L_{< \beta} \mathfrak{a}} - 1, L_{< \beta}
    \mathfrak{a}- 1, L_{\beta} \mathfrak{a} \right\}\\
    & = & \left\{ \mathbb{R}, L_{\beta} \mathfrak{a}' + \frac{1}{L_{< \beta}
    \mathfrak{a}'} - 1 | L_{\beta} \mathfrak{a}'' - \frac{1}{L_{< \beta}
    \mathfrak{a}} - 1, L_{\beta} \mathfrak{a}, L_{< \beta} \mathfrak{a}
    \right\} .
  \end{eqnarray*}
  In order to conclude that~$L_{\beta} L_{\rho} \mathfrak{a}=
  L_{\beta} \mathfrak{a}- 1$, it remains to show that $L_{\beta} \mathfrak{a}-
  1 < L_{\beta} \mathfrak{a}- (L_{< \beta} \mathfrak{a})^{- 1}$ and that
  $L_{\beta} L_{\rho} \mathfrak{a}< L_{\beta} \mathfrak{a}$. The
  first inequality holds because $(L_{< \beta} \mathfrak{a})^{- 1}$ is a set
  of infinitesimal numbers. An easy induction shows that $L_{\rho}
  a < a$ for all $a \in \mathbf{No}^{>, \succ}$. The second inequality
  follows, because~$L_{\beta}$ is strictly increasing on
  $\mathbf{Mo}_{\beta}$. This completes our inductive proof.
\end{proof}

Combining our results so far, we have proved that $(\mathbf{No},
(L_{\omega^{\eta}})_{\eta < \nu})$ is a hyperserial skeleton of force $\nu$.

\subsection{Confluence}\label{subsection-confluence}

We next prove that $(\mathbf{No}, (L_{\omega^{\eta}})_{\eta < \nu})$ is
$\nu$-confluent.

\begin{lemma}
  \label{lem-limit-hypermon}If $\mu$ is a non-zero limit ordinal, then the
  function groups $\mathcal{E}_{\beta}'$ and $\mathcal{E}_{\beta}^{\ast}$ are
  mutually pointwise cofinal. In particular, we have $\mathbf{Mo}_{\beta} =
  \mathbf{Mo}_{\beta}^{\ast}$ and $\mathbf{Tr}_{\beta} =
  \mathbf{Tr}_{\beta}^{\ast}$.
\end{lemma}

\begin{proof}
  For $\gamma \in (0, \beta)$ and $r \in \mathbb{R}^{>}$, we have $E_{\gamma}
  H_r L_{\gamma} < E_{\gamma}$ since $H_r < E_{\gamma}$. We have
  \begin{eqnarray*}
    \{ L_{\rho}, E_{\rho} \suchthat \rho \in (0, \beta) \} & \legeangle &
    \mathcal{E}_{\beta}^{\ast},
  \end{eqnarray*}
  whereas \hyperref[ind-ax2]{$\mathbf{I_{2,\mu}}$} yields
  \begin{eqnarray*}
    \{ E_{\rho} H_r L_{\rho} \suchthat \rho \in (0, \beta) \} & \legeangle &
    \mathcal{E}_{\beta}' .
  \end{eqnarray*}
  Therefore $\mathcal{E}_{\beta}' \leqangle \mathcal{E}_{\beta}^{\ast}$. For
  $\rho < \beta$, there is $\eta < \mu$ with $\rho < \omega^{\eta}$. By
  (\ref{eq-functional-total}), we have
  \[ \widespacing{E_{\rho} < E_{\omega^{\eta}} = E_{\omega^{\eta + 1}} T_1
     L_{\omega^{\eta + 1}} < E_{\omega^{\eta + 1}} H_2 L_{\omega^{\eta + 1}},}
  \]
  which proves the inequality $\mathcal{E}_{\beta}^{\ast} \leqangle
  \mathcal{E}_{\beta}'$.
\end{proof}

\begin{lemma}
  \label{lem-confluence}For each $a \in \mathbf{No}^{>, \succ}$, any
  $\sqsubseteq$-minimal element of $\mathcal{E}_{\alpha} [a]$ is $L_{<
  \alpha}$-atomic.
\end{lemma}

\begin{proof}
  Let $\mathfrak{A}$ denote the class of numbers $\mathfrak{a} \in
  \mathbf{No}^{>, \succ}$ that are $\sqsubseteq$-minimal in
  $\mathcal{E}_{\alpha} [\mathfrak{a}]$. Any such $\sqsubseteq$-minimal number
  $\mathfrak{a}$ is also $\sqsubseteq$-minimal in $\mathcal{E}_{\beta}'
  [\mathfrak{a}] =\mathcal{E}_{\beta} [\mathfrak{a}] \subseteq
  \mathcal{E}_{\alpha} [\mathfrak{a}]$, hence $L_{< \beta}$-atomic. Thus
  $L_{\beta}$ is defined on $\mathfrak{A}$. It is enough to prove that
  $\mathfrak{A}$ is closed under $L_{\beta}$ in order to obtain that
  $\mathfrak{A} \subseteq \mathbf{Mo}_{\alpha}$.
  
  Consider $\mathfrak{a} \in \mathfrak{A}$, and recall that we have
  \begin{eqnarray}
    L_{\beta} \mathfrak{a} & = & \left\{ \mathbb{R}, L_{\beta} \mathfrak{a}' +
    \frac{1}{L_{< \beta} \mathfrak{a}'} \suchthat \mathfrak{a}' \in
    \mathfrak{a}_L^{\smash{\mathbf{Mo}_{\beta}}} | L_{\beta}
    \mathfrak{a}_R^{\smash{\mathbf{Mo}_{\beta}}} - \frac{1}{L_{< \beta}
    \mathfrak{a}}, L_{< \beta} \mathfrak{a} \right\} .  \label{eq-confluence}
  \end{eqnarray}
  Assume for contradiction that $L_{\beta} \mathfrak{a}$ is not
  $\sqsubseteq$-minimal in $\mathcal{E}_{\alpha} [L_{\beta} \mathfrak{a}]$. So
  there is a $\mathfrak{b} \in \mathcal{E}_{\alpha} [L_{\beta} \mathfrak{a}]$
  with~$\mathfrak{b} \sqsubset L_{\beta} \mathfrak{a}$. This implies that
  $\mathfrak{b}$ lies outside the cut defining $L_{\beta} \mathfrak{a}$, so
  $\mathfrak{b}$ is larger than a right option of (\ref{eq-confluence}) or
  smaller than a left option of (\ref{eq-confluence}).
  
  Assume first that $\mathfrak{b}< L_{\beta} \mathfrak{a}$. So there is an
  $\mathfrak{a}' \in \mathfrak{a}_L^{\smash{\mathbf{Mo}_{\beta}}}$ with
  $\mathfrak{b} \preccurlyeq L_{\beta} \mathfrak{a}'$. We have
  $\mathfrak{d}_{\alpha} (\mathfrak{a}) =\mathfrak{d}_{\alpha} (\mathfrak{b})$
  so there is an $n \in \mathbb{N}$ with
  \begin{eqnarray*}
    (L_{\beta} \circ \mathfrak{d}_{\beta})^{\circ n} (\mathfrak{b}) & \asymp &
    (L_{\beta} \circ \mathfrak{d}_{\beta})^{\circ n} (L_{\beta} \mathfrak{a})
    .
  \end{eqnarray*}
  Thus
  \begin{eqnarray*}
    (L_{\beta} \circ \mathfrak{d}_{\beta})^{\circ (n + 1)} (\mathfrak{a}') &
    \asymp & (L_{\beta} \circ \mathfrak{d}_{\beta})^{\circ (n + 1)}
    (\mathfrak{a}) .
  \end{eqnarray*}
  This contradicts the $\sqsubseteq$-minimality of~$\mathfrak{a}$.
  
  Now consider the other case when~$\mathfrak{b}> L_{\beta} \mathfrak{a}$. In
  particular, $\mathfrak{b}$ must be larger than a right option of
  (\ref{eq-confluence}). Symmetric arguments imply that we cannot have
  $\mathfrak{b} \succcurlyeq L_{\beta} \mathfrak{a}''$ for some
  $\mathfrak{a}'' \in \mathfrak{a}_R^{\smash{\mathbf{Mo}_{\beta}}}$. So there
  must exist a $\gamma < \beta$ with $\mathfrak{b} \geqslant L_{\gamma}
  \mathfrak{a}$. If $\mu$ is a limit ordinal, then $\gamma < \mu_-$ so
  Lemma~\ref{lem-limit-hypermon} yields $\mathfrak{d}_{\beta} (L_{\gamma}
  \mathfrak{a}) =\mathfrak{a}$, whence~$\mathfrak{d}_{\beta} (\mathfrak{b})
  \succcurlyeq \mathfrak{a}$. If $\mu$ is a successor ordinal, then there is a
  $k \in \mathbb{N}$ with $\gamma \leqslant \beta_{/ \omega} k$, so
  \[ \widespacing{\mathfrak{d}_{\beta} (\mathfrak{b}) \geqslant
     \mathfrak{d}_{\beta} (L_{(\beta_{/ \omega}) k} \mathfrak{a}) =
     L_{(\beta_{/ \omega}) k} \mathfrak{a}} \]
  and Proposition~\ref{prop-functional-equation} yields $L_{\beta}
  \mathfrak{d}_{\beta} (\mathfrak{b}) \succcurlyeq L_{\beta} \mathfrak{a}- k
  \succcurlyeq L_{\beta} \mathfrak{a}$. In both cases, we thus have $L_{\beta}
  \mathfrak{d}_{\beta} (\mathfrak{b}) \succcurlyeq L_{\beta} \mathfrak{a}$.
  For any integer $n > 1$, we deduce that
  \[ \widespacing{(L_{\beta} \circ \mathfrak{d}_{\beta})^{\circ n}
     (\mathfrak{b}) \geqslant (L_{\beta} \circ \mathfrak{d}_{\beta})^{\circ n}
     (\mathfrak{a}) > (L_{\beta} \circ \mathfrak{d}_{\beta})^{\circ (n + 1)}
     (\mathfrak{a}) = (L_{\beta} \circ \mathfrak{d}_{\beta})^{\circ n}
     (L_{\beta} \mathfrak{a}) .} \]
  This contradicts the fact that $\mathfrak{b}$ lies in $\mathcal{E}_{\alpha}
  [L_{\beta} \mathfrak{a}]$.
  
  We have shown that the cases $\mathfrak{b}< L_{\beta} \mathfrak{a}$ and
  $\mathfrak{b}> L_{\beta} \mathfrak{a}$ both lead to a contradiction.
  Consequently, $L_{\beta} \mathfrak{a}$ is $\sqsubseteq$-minimal in
  $\mathcal{E}_{\alpha} [L_{\beta} \mathfrak{a}]$ and we conclude that
  $L_{\beta} \mathfrak{A} \subseteq \mathfrak{A}$, as claimed.
\end{proof}

\begin{corollary}
  \label{cor-confluence}$(\mathbf{No}, (L_{\omega^{\eta}})_{\eta < \nu})$ is
  $\nu$-confluent.
\end{corollary}

\begin{proof}
  We already know that $(\mathbf{No}, (L_{\omega^{\eta}})_{\eta < \mu})$ is
  $\mu$-confluent by \ref{ind-ax1}. Recall that $(\mathbf{No}, \sqsubseteq)$
  is well-founded, so each class $\mathcal{E}_{\alpha} [a]$ for $a \in
  \mathbf{No}^{>, \succ}$ contains a $\sqsubseteq$-minimal element.
  Lemma~\ref{lem-confluence} therefore implies that $\mathbf{No}$ is
  $\nu$-confluent.
\end{proof}

The corollary implies that $(\mathbf{No}, (L_{\omega^{\eta}})_{\eta < \nu})$
is a confluent hyperserial skeleton of force $\nu$. Moreover, the class
$\mathbf{No}_{\succ, \beta}$ is that of $\trianglelefteqslant$-minima and thus
$\sqsubseteq$-minima in the convex classes
\begin{eqnarray*}
  \mathcal{L}_{\beta} [a] & = & \left\{ b \in a + \mathbf{No}^{\prec}
  \suchthat b = a \vee \left( \exists \gamma < \beta, a <
  \hl_{\beta}^{\mathord{\uparrow} \gamma} \circ | a - b |^{- 1} \right)
  \right\},
\end{eqnarray*}
for $a \in \mathbf{No}^{>, \succ}$. In other words, we have
$\mathbf{No}_{\succ, \beta} = \mathbf{Smp}_{\mathcal{L}_{\beta}}$. In order to
conclude that $\mathbf{No}_{\succ, \beta}$ is a~surreal substructure, we still
need to prove that the convex partition $\mathcal{L}_{\beta}$ is thin. This
will be done at the end of section~\ref{subsection-hyperexponentials} below.

\begin{proposition}
  \label{prop-uniform-hyperlog}The cut equation
  \tmtextup{(\ref{eq-poor-hyperlog})} is uniform.
\end{proposition}

\begin{proof}
  Let $(\mathfrak{L}_{\mathfrak{a}}, \mathfrak{R}_{\mathfrak{a}})$ be a cut
  representation in $\mathbf{Mo}_{\beta}$ and set $\mathfrak{a} \assign \{
  \mathfrak{L}_{\mathfrak{a}} |\mathfrak{R}_{\mathfrak{a}}
  \}_{\mathbf{Mo}_{\beta}}$. We have
  \[ \widespacing{\mathcal{L}_{\beta} [L_{\beta} \mathfrak{L}_{\mathfrak{a}}]
     <\mathcal{L}_{\beta} [L_{\beta} \mathfrak{a}] <\mathcal{L}_{\beta}
     [L_{\beta} \mathfrak{R}_{\mathfrak{a}}] .} \]
  By (\ref{L-when-E}), this shows that
  \begin{eqnarray*}
    L_{\beta} \mathfrak{a} & \in & \left( \mathbb{R}, L_{\beta} \mathfrak{l}+
    \frac{1}{L_{< \beta} \mathfrak{l}} \suchthat \mathfrak{l} \in
    \mathfrak{L}_{\mathfrak{a}} | L_{\beta} \mathfrak{R}_{\mathfrak{a}} -
    \frac{1}{L_{< \beta} \mathfrak{a}}, L_{< \beta} \mathfrak{a} \right) .
  \end{eqnarray*}
  In particular, the number
  \begin{eqnarray*}
    \varphi & \assign & \left\{ \mathbb{R}, L_{\beta} \mathfrak{l}+
    \frac{1}{L_{< \beta} \mathfrak{l}} \suchthat \mathfrak{l} \in
    \mathfrak{L}_{\mathfrak{a}} | L_{\beta} \mathfrak{R}_{\mathfrak{a}} -
    \frac{1}{L_{< \beta} \mathfrak{a}}, L_{< \beta} \mathfrak{a} \right\}
  \end{eqnarray*}
  is well-defined, with $\varphi \sqsubseteq L_{\beta} \mathfrak{a}$. As in
  the proof of Proposition~\ref{cor-hyperlog-uniform}, we have $L_{\beta}
  \mathfrak{a} \sqsubseteq \varphi$, whence $\varphi = L_{\beta}
  \mathfrak{a}$. We conclude that the cut equation (\ref{eq-poor-hyperlog}) is
  uniform.
\end{proof}

\subsection{Hyperexponentials}\label{subsection-hyperexponentials}

We have shown that $(\mathbf{No}, (L_{\omega^{\eta}})_{\eta < \nu})$ is a
hyperserial skeleton of force $(\nu, \mu)$. In order to prove that
$(\mathbf{No}, (L_{\omega^{\eta}})_{\eta < \nu})$ has force $(\nu, \nu)$, it
remains to prove that every $\beta$-truncated number $\varphi$ has
a~hyperexponential $E_{\beta} \varphi$. This is the purpose of this
subsection.

\begin{proposition}
  \label{prop-surjective-hyperlog}We have $L_{\beta}  \mathbf{Mo}_{\beta} =
  \mathbf{No}_{\succ, \beta}$, and $E_{\beta}$ has the following cut equation
  for all $\varphi \in \mathbf{No}_{\succ, \beta}$:
  
  \begin{equation}
    \label{eq-poor-hyperexp} 
    E_{\beta} \varphi = \left\{ E_{< \beta} \varphi, \: E_{< \beta}
    \left( \frac{1}{\varphi_R^{\mathbf{No}_{\succ, \beta}} - \varphi} \right),
    \mathcal{E}_{\beta}' E_{\beta} \varphi_L^{\smash{\mathbf{No}_{\succ,
    \beta}}} | \mathcal{E}_{\beta}' E_{\beta}
    \varphi_R^{\smash{\mathbf{No}_{\succ, \beta}}} \right\} .
  \end{equation}
\end{proposition}

\begin{proof}
  We prove the result by induction on $(\mathbf{No}_{\succ, \beta},
  \sqsubseteq)$. Let $\varphi \in \mathbf{No}_{\succ, \beta}$ such that
  $E_{\beta}$ is defined on $\varphi_{\sqsubset}^{\smash{\mathbf{No}_{\succ,
  \beta}}}$ with the given equation. We will first show that the number
  \begin{eqnarray}
    \mathfrak{a} & \assign & \left\{ E_{< \beta} \varphi, E_{< \beta} \left(
    \frac{1}{\varphi_R^{\smash{\mathbf{No}_{\succ, \beta}}} - \varphi}
    \right), \mathcal{E}_{\beta}' E_{\beta}
    \varphi_L^{\smash{\mathbf{No}_{\succ, \beta}}} | \mathcal{E}_{\beta}'
    E_{\beta} \varphi_R^{\smash{\mathbf{No}_{\succ, \beta}}} \right\} 
    \label{eq-a}
  \end{eqnarray}
  is well-defined. We will then prove that $L_{\beta} \mathfrak{a}= \varphi$.
  
  Let $\varphi' \in \varphi_L^{\smash{\mathbf{No}_{\succ, \beta}}}$ and
  $\varphi'' \in \varphi_R^{\smash{\mathbf{No}_{\succ, \beta}}}$. If $\varphi'
  \in (\varphi'')_L^{\smash{\mathbf{No}_{\succ, \beta}}}$, then $E_{\beta}
  \varphi'' >\mathcal{E}_{\beta}' E_{\beta} \varphi'$ by the definition
  of~$E_{\beta} \varphi''$. So $\mathcal{E}_{\beta}' E_{\beta} \varphi'
  <\mathcal{E}_{\beta}' E_{\beta} \varphi''$. Otherwise, we have $\varphi''
  \in (\varphi')_R^{\smash{\mathbf{No}_{\succ, \beta}}}$, whence
  $\mathcal{E}_{\beta}' E_{\beta} \varphi'' > E_{\beta} \varphi'$ by
  definition of $E_{\beta} \varphi'$, so $\mathcal{E}_{\beta}' E_{\beta}
  \varphi' <\mathcal{E}_{\beta}' E_{\beta} \varphi''$. So we always have
  \begin{eqnarray*}
    \mathcal{E}_{\beta}' E_{\beta} \varphi_L^{\smash{\mathbf{No}_{\succ,
    \beta}}} & < & \mathcal{E}_{\beta}' E_{\beta}
    \varphi_R^{\smash{\mathbf{No}_{\succ, \beta}}} .
  \end{eqnarray*}
  We also have $E_{< \beta} \varphi'' < E_{\beta} \varphi''$, so~$E_{< \beta}
  \varphi <\mathcal{E}_{\beta}' E_{\beta} \varphi''$. This proves that $E_{<
  \beta} \varphi <\mathcal{E}_{\beta}' E_{\beta}
  \varphi_R^{\smash{\mathbf{No}_{\succ, \beta}}}$. It remains to show that
  \begin{eqnarray*}
    E_{< \beta} \left( \frac{1}{\varphi_R^{\smash{\mathbf{No}_{\succ, \beta}}}
    - \varphi} \right) & < & \mathcal{E}_{\beta}' E_{\beta} \left(
    \varphi_R^{\smash{\mathbf{No}_{\succ, \beta}}} \right) .
  \end{eqnarray*}
  Note that $\varphi_R^{\smash{\mathbf{No}_{\succ, \beta}}}
  >\mathcal{L}_{\beta} [\varphi]$, so by the definition of
  $\mathcal{L}_{\beta} [\varphi]$, we have
  \begin{equation}
    \label{ineq-truncated} \widespacing{L_{\beta}^{\mathord{\uparrow} < \beta}
    \left( \frac{1}{\varphi_R^{\smash{\mathbf{No}_{\succ, \beta}}} - \varphi}
    \right) < \varphi < \varphi_R^{\smash{\mathbf{No}_{\succ, \beta}}} .}
  \end{equation}
  Hence $E_{< \beta} \left( \left( \varphi_R^{\smash{\mathbf{No}_{\succ,
  \beta}}} - \varphi \right)^{- 1} \right) < E_{\beta}
  \varphi_R^{\smash{\mathbf{No}_{\succ, \beta}}}$, which completes the proof
  that $\mathfrak{a}$ is well-defined.
  
  Let us now prove that~$L_{\beta} \mathfrak{a}= \varphi$. Note that
  $\mathfrak{a} \in \mathbf{Mo}_{\beta}$ by
  Proposition~\ref{prop-simplicity-condition}. First assume that $\mu$ is
  a~limit ordinal. Lemma~\ref{lem-limit-hypermon} yields $\langle E_{< \beta}
  \rangle \legeangle \mathcal{E}_{\beta}$, so we may write
  \begin{eqnarray*}
    \mathfrak{a} & = & \left\{ \mathfrak{d}_{\beta} (\varphi),
    \mathfrak{d}_{\beta} \left( \frac{1}{\varphi_R^{\smash{\mathbf{No}_{\succ,
    \beta}}} - \varphi} \right), E_{\beta}
    \varphi_L^{\smash{\mathbf{No}_{\succ, \beta}}} | E_{\beta}
    \varphi_R^{\smash{\mathbf{No}_{\succ, \beta}}}
    \right\}_{\mathbf{Mo}_{\beta}} .
  \end{eqnarray*}
  By (\ref{L-when-E}), for $b \in \mathbf{No}^{>, \succ}$ the classes that
  $\mathcal{L}_{\beta} [L_{\beta} b]$ and $L_{\beta} b \pm (L_{< \beta} b)^{-
  1}$ are mutually cofinal and coinitial. Moreover, we have $L_{\beta}
  E_{\beta} \psi = \psi$ for all $\psi \in
  \varphi_{\sqsubset}^{\smash{\mathbf{No}_{\succ, \beta}}}$, by our hypothesis
  on $\varphi$. Hence, Proposition~\ref{prop-uniform-hyperlog} and
  (\ref{L-when-E}) imply that $L_{\beta} \mathfrak{a}$ is the simplest positive infinite element in the cut
  \begin{eqnarray*}
    \left( \mathcal{L}_{\beta}
    [L_{\beta} \mathfrak{d}_{\beta} (\varphi)], \mathcal{L}_{\beta} \left[
    L_{\beta} \mathfrak{d}_{\beta} \left(
    \frac{1}{\varphi_R^{\smash{\mathbf{No}_{\succ, \beta}}} - \varphi} \right)
    \right], \mathcal{L}_{\beta} \left[ \varphi_L^{\smash{\mathbf{No}_{\succ,
    \beta}}} \right] | \varphi_R^{\smash{\mathbf{No}_{\succ, \beta}}} -
    \frac{1}{L_{< \beta} \mathfrak{a}} \right) .
  \end{eqnarray*}
  Note that $L_{\beta} \mathfrak{a} \in \left(
  \varphi_L^{\smash{\mathbf{No}_{\succ, \beta}}} |
  \varphi_R^{\smash{\mathbf{No}_{\succ, \beta}}} \right)_{\mathbf{No}_{\succ,
  \beta}}$, so $\varphi \sqsubseteq L_{\beta} \mathfrak{a}$. Note also that $L_{\beta}
  \mathfrak{d}_{\beta} (\varphi) \in \mathcal{L}_{\beta} [L_{\beta} \varphi] <
  \varphi$. We have
  \begin{eqnarray*}
    L_{\beta} \mathfrak{d}_{\beta} \left(
    \frac{1}{\varphi_R^{\smash{\mathbf{No}_{\succ, \beta}}} - \varphi} \right)
    & \in & L_{\beta} \mathcal{E}_{\beta}' \left[
    \frac{1}{\varphi_R^{\smash{\mathbf{No}_{\succ, \beta}}} - \varphi}
    \right],
  \end{eqnarray*}
  where
  \begin{align*}
    L_{\beta} \mathcal{E}_{\beta}' \left[
    \frac{1}{\varphi_R^{\smash{\mathbf{No}_{\succ, \beta}}} - \varphi} \right]
    & = L_{\beta} \mathcal{E}_{\beta}^{\ast} \left[
    \frac{1}{\varphi_R^{\smash{\mathbf{No}_{\succ, \beta}}} - \varphi} \right]
    \hspace*{\fill} && \text{(by Lemma~\ref{lem-limit-hypermon})}\\
    & \legeangle L_{\beta}^{\mathord{\uparrow} \mathord{<} \beta} \left(
    \frac{1}{\varphi_R^{\smash{\mathbf{No}_{\succ, \beta}}} - \varphi}
    \right)\\
    & < \varphi . \hspace*{\fill} && \text{(by (\ref{ineq-truncated}))}
  \end{align*}
  So $L_{\beta} \mathfrak{d}_{\beta} \left(
  \varphi_R^{\smash{\mathbf{No}_{\succ, \beta}}} - \varphi \right)^{- 1} <
  \varphi$. Since $\varphi \in \mathbf{No}_{\succ, \alpha}$, Proposition~\ref{prop-simplicity-condition} yields the inequality
  $\mathcal{L}_{\beta} \left[ \varphi_L^{\smash{\mathbf{No}_{\succ, \beta}}}
  \right] < \varphi$.
  Finally, we have $\mathfrak{a}> E_{< \beta} \left( \left(
  \varphi_R^{\smash{\mathbf{No}_{\succ, \beta}}} - \varphi \right)^{- 1}
  \right)$ by definition. So $\varphi_R^{\smash{\mathbf{No}_{\succ, \beta}}} - (L_{< \beta}
  \mathfrak{a})^{- 1} > \varphi$. This proves that $L_{\beta} \mathfrak{a}
  \sqsubseteq \varphi$, so $L_{\beta} \mathfrak{a}= \varphi$.
  
  Assume now that $\mu$ is a successor ordinal. For all $b \in \mathbf{No}^{>,
  \succ}$, the sets $E_{< \beta} \varphi$, $E_{< \beta} \mathfrak{d}_{\beta}
  (\varphi)$, and $E_{\beta_{/ \omega} \mathbb{N}} \mathfrak{d}_{\beta}
  (\varphi)$ are mutually cofinal. So we can rewrite (\ref{eq-a}) as
  \begin{eqnarray*}
    \mathfrak{a} & = & \left\{ E_{\beta_{/ \omega} \mathbb{N}}
    \mathfrak{d}_{\beta} (\varphi), E_{\beta_{/ \omega} \mathbb{N}}
    \mathfrak{d}_{\beta} \left( \frac{1}{\varphi_R^{\smash{\mathbf{No}_{\succ,
    \beta}}} - \varphi} \right), \mathcal{E}_{\beta}' E_{\beta}
    \varphi_L^{\smash{\mathbf{No}_{\succ, \beta}}} | \mathcal{E}_{\beta}'
    E_{\beta} \varphi_R^{\smash{\mathbf{No}_{\succ, \beta}}} \right\}\\
    & = & \left\{ E_{\beta_{/ \omega} \mathbb{N}} \mathfrak{d}_{\beta}
    (\varphi), E_{\beta_{/ \omega} \mathbb{N}} \mathfrak{d}_{\beta} \left(
    \frac{1}{\varphi_R^{\smash{\mathbf{No}_{\succ, \beta}}} - \varphi}
    \right), E_{\beta} \varphi_L^{\smash{\mathbf{No}_{\succ, \beta}}} |
    E_{\beta} \varphi_R^{\smash{\mathbf{No}_{\succ, \beta}}}
    \right\}_{\mathbf{Mo}_{\beta}} .
  \end{eqnarray*}
  As in the limit case, Proposition~\ref{prop-uniform-hyperlog} implies that $L_{\beta} \mathfrak{a}$ is the simplest element in the cut
  \[
    \left( \mathcal{L}_{\beta}
    \left[ L_{\beta}^{\mathord{\uparrow} \gamma} \mathfrak{d}_{\beta}
    (\varphi) \right], \mathcal{L}_{\beta} \left[
    L_{\beta}^{\mathord{\uparrow} \gamma} \mathfrak{d}_{\beta} \left(
    \frac{1}{\varphi_R^{\smash{\mathbf{No}_{\succ, \beta}}} - \varphi} \right)
    \right], \mathcal{L}_{\beta} \left[ \varphi_L^{\smash{\mathbf{No}_{\succ,
    \beta}}} \right] | \varphi_R^{\smash{\mathbf{No}_{\succ, \beta}}} -
    \frac{1}{L_{\gamma} \mathfrak{a}} \right)
  \]
  where $\gamma$ ranges in $\beta$.
  Let $\gamma < \beta$. There is an $n \in \mathbb{N}$ with $\gamma < \beta_{/
  \omega} n$. Since $L_{\beta} \varphi < \varphi - (n + 1)$, we have
  \[ \widespacing{\varphi > L_{\beta}^{\mathord{\mathord{\uparrow}} \beta_{/
     \omega}  (n + 1)} \mathfrak{d}_{\beta} (\varphi) \geqslant
     L_{\beta}^{\mathord{\mathord{\uparrow}} \gamma} \mathfrak{d}_{\beta}
     (\varphi) + 1.} \]
  In particular $\varphi >\mathcal{L}_{\beta} \left[
  L_{\beta}^{\mathord{\mathord{\uparrow}} \gamma} \mathfrak{d}_{\beta}
  (\varphi) \right]$. 
  
  We saw (\ref{ineq-truncated}) that
  $L_{\beta}^{\mathord{\mathord{\uparrow}} \gamma} \mathfrak{d}_{\beta} \left(
  \left( \varphi_R^{\smash{\mathbf{No}_{\succ, \beta}}} - \varphi \right)^{-
  1} \right) < \varphi$, so $\varphi$ lies strictly above the class $\mathcal{L}_{\beta} \left[
  L_{\beta}^{\mathord{\mathord{\uparrow}} \gamma} \mathfrak{d}_{\beta} \left(
  \left( \varphi_R^{\smash{\mathbf{No}_{\succ, \beta}}} - \varphi \right)^{-
  1} \right) \right]$. We obtain the inequalities
  \[ \widespacing{\mathcal{L}_{\beta} \left[
     \varphi_L^{\smash{\mathbf{No}_{\succ, \beta}}} \right] < \varphi <
     \varphi_R^{\smash{\mathbf{No}_{\succ, \beta}}} - (L_{< \beta}
     \mathfrak{a})^{- 1}} \]
  in a similar way as in the limit case.
  
  We conclude that $\varphi = L_{\beta} \mathfrak{a}$ holds in general. It
  follows by induction that the formula for $E_{\beta}$ is valid. In
  particular $L_{\beta} : \mathbf{Mo}_{\beta} \longrightarrow
  \mathbf{No}_{\succ, \beta}$ is surjective.
\end{proof}

With Proposition~\ref{prop-surjective-hyperlog}, we have completed the proof
of $\mathbf{I}_{1, \nu}$. By (\ref{exp-class-char}), we have
$\mathcal{E}_{\beta \omega} [a] =\mathcal{E}_{\beta \omega}' [a]$ for all $a
\in \mathbf{No}^{>, \succ}$. Given $a \in \mathbf{No}_{\succ, \beta}$, we also
deduce from (\ref{L-when-E}) that the set $a \pm (L_{< \beta} E_{\beta} a)^{-
1}$ is cofinal and coinitial in $\mathcal{L}_{\beta} [a]$. The convex
partition defined by $\mathcal{L}_{\beta}$ is thus thin. By
Proposition~\ref{prop-thin-substructure}, the class $\mathbf{No}_{\succ,
\beta}$ is a surreal substructure whowe parametrisation $\Xi \assign\Xi_{\mathbf{No}_{\succ, \beta}}$ uniform cut equation
\begin{equation}
  \forall a \in \mathbf{No}, \quad \Xi a = \{
  \mathbb{R}, \mathcal{L}_{\beta} [\Xi a_L]
  |\mathcal{L}_{\beta} [\Xi a_R] \}
  \label{eq-truncated}
\end{equation}
For $a \in \mathbf{No}$, we have $\mathcal{L}_{\beta}
[\Xi a] < \Xi a_R$,
so $\Xi a < \Xi a_R
- (L_{< \beta} E_{\beta} \Xi a)^{- 1}$. We deduce
that the following is equivalent to (\ref{eq-truncated}):
\begin{eqnarray}
  \Xi a & = & \left\{ \mathbb{R},
  \Xi a' + \frac{1}{L_{< \beta} E_{\beta}
  \Xi a'} |
  \Xi a'' - \frac{1}{L_{< \beta} E_{\beta}
  \Xi a} \right\}  \label{eq-truncated-2}
\end{eqnarray}
(where $a'$ and $a''$ range in $a_L$ and $a_R$ respectively).
\subsection{End of the inductive
proof}\label{subsection-end-of-the-inductive-proof}

We now prove $\mathbf{I}_{2, \nu}$, $\mathbf{I}_{3, \nu}$ and
Theorem~\ref{th-confluent-hyperserial-field}.

\begin{lemma}
  \label{lem-limit-asymptotics}If $\mu$ is a limit ordinal, then we have
  $E_{\beta} T_1 L_{\beta} > E_{< \beta}$ on $\mathbf{No}^{>, \succ}$.
\end{lemma}

\begin{proof}
  Let $a \in \mathbf{No}^{>, \succ}$. We have $\sharp_{\beta} (L_{\beta} a +
  1) > \sharp_{\beta} (L_{\beta} a)$, so (\ref{eq-hyperexp-proj}) yields
  \[ \widespacing{\mathfrak{d}_{\beta} (E_{\beta} (L_{\beta} a + 1)) =
     E_{\beta} (\sharp_{\beta} (L_{\beta} a + 1)) \succ E_{\beta}
     (\sharp_{\beta} (L_{\beta} a)) =\mathfrak{d}_{\beta} (a) .} \]
  We deduce that $E_{\beta} (L_{\beta} a + 1) >\mathcal{E}_{\beta} a$ so
  $E_{\beta} (L_{\beta} a + 1) > E_{< \beta} a$ by
  Lemma~\ref{lem-limit-hypermon}.
\end{proof}

\begin{proposition}
  \label{prop-I-3}For $r, s \in \mathbb{R}$ with $s > 1$ and $\gamma < \rho <
  \alpha$, we have $E_{\gamma} H_r L_{\gamma} < E_{\rho} H_s L_{\rho}$ on
  $\mathbf{No}^{>, \succ}$, i.e. $\mathbf{I}_{2, \nu}$ holds.
\end{proposition}

\begin{proof}
  Throughout this proof, we consider inequalities and equalities of functions
  on $\mathbf{No}^{>, \succ}$. Write $\gamma = \beta m + \iota$ and $\rho =
  \beta n + \theta$ where $m, n < \omega$ and $\iota, \theta < \beta$. We have
  \begin{eqnarray*}
    E_{\gamma} H_r L_{\gamma} & = & E_{\beta m} E_{\iota} H_r L_{\iota}
    L_{\beta m} \text{\quad and}\\
    E_{\rho} H_r L_{\rho} & = & E_{\beta n} E_{\theta} H_s L_{\theta} L_{\beta
    n} .
  \end{eqnarray*}
  If $m = n$, then $\iota < \theta$, so \hyperref[ind-ax2]{$\mathbf{I_{2,\mu}}$} yields $E_{\iota} H_r
  L_{\iota} < E_{\theta} H_s L_{\theta}$, whence $E_{\gamma} H_r L_{\gamma} <
  E_{\rho} H_s L_{\rho}$. Assume that~{$m < n$}. If $\mu_-$ is a successor
  ordinal, then there is $p < \omega$ with $\iota < \beta_{/ \omega} p$. By
  \hyperref[ind-ax2]{$\mathbf{I_{2,\mu}}$}, we have $E_{\theta} H_s L_{\theta} \geqslant H_s > T_p$. So
  $E_{\beta}  (E_{\theta} H_s L_{\theta}) L_{\beta} > E_{\beta} T_p L_{\beta}
  = E_{\beta_{/ \omega} p}$. We conclude by noting that $E_{\beta_{/ \omega}
  p} > E_{\iota} > E_{\iota} H_r L_{\iota}$. If $\mu_-$ is a~limit ordinal,
  then $E_{\theta} H_s L_{\theta} > T_1$ so $E_{\beta}  (E_{\theta} H_s
  L_{\theta}) L_{\beta} > E_{\iota} > E_{\iota} H_r L_{\iota}$ by
  Lemma~\ref{lem-limit-asymptotics}. It follows that for $k \in
  \mathbb{N}^{>}$, we have $E_{\beta (k + 1)} E_{\theta} H_s L_{\theta}
  L_{\beta (k + 1)} > E_{\beta k} E_{\iota} H_r L_{\iota} L_{\beta k}$. An
  easy induction on $k$ yields the~result.
\end{proof}

\begin{proposition}
  \label{prop-I3}$\mathbf{Mo}_{\alpha}'$ is the class of $L_{< \alpha}$-atomic
  numbers, i.e. $\mathbf{I}_{3, \nu}$ holds.
\end{proposition}

\begin{proof}
  Let $a \in \mathbf{No}^{>, \succ}$. By Corollary~\ref{cor-confluence}, the
  simplest element of $\mathcal{E}_{\alpha} [a]$ is $L_{< \alpha}$-atomic.
  Since $\mathcal{E}_{\alpha} [a] =\mathcal{E}_{\alpha}' [a]$, we deduce that
  $\mathbf{Mo}_{\alpha}' \subseteq \mathbf{Mo}_{\alpha}$.
  
  Conversely, given $\mathfrak{a} \in \mathbf{Mo}_{\alpha}$, we
  have~$\mathfrak{b} \assign \pi_{\mathcal{E}_{\alpha}'} (\mathfrak{a}) \in
  \mathbf{Mo}_{\alpha}' \subseteq \mathbf{Mo}_{\alpha}$. Now $\mathfrak{b} \in
  \mathcal{E}_{\alpha}' [\mathfrak{a}]$, so by $\mathbf{I}_{2, \nu}$, there
  are $r, s \in \mathbb{R}^{>}$ and $\gamma < \alpha$ with~$E_{\gamma}
  (rL_{\gamma} \mathfrak{a}) <\mathfrak{b}< E_{\gamma} (sL_{\gamma}
  \mathfrak{a})$. Hence, $L_{\gamma} \mathfrak{b} \asymp L_{\gamma}
  \mathfrak{a}$, $L_{\gamma} \mathfrak{b}= L_{\gamma} \mathfrak{a}$ and
  $\mathfrak{b}=\mathfrak{a}$. We conclude that $\mathfrak{a} \in
  \mathbf{Mo}_{\alpha}'$.
\end{proof}

In particular, the class $\mathbf{Mo}_{\alpha}$ is a surreal substructure. We
have proved $\mathbf{I}_{1, \nu}, \mathbf{I}_{2, \nu}$, and $\mathbf{I}_{3,
\nu}$, so we obtain the following by induction:

\begin{theorem}
  The field $(\mathbf{No}, (L_{\omega^{\eta}})_{\eta \in \mathbf{On}})$ is a
  confluent hyperserial skeleton of force $(\mathbf{On}, \mathbf{On})$.
\end{theorem}

Combining this with Propositions~\ref{th-skeleton-field}
and~\ref{th-bijective-hyperlogarithm}, we obtain
Theorem~\ref{th-confluent-hyperserial-field}. Let us finally show
that~$(\mathbf{No}, \circ)$ contains only one $L_{< \mathbf{On}}$-atomic
element.

\begin{proposition}
  \label{prop-omega-atomic}The number $\omega$ is the only $L_{<
  \mathbf{On}}$-atomic element in $\textbf{No}$. For all $a \in
  \mathbf{No}^{>, \succ}$, there is $\gamma \in \mathbf{On}$ with $L_{\gamma}
  a \asymp L_{\gamma} \omega$.
\end{proposition}

\begin{proof}
  The number $\omega$ lies in $\mathbf{Mo}_{\smash{\omega^{\mu}}}$ for all
  $\mu \in \mathbf{On}$, so it is $L_{< \mathbf{On}}$-atomic. For $\nu \in
  \mathbf{On}$, the number $E_{\omega^{\nu}} \omega = \left\{ E_{<
  \omega^{\nu}} \omega \: | \: \varnothing \right\}$ is an ordinal. As a sign
  sequence, the number $L_{\omega^{\nu}} \omega = \left\{ \varnothing \: | \:
  L_{< \omega^{\nu}} \omega \right\}_{\mathbf{No}^{>, \succ}}$ is $\omega$
  followed by a string containing only minuses {\cite[Lemma~2.6]{vdH:bm}}.
  Since the sequences $(E_{\omega^{\nu}} \omega)_{\nu \in \mathbf{On}}$ and
  $(L_{\omega^{\nu}} \omega)_{\nu \in \mathbf{On}}$ are strictly increasing
  and strictly decreasing respectively, the classes $\{ E_{\omega^{\nu}}
  \omega \suchthat \nu \in \mathbf{On} \}$ and $\{ L_{\omega^{\nu}} \omega
  \suchthat \nu \in \mathbf{On} \}$ are respectively cofinal and coinitial in
  $\mathbf{No}^{>, \succ} = \{ a \in \mathbf{No} \suchthat \omega \sqsubseteq
  a \}$. Thus for $a \in \mathbf{No}^{>, \succ}$, there is $\nu \in
  \mathbf{On}$ with $E_{\omega^{\nu}} \omega > a > L_{\omega^{\nu}} \omega$,
  whence~$L_{\omega^{\nu + 1}} \omega \asymp L_{\omega^{\nu + 1}} a$.
\end{proof}

\section{Remarkable identities}\label{section-useful-identities}

In this section, we give various identities regarding the function groups
introduced in Section~\ref{rem-subclasses}. In what follows, $\nu$ is a
non-zero ordinal and $\alpha \assign \omega^{\nu}$.

\subsection{Simplified cut equations for $L_{\alpha}$ and $E_{\alpha}$}

Given $\varphi \in \mathbf{No}^{>, \succ}$, let $E_{\vartriangleleft \alpha}
\assign \{ E_{(\alpha_{/ \omega}) n} \varphi \suchthat n \in \mathbb{N} \}$ if
$\nu$ is a successor ordinal and $E_{\vartriangleleft \alpha} \varphi \assign
\{ \varphi \}$ if $\nu$ is a~limit ordinal. In this subsection, we will derive
the following simplified cut equations for $L_{\alpha}$, for 
$\mathfrak{a} \in \mathbf{Mo}_{\alpha}$ and $E_{\alpha}$ for~$\varphi \in \mathbf{No}_{\succ, \alpha}$:
\begin{eqnarray}
  L_{\alpha}
  \mathfrak{a} & = & \{ L_{\alpha} \mathfrak{a}_L^{\mathbf{Mo}_{\alpha}}
  |L_{\alpha} \mathfrak{a}_R^{\mathbf{Mo}_{\alpha}}, L_{< \alpha} \mathfrak{a}
  \}_{\mathbf{No}_{\succ, \alpha}}  \label{eq-rich-hyperlog}\\
  & = & \left\{ \mathbb{R}, L_{\alpha} \mathfrak{a}' + \frac{1}{L_{< \alpha}
  \mathfrak{a}'}  |
  L_{\alpha} \mathfrak{a}'' - \frac{1}{L_{< \alpha} \mathfrak{a}''}, L_{<
  \alpha} \mathfrak{a} \right\}, 
  \label{eq-rich-hyperlog-2}\\
  E_{\alpha}
  \varphi & = & \{ E_{\vartriangleleft \alpha} \mathfrak{d}_{\alpha}
  (\varphi), E_{\alpha} \varphi_L^{\mathbf{No}_{\succ, \alpha}} | E_{\alpha}
  \varphi_R^{\mathbf{No}_{\succ, \alpha}} \}_{\mathbf{Mo}_{\alpha}} 
  \label{eq-rich-hyperexp}\\
  & = & \{ E_{< \alpha} \varphi, \mathcal{E}_{\alpha} E_{\alpha}
  \varphi_L^{\mathbf{No}_{\succ, \alpha}} | \mathcal{E}_{\alpha} E_{\alpha}
  \varphi_R^{\mathbf{No}_{\succ, \alpha}} \}  \label{eq-rich-hyperexp-2}
\end{eqnarray}
(where $\mathfrak{a}'$ and $\mathfrak{a}''$ range in
  $\mathfrak{a}_L^{\mathbf{Mo}_{\alpha}}$ and $\mathfrak{a}_R^{\mathbf{Mo}_{\alpha}}$ respectively).
For all $a \in \mathbf{No}^{>, \succ}$, the set $E_{\vartriangleleft \alpha}
\mathfrak{d}_{\alpha} (a)$ contains only $L_{< \alpha}$-atomic numbers, so
(\ref{eq-rich-hyperexp}) is indeed a cut equation of the form $\{ \rho |
\lambda \}_{\mathbf{Mo}_{\alpha}}$.

\begin{remark}
  The changes with respect to (\ref{eq-poor-hyperlog}) and
  (\ref{eq-poor-hyperexp}) lie in the occurrence of $\mathfrak{a}''$ instead
  of~$\mathfrak{a}$ in~(\ref{eq-rich-hyperlog-2}) and the (related) absence of
  the left option $E_{< \alpha} ((\varphi_R^{\mathbf{No}_{\succ, \alpha}} -
  \varphi)^{- 1})$ in (\ref{eq-rich-hyperexp-2}).
  So~(\ref{eq-rich-hyperlog-2}) and~(\ref{eq-rich-hyperexp-2}) give lighter
  sets of conditions than those in~(\ref{eq-poor-hyperlog})
  and~(\ref{eq-poor-hyperexp}) to define $L_{\alpha}$ and $E_{\alpha}$. This
  seemingly meager simplification will be crucial in further work. Indeed,
  combined with Proposition~\ref{prop-nearly-extensive}, this allows one to
  determine large classes of numbers $a, b$ with ${a \sqsubseteq b}
  \Longrightarrow E_{\alpha} a \sqsubseteq E_{\alpha} b$.
\end{remark}

First note that the cut equations (\ref{eq-rich-hyperlog}) and
(\ref{eq-rich-hyperexp}) if they hold are uniform (see
{\cite[Remark~1]{BvdHM:surhyp}}). Moreover, we claim that
(\ref{eq-rich-hyperlog}, \,\ref{eq-rich-hyperlog-2}) are equivalent and that
(\ref{eq-rich-hyperexp}, \,\ref{eq-rich-hyperexp-2}) are equivalent. Indeed,
recall that for a thin convex partition $\tmmathbf{\Pi}$ of a surreal
substructure $\mathbf{S}$ and any cut representation $(L, R)$ in
$\mathbf{Smp}_{\tmmathbf{\Pi}}$, one has
\begin{eqnarray*}
  \{ L|R \}_{\mathbf{Smp}_{\tmmathbf{\Pi}}} & = & \{ \tmmathbf{\Pi} [L] |
  \nobracket \tmmathbf{\Pi} [R] \} \nobracket_{\mathbf{S}} .
\end{eqnarray*}
For $\mathfrak{a}' \in \mathfrak{a}_L^{\mathbf{Mo}_{\alpha}}$ and
$\mathfrak{a}'' \in \mathfrak{a}_R^{\mathbf{Mo}_{\alpha}}$ the classes
$L_{\alpha} \mathfrak{a}' + (L_{< \alpha} \mathfrak{a}')^{- 1}$ and
$\mathcal{L}_{\alpha} [L_{\alpha} \mathfrak{a}']$ are mutually cofinal
by~(\ref{L-when-E}). Similarly, $L_{\alpha} \mathfrak{a}'' - (L_{< \alpha}
\mathfrak{a}'')^{- 1}$ and $\mathcal{L}_{\alpha} [L_{\alpha} \mathfrak{a}'']$
are mutually coinitial. By Lemma~\ref{lem-limit-hypermon}, the classes $E_{<
\alpha} \varphi$ and $\mathcal{E}_{\alpha} [E_{\vartriangleleft \alpha}
\mathfrak{d}_{\alpha} (\varphi)]$ are mutually cofinal. So it is enough to
prove that (\ref{eq-rich-hyperlog}) and~(\ref{eq-rich-hyperexp}) are valid cut
equations for $L_{\alpha}$ and $E_{\alpha}$ respectively.

\begin{lemma}
  If $\nu$ is a successor ordinal, then the identities
  {\tmem{(\ref{eq-rich-hyperlog})}} and {\tmem{(\ref{eq-rich-hyperexp})}}
  hold.
\end{lemma}

\begin{proof}
  Let $\mathfrak{a} \in \mathbf{Mo}_{\alpha}$ and set
  \begin{eqnarray*}
    \varphi & \assign & \{ L_{\alpha} \mathfrak{a}_L^{\mathbf{Mo}_{\alpha}}
    |L_{\alpha} \mathfrak{a}_R^{\mathbf{Mo}_{\alpha}}, L_{< \alpha}
    \mathfrak{a} \}_{\mathbf{No}_{\succ, \alpha}}\\
    & = & \left\{ \mathbb{R}, L_{\alpha} \mathfrak{a}' + \frac{1}{L_{<
    \alpha} \mathfrak{a}'} \suchthat \mathfrak{a}' \in
    \mathfrak{a}_L^{\mathbf{Mo}_{\alpha}} |L_{\alpha} \mathfrak{a}'' -
    \frac{1}{L_{< \alpha} \mathfrak{a}''}, L_{< \alpha} \mathfrak{a} \suchthat
    \mathfrak{a}'' \in \mathfrak{a}_R^{\mathbf{Mo}_{\alpha}} \right\} .
  \end{eqnarray*}
  We have $\mathcal{L}_{\alpha} [L_{\alpha}
  \mathfrak{a}_L^{\mathbf{Mo}_{\alpha}}] < \varphi < L_{< \alpha}
  \mathfrak{a}$ so in view of (\ref{eq-poor-hyperlog}), it is enough to prove
  that $\varphi < L_{\alpha} \mathfrak{a}_R^{\mathbf{Mo}_{\alpha}} - (L_{<
  \alpha} \mathfrak{a})^{- 1}$ to conclude that $\varphi = L_{\alpha}
  \mathfrak{a}$. Let $\mathfrak{a}'' \in
  \mathfrak{a}_R^{\mathbf{Mo}_{\alpha}}$. If $\mathfrak{a}'' \in
  \mathcal{E}_{\alpha}^{\ast} [\mathfrak{a}]$, then the inequality $\varphi <
  L_{\alpha} \mathfrak{a}''$ entails $\varphi <\mathcal{L}_{\alpha}
  [L_{\alpha} \mathfrak{a}'']$ whence $\varphi < L_{\alpha} \mathfrak{a}'' -
  (L_{< \alpha} \mathfrak{a}'')^{- 1}$ and $\varphi < L_{\alpha}
  \mathfrak{a}'' - (L_{< \alpha} \mathfrak{a})^{- 1}$. Otherwise, we have
  $\mathfrak{a}< L_{< \alpha} \mathfrak{a}''$, so $L_{\alpha} \mathfrak{a}<
  L_{\alpha} \mathfrak{a}'' - 2$, and $L_{\alpha} \mathfrak{a}'' - (L_{<
  \alpha} \mathfrak{a})^{- 1} > L_{\alpha} \mathfrak{a}+ 1$. It is enough to
  prove that~$L_{\alpha} \mathfrak{a}+ 1 \geqslant \varphi$. Recall that
  \begin{eqnarray*}
    L_{\alpha} \mathfrak{a}+ 1 = \left\{ L_{\alpha} \mathfrak{a},
    L_{\alpha} \mathfrak{a}' + \frac{1}{L_{< \alpha} \mathfrak{a}'} + 1
    |L_{\alpha} \mathfrak{a}'' - \frac{1}{L_{< \alpha}
    \mathfrak{a}} + 1, L_{< \alpha} \mathfrak{a} \right\}
  \end{eqnarray*}
  by (\ref{eq-uniform-sum}), where $\mathfrak{a}'$ and $\mathfrak{a}''$ range in $\mathfrak{a}_L^{\mathbf{Mo}_{\alpha}}$ and $\mathfrak{a}_R^{\mathbf{Mo}_{\alpha}}$ respectively. We see that $L_{\alpha} \mathfrak{a}' +
  \frac{1}{L_{< \alpha} \mathfrak{a}'} < L_{\alpha} \mathfrak{a}+ 1$ for all
  $\mathfrak{a}' \in \mathfrak{a}_L^{\mathbf{Mo}_{\alpha}}$. We have $1 -
  \frac{1}{L_{< \alpha} \mathfrak{a}} \succ \frac{1}{L_{< \alpha}
  \mathfrak{a}_R^{\mathbf{Mo}_{\alpha}}}$ so $L_{\alpha}
  \mathfrak{a}_R^{\mathbf{Mo}_{\alpha}} - \frac{1}{L_{< \alpha} \mathfrak{a}}
  + 1 > \varphi$. Thus $\varphi \leqslant L_{\alpha} \mathfrak{a}+ 1$. So
  (\ref{eq-rich-hyperlog}) holds.
  
  Now let $\psi \in \mathbf{No}_{\succ, \alpha}$ and set
  \begin{eqnarray*}
    \mathfrak{b} & \assign & \{ E_{\alpha_{/ \omega} \mathbb{N}}
    \mathfrak{d}_{\alpha} (\psi), E_{\alpha} \psi_L^{\mathbf{No}_{\succ,
    \alpha}} |E_{\alpha} \psi_R^{\mathbf{No}_{\succ, \alpha}}
    \}_{\mathbf{Mo}_{\alpha}} .
  \end{eqnarray*}
  By uniformity of (\ref{eq-rich-hyperlog}), we have
  \begin{eqnarray*}
    L_{\alpha} \mathfrak{b} & = & \{ L_{\alpha} E_{\alpha_{/ \omega}
    \mathbb{N}} \mathfrak{d}_{\alpha} (\psi), \psi_L^{\mathbf{No}_{\succ,
    \alpha}} | \psi_R^{\mathbf{No}_{\succ, \alpha}}, L_{< \alpha} \mathfrak{b}
    \}_{\mathbf{No}_{\succ, \alpha}},
  \end{eqnarray*}
  whence $L_{\alpha} \mathfrak{b} \sqsupseteq \{ \psi_L^{\mathbf{No}_{\succ,
  \alpha}} | \psi_R^{\mathbf{No}_{\succ, \alpha}} \}_{\mathbf{No}_{\succ,
  \alpha}} = \psi$. Conversely, $\mathfrak{b}> E_{\alpha_{/ \omega}
  \mathbb{N}} \mathfrak{d}_{\alpha} (\psi)$ and $\mathfrak{b}> E_{< \alpha}
  \psi$, so $\psi < L_{< \alpha} \mathfrak{b}$. We have $L_{\alpha}
  E_{\alpha_{/ \omega} \mathbb{N}} \mathfrak{d}_{\alpha} (\psi) = L_{\alpha}
  \mathfrak{d}_{\alpha} (\psi) +\mathbb{N}$. Since $L_{\alpha}
  \mathfrak{d}_{\alpha} (\psi) < L_{\alpha_{/ \omega}} \mathfrak{d}_{\alpha}
  (\psi) \prec \psi$, this yields~$L_{\alpha} E_{\alpha_{/ \omega} \mathbb{N}}
  \mathfrak{d}_{\alpha} (\psi) < \psi$. This proves that $\psi$ lies in the
  cut defining $L_{\alpha} \mathfrak{b}$. We conclude that $\psi = L_{\alpha}
  \mathfrak{b}$, hence (\ref{eq-rich-hyperexp}) holds.
\end{proof}

We now assume that $\nu$ is a limit ordinal. For $z \in \mathbf{No}$, define
\begin{eqnarray*}
  F (z) & \assign & \{ \mathfrak{d}_{\alpha} (\Xi_{\mathbf{No}_{\succ,
  \alpha}} z), F (z_L) |F (z_R) \}_{\mathbf{Mo}_{\alpha}}, \quad \text{and}\\
  \Xi z & \assign & \{ \mathbb{R}, \Xi z' + (L_{< \alpha} F (z'))^{- 1}
  \suchthat z' \in z_L | \Xi z_R - (L_{< \alpha} F (z))^{- 1} \} .
\end{eqnarray*}
\begin{lemma}
  For all $z \in \mathbf{No}$, we have
  \begin{eqnarray}
    F (z) & \text{is} & \text{defined}  \label{F-defined}\\
    \Xi z & \text{is} & \text{defined}  \label{xi-defined}\\
    \Xi z & = & \Xi_{\mathbf{No}_{\succ, \alpha}} z  \label{eq-alpha}\\
    F (z) & = & E_{\alpha} \Xi z  \label{F-small}
  \end{eqnarray}
\end{lemma}

\begin{proof}
  We prove the result by induction on $(\mathbf{No}, \sqsubseteq)$. Let $z \in
  \mathbf{No}$ be such that (\ref{F-defined}), (\ref{xi-defined}),
  (\ref{eq-alpha}) and~(\ref{F-small}) hold for all $y \in \mathbf{No}$ with
  $y \sqsubset z$.
  
  For $z'' \in z_R$ and $z' \in z_L$, we have $\mathfrak{d}_{\alpha}
  (\Xi_{\mathbf{No}_{\succ, \alpha}} z) \leqslant \mathfrak{d}_{\alpha}
  (\Xi_{\mathbf{No}_{\succ, \alpha}} z'') < F (z'')$. We have $F (z') < F
  (z'')$ by definition of $F (z'')$ if $z' \in (z'')_L$ and by definition of
  $F (z')$ if $z'' \in (z')_R$. This proves that $F (z)$ is defined.
  
  Let $z' \in z_L$ and $z'' \in z_R$. If $z' \in (z'')_L$, then we have $\Xi
  z'' > \Xi z' + (L_{< \alpha} F (z'))^{- 1}$ by definition of~$\Xi z''$.
  Since $F (z') < F (z)$ and $F (z), F (z') \in \mathbf{Mo}_{\alpha}$, we have
  $L_{\gamma} F (z') \prec L_{\gamma} F (z)$ for all $\gamma < \alpha$. 
  
  We deduce that $\Xi z'' - (L_{< \alpha} F (z))^{- 1} > \Xi z' + (L_{< \alpha} F
  (z'))^{- 1}$. If $z'' \in (z')_L$, then $\Xi z' < \Xi z'' - (L_{< \alpha} F
  (z'))^{- 1}$ by definition of $\Xi z'$. Since $F (z') < F (z)$, we obtain
  $\Xi z'' - (L_{< \alpha} F (z))^{- 1} > \Xi z' + (L_{< \alpha} F (z'))^{-
  1}$. This proves that $\Xi z$ is defined.
  
  Since (\ref{eq-alpha}) and (\ref{F-small}) hold on $z_{\sqsubset}$, we have
  \begin{eqnarray*}
    \Xi z=\{ \mathbb{R}, \Xi_{\mathbf{No}_{\succ, \alpha}} z' + \frac{1}{L_{<
    \alpha} E_{\alpha} \Xi_{\mathbf{No}_{\succ, \alpha}} z'} | \Xi_{\mathbf{No}_{\succ, \alpha}} z'' - \frac{1}{L_{< \alpha}
    E_{\alpha} \Xi_{\mathbf{No}_{\succ, \alpha}} z} \}
  \end{eqnarray*}
  (for $z' \in z_L$ and $z'' \in z_R$).
  By (\ref{eq-truncated-2}), this yields $\Xi z = \Xi_{\mathbf{No}_{\succ,
  \alpha}} z$, so (\ref{eq-alpha}) holds for $z$.
  
  From (\ref{eq-alpha}), we get $\mathfrak{d}_{\alpha}
  (\Xi_{\mathbf{No}_{\succ, \alpha}} z) =\mathfrak{d}_{\alpha} (\Xi z)$. By
  Proposition~\ref{prop-uniform-hyperlog} and our assumption that
  (\ref{F-small}) holds on $z_{\sqsubset}$, we have
  \begin{eqnarray*}
    L_{\alpha} F (z) & = & \{ \mathbb{R}, \mathcal{L}_{\alpha} [L_{\alpha}
    \mathfrak{d}_{\alpha} (\Xi z)], \mathcal{L}_{\alpha} [L_{\alpha} F (z_L)]
    |L_{\alpha} F (z_R) - \frac{1}{L_{< \alpha} F (z)}, L_{< \alpha} F (z) \}\\
    & = &\{ \mathbb{R}, \mathcal{L}_{\alpha} [L_{\alpha}
    \mathfrak{d}_{\alpha} (\Xi z)], \mathcal{L}_{\alpha} [\Xi z_L] | \Xi z_R -
    \frac{1}{L_{< \alpha} F (z)}, L_{< \alpha} F (z) \} .
  \end{eqnarray*}
  Recall that $\Xi z = \{ \mathbb{R}, \mathcal{L}_{\alpha} [\Xi z_L] | \Xi z_R
  - (L_{< \alpha} F (z))^{- 1} \}$. Therefore it suffices to show that $\Xi z$
  lies in the cut $(\mathcal{L}_{\alpha} [L_{\alpha} \mathfrak{d}_{\alpha}
  (\Xi z)] |L_{< \alpha} F (z))$ to conclude that $L_{\alpha} F (z) = \Xi z$
  and thus that $F (z) = E_{\alpha} \Xi z$. Now $L_{\alpha}
  \mathfrak{d}_{\alpha} (\Xi z) <\mathcal{E}_{\alpha}^{\ast} [\Xi z]$ so
  $L_{\alpha} \mathfrak{d}_{\alpha} (\Xi z) \prec \Xi z$ and
  $\mathcal{L}_{\alpha} [L_{\alpha} \mathfrak{d}_{\alpha} (\Xi z)] < \Xi z$.
  We have $F (z) >\mathfrak{d}_{\alpha} (\Xi z)$, where $F (z) \in
  \mathbf{Mo}_{\alpha}$. Since $\nu$ is a limit ordinal,
  Lemma~\ref{lem-limit-hypermon} implies that $F (z) > E_{< \alpha} \Xi z$, so
  $\Xi z < L_{< \alpha} F (z)$. This completes the proof that~$F (z) =
  E_{\alpha} \Xi z$.
\end{proof}

\begin{corollary}
  \label{cor-true-def}The identities {\tmem{(\ref{eq-rich-hyperlog})}},
  {\tmem{(\ref{eq-rich-hyperlog-2}),}} {\tmem{(\ref{eq-rich-hyperexp})}}, and
  {\tmem{(\ref{eq-rich-hyperexp-2})}} all hold.
\end{corollary}

\begin{proof}
  It is enough to prove (\ref{eq-rich-hyperlog}) and (\ref{eq-rich-hyperexp}).
  The identity (\ref{eq-rich-hyperexp}) follows from (\ref{eq-alpha})
  and~(\ref{F-small}). In order to obtain (\ref{eq-rich-hyperlog}), we
  consider $\mathfrak{a} \in \mathbf{Mo}_{\alpha}$, set $\psi \assign \{
  L_{\alpha} \mathfrak{a}_L^{\mathbf{Mo}_{\alpha}} |L_{\alpha}
  \mathfrak{a}_R^{\mathbf{Mo}_{\alpha}}, L_{< \alpha} \mathfrak{a}
  \}_{\mathbf{No}_{\succ, \alpha}}$, and we show that~$\mathfrak{a}=
  E_{\alpha} \psi$. Since (\ref{eq-rich-hyperexp}) is uniform, we have
  \begin{eqnarray*}
    E_{\alpha} \psi & = & \{ \mathfrak{d}_{\alpha} (\psi), E_{\alpha}
    L_{\alpha} \mathfrak{a}_L^{\mathbf{Mo}_{\alpha}} |E_{\alpha} L_{\alpha}
    \mathfrak{a}_R^{\mathbf{Mo}_{\alpha}}, E_{\alpha} L_{< \alpha}
    \mathfrak{a} \}_{\mathbf{Mo}_{\alpha}}\\
    & = & \{ \mathfrak{d}_{\alpha} (\psi),
    \mathfrak{a}_L^{\mathbf{Mo}_{\alpha}}
    |\mathfrak{a}_R^{\mathbf{Mo}_{\alpha}}, E_{\alpha} L_{< \alpha}
    \mathfrak{a} \}_{\mathbf{Mo}_{\alpha}} .
  \end{eqnarray*}
  We have $\mathfrak{d}_{\alpha} (\psi) <\mathfrak{a}$ because $\psi < L_{<
  \alpha} \mathfrak{a}$, and $E_{\alpha} L_{< \alpha}
  \mathfrak{a}>\mathfrak{a}$ because $E_{\alpha} > E_{< \alpha}$ on
  $\mathbf{No}^{>, \succ}$. Since $\mathfrak{a}= \{
  \mathfrak{a}_L^{\mathbf{Mo}_{\alpha}} |\mathfrak{a}_R^{\mathbf{Mo}_{\alpha}}
  \}_{\mathbf{Mo}_{\alpha}}$, we deduce that $E_{\alpha} \psi =\mathfrak{a}$.
\end{proof}

\begin{remark}
  The simplified cut equations for $E_{\alpha}, L_{\alpha}$ can be viewed as
  alternative definitions for those functions, since they hold inductively on
  their domain of definition. It is unclear how to develop our theory directly
  upon these alternative definitions. In particular, does there exists a
  direct way to see that the cut equation (\ref{eq-rich-hyperlog-2}) is
  warranted, and that the corresponding function satisfies \hyperref[regularity]{$\mathbf{R_{\mu}}$}
  and \hyperref[monotonicity]{$\mathbf{M_{\mu}}$}?
\end{remark}

\subsection{Identities involving $\mathbf{Tr}_{\alpha}$ and
$\mathbf{Tr}_{\alpha}^{\ast}$.}

\begin{proposition}
  \label{identification-prop}Defining $\mathbf{Tr}_{\alpha} \assign
  \mathbf{Smp}_{\mathcal{L}_{\alpha}'}$ as in Section~\ref{rem-subclasses}, we
  have $\mathbf{Tr}_{\alpha} = \mathbf{No}_{\succ, \alpha}$.
\end{proposition}

\begin{proof}
  Let $\varphi \in \mathbf{No}_{\succ, \alpha}$. We have $E_{\alpha}
  \mathcal{L}_{\alpha} [\varphi] =\mathcal{E}_{\alpha} [E_{\alpha} \varphi]$
  by {\cite[Proposition~7.22]{BvdHK:hyp}}. Recall that $\mathcal{E}_{\alpha}
  [a] =\mathcal{E}_{\alpha}' [a]$ for all $a \in \mathbf{No}^{>, \succ}$. Now
  $\mathcal{E}_{\alpha}' \circ E_{\alpha} = E_{\alpha} \circ
  \mathcal{L}_{\alpha}'$ by definition of $\mathcal{L}_{\alpha}'$, so
  $E_{\alpha} \mathcal{L}_{\alpha} [\varphi] = E_{\alpha}
  \mathcal{L}_{\alpha}' [\varphi]$ and~$\mathcal{L}_{\alpha} [\varphi]
  =\mathcal{L}_{\alpha}' [\varphi]$. By definition of $\mathbf{Tr}_{\alpha}$,
  we conclude that $\mathbf{Tr}_{\alpha} = \mathbf{Smp}_{\mathcal{L}_{\alpha}}
  = \mathbf{No}_{\succ, \alpha}$.
\end{proof}

Assume that $\nu$ is a successor ordinal. Then we have $\mathbf{No}_{\succ,
\alpha} = \mathbf{No}_{\succ, \alpha} +\mathbb{R}$ by
(\ref{eq-succ-truncated}), so the functions $T_r \Xi_{\mathbf{No}_{\succ,
\alpha}}$ and $\Xi_{\mathbf{No}_{\succ, \alpha}} T_r$ are both strictly
increasing bijections from $\mathbf{No}$ onto $\mathbf{No}_{\succ, \alpha}$.

\begin{lemma}
  \label{lem-Tr'-translation-invariant}Assume that $\nu$ is a successor
  ordinal. Then for $r \in \mathbb{R}$, we have $T_r \Xi_{\mathbf{No}_{\succ,
  \alpha}} = \Xi_{\mathbf{No}_{\succ, \alpha}} T_r$ on $\mathbf{No}$.
\end{lemma}

\begin{proof}
  Let us abbreviate $\Xi \assign \Xi_{\mathbf{No}_{\succ, \alpha}}$. We prove
  the lemma by induction on $(\mathbf{No}, \sqsubseteq) \times (\mathbb{R},
  \sqsubseteq)$. Let $(z, r) \in \mathbf{No} \times \mathbb{R}$ with
  \begin{eqnarray*}
    \Xi y + s & = & \Xi (y + s)
  \end{eqnarray*}
  whenever $(y, s) \in \mathbf{No} \times \mathbb{R}$ is strictly simpler than
  $(z, r)$. We let $z', z'', r', r''$ denote generic elements of $z_L, z_R,
  r_L, r_R$ and we note that $r', r'' \in \mathbb{R}$. By
  (\ref{eq-truncated}), we have
  \begin{eqnarray*}
    \Xi (z + r) & = & \left\{\mathbb{R}, \Xi (z' + r) + \frac{1}{L_{< \alpha} E_{\alpha}
    \Xi (z' + r)},  \Xi (z + r') + \frac{1}{L_{< \alpha}
    E_{\alpha} \Xi (z + r')} \right|\\
    &  & \left. \Xi (z + r'') - \frac{1}{L_{< \alpha}
    E_{\alpha} \Xi (z + r'')}, \Xi (z'' + r) - \frac{1}{L_{<
    \alpha} E_{\alpha} \Xi (z'' + r)} \right\}\\
    & = & \left\{\mathbb{R}, T_r \Xi z' + \frac{1}{L_{< \alpha} E_{\alpha} T_r \Xi z'},
    \hspace{1.2em} T_{r'} \Xi z + \frac{1}{L_{< \alpha} E_{\alpha} T_{r'} \Xi
    z} \right|\\
    &  & \left.  T_{r''} \Xi z - \frac{1}{L_{< \alpha}
    E_{\alpha} T_{r''} \Xi z}, T_r \Xi z'' - \frac{1}{L_{<
    \alpha} E_{\alpha} T_r \Xi z''} \right\}.
  \end{eqnarray*}
  Recall that $\nu$ is a successor ordinal. Since (\ref{eq-functional-total})
  holds for all $a \in \mathbf{No}^{>, \succ}$, the sets $L_{< \alpha}
  E_{\alpha} \mathcal{T}a$ and $L_{< \alpha} E_{\alpha} a$ are mutually
  cofinal and coinitial. Moreover $T_s (z + b) = T_s z + b$ for all $s \in
  \mathbb{R}$ and $b \in \mathbf{No}$, so
  \begin{eqnarray*}
    \Xi (z + r) & = & \left\{\mathbb{R}, T_r \left( \Xi z' + \frac{1}{L_{< \alpha}
    E_{\alpha} \Xi z'} \right), \; T_{r'} \left( \Xi z + \frac{1}{L_{< \alpha}
    E_{\alpha} \Xi z} \right) \right|\\
    &  & \left. T_{r''} \left( \Xi z - \frac{1}{L_{< \alpha}
    E_{\alpha} \Xi z} \right), \; T_r \left( \Xi z'' - \frac{1}{L_{< \alpha}
    E_{\alpha} \Xi z''} \right) \right\}.
  \end{eqnarray*}
  By (\ref{eq-uniform-sum}), the number $T_r \Xi z$ is simplest in the cut
 \[\left(\mathbb{R}, T_r \left( \Xi z' + \frac{1}{L_{< \alpha}
    E_{\alpha} \Xi z'} \right), \; T_{r'} \Xi z | T_{r''} \Xi z, \; T_r \left(
    \Xi z'' - \frac{1}{L_{< \alpha} E_{\alpha} \Xi z''} \right)
    \right).\]
  The numbers $T_r \Xi z, T_{r'} \Xi z$ and $T_{r''} \Xi z$ are
  $\alpha$-truncated so $T_r \Xi z$ lies in the cut
  \[\mathbb{R}, \left( \bigcup_{r'} T_{r'} \left( \Xi z + \frac{1}{L_{< \alpha}
     E_{\alpha} \Xi z} \right) | \bigcup_{r''} T_{r''} \left( \Xi z -
     \frac{1}{L_{< \alpha} E_{\alpha} \Xi z} \right) \right) . \]
  We deduce that $T_r \Xi z = \Xi T_r z$. The result follows by induction.
\end{proof}

\begin{lemma}
  \label{lem-Tr'-successor-identity}If $\nu$ is a successor ordinal, then we
  have $\mathcal{T} \legeangle \mathcal{L}_{\alpha}^{\ast}$ on
  $\mathbf{No}^{>, \succ}$. Consequently, {$\mathbf{Tr}^{\ast}_{\alpha} =
  \mathbf{No}_{\succ}^{>}$}.
\end{lemma}

\begin{proof}
  The set $E_{< \alpha}$ is pointwise cofinal in
  $\mathcal{E}^{\ast}_{\alpha}$. So $L_{\alpha} E_{< \alpha} E_{\alpha}$ is
  pointwise cofinal in~$\mathcal{L}^{\ast}_{\alpha}$. For $\gamma < \alpha$,
  there is $n \in \mathbb{N}$ such that $\gamma \leqslant \alpha_{/ \omega}
  n$. We have
  \[ \widespacing{L_{\alpha} E_{\gamma} E_{\alpha} \leqslant L_{\alpha}
     E_{\alpha_{/ \omega} n} E_{\alpha} = (L_{\alpha} E_{\alpha_{/ \omega}}
     E_{\alpha})^{\circ n} = (L_{\alpha} E_{\alpha} T_1)^{\circ n} = T^{\circ
     n}_1 = T_n \in \mathcal{T} \hspace{-1em} .} \]
  We deduce that $\mathcal{T} \legeangle \mathcal{L}_{\alpha}^{\ast}$ on
  $\mathbf{No}^{>, \succ}$, whence $\mathbf{Tr}^{\ast}_{\alpha} =
  \mathbf{Smp}_{\mathcal{T}} = \mathbf{No}_{\succ}^{>}$.
\end{proof}

\subsection{Identities involving $\mathbf{Mo}_{\alpha}$ and
$\mathbf{Mo}^{\ast}_{\alpha}$.}

\begin{lemma}
  \label{lem-hyperexp-hypermonomial-identity}If $\nu$ is a successor ordinal,
  then for $z \in \mathbf{No}$ we have
  \begin{eqnarray*}
    \Xi_{\mathbf{Mo}_{\alpha}} (z - 1) & = & L_{\alpha_{/ \omega}}
    \Xi_{\mathbf{Mo}_{\alpha}} z.
  \end{eqnarray*}
\end{lemma}

\begin{proof}
  This can be seen as a converse to the proof of the identity
  (\ref{eq-weaker-hyperlog}). We proceed by induction on $(\mathbf{No},
  \sqsubseteq)$. Let $z$ be such that the relation holds on $z_{\sqsubset}$.
  By (\ref{eq-weaker-hyperlog}), we have
  \begin{eqnarray*}
    L_{\alpha_{/ \omega}} \Xi_{\mathbf{Mo}_{\alpha}} z & = & \{ L_{\alpha_{/
    \omega}}  (\Xi_{\mathbf{Mo}_{\alpha}} z)_L^{\mathbf{Mo}_{\alpha}}
    |L_{\alpha_{/ \omega}}  (\Xi_{\mathbf{Mo}_{\alpha}}
    z)_R^{\mathbf{Mo}_{\alpha}}, \Xi_{\mathbf{Mo}_{\alpha}} z
    \}_{\mathbf{Mo}_{\alpha}}\\
    & = & \{ L_{\alpha_{/ \omega}} \Xi_{\mathbf{Mo}_{\alpha}} z_L
    |L_{\alpha_{/ \omega}} \Xi_{\mathbf{Mo}_{\alpha}} z_R,
    \Xi_{\mathbf{Mo}_{\alpha}} z \}_{\mathbf{Mo}_{\alpha}}\\
    & = & \{ \Xi_{\mathbf{Mo}_{\alpha}} (z_L - 1) |
    \Xi_{\mathbf{Mo}_{\alpha}} (z_R - 1), \Xi_{\mathbf{Mo}_{\alpha}} z
    \}_{\mathbf{Mo}_{\alpha}} \\
    & = & \Xi_{\mathbf{Mo}_{\alpha}}  \{ z_L - 1| z_R - 1, z \}\\
    & = & \Xi_{\mathbf{Mo}_{\alpha}} (z - 1) \text{\quad by
    (\ref{eq-uniform-sum})} .
  \end{eqnarray*}
  We conclude by induction.
\end{proof}

Noting that $E_{\alpha_{/ \omega}} = E_{\alpha} T_1 L_{\alpha}$ on
$\mathbf{No}^{>, \succ}$, the previous relation further generalizes as
follows.

\begin{proposition}
  \label{prop-hyperexp-hypermon-identity}Assume that $\nu$ is a successor
  ordinal and let $r \in \mathbb{R}$. Then
  \begin{eqnarray}
    \Xi_{\mathbf{Mo}_{\alpha}} T_r & = & E_{\alpha} T_r L_{\alpha}
    \Xi_{\mathbf{Mo}_{\alpha}}  \label{eq-real-iterate}
  \end{eqnarray}
\end{proposition}

\begin{proof}
  We proceed by induction. Let $(z, r) \in \mathbf{No} \times \mathbb{R}$ be
  such that
  \begin{eqnarray*}
    \Xi_{\mathbf{Mo}_{\alpha}} T_s y & = & E_{\alpha} T_s L_{\alpha}
    \Xi_{\mathbf{Mo}_{\alpha}} y
  \end{eqnarray*}
  for all strictly simpler $(y, s) \in \mathbf{No} \times \mathbb{R}$ with
  respect to the product order $\mathord{\sqsubseteq} \times
  \mathord{\sqsubseteq}$. For $s \in \mathbb{R}$, let $\phi_s$ be the function
  $b \longmapsto E_{\alpha} T_s L_{\alpha} b$ on $\mathbf{No}^{>, \succ}$ and
  let $\mathfrak{a} \assign \Xi_{\mathbf{Mo}_{\alpha}} z$. By
  (\ref{eq-uniform-sum}) and (\ref{eq-uniform-G}), we have
  \begin{eqnarray*}
    \Xi_{\mathbf{Mo}_{\alpha}} (z + r) & = & \left\{\mathbb{R},
    \mathcal{E}_{\alpha} \Xi_{\mathbf{Mo}_{\alpha}} (z_L + r),
    \mathcal{E}_{\alpha} \Xi_{\mathbf{Mo}_{\alpha}} (z + r_L)
    \right|\\
    &  & \quad \left. \mathcal{E}_{\alpha} \Xi_{\mathbf{Mo}_{\alpha}} (z_R + r),
    \mathcal{E}_{\alpha} \Xi_{\mathbf{Mo}_{\alpha}} (z + r_R) \right\}\\
    & = & \{ \mathbb{R}, \mathcal{E}_{\alpha} \phi_r
    (\mathfrak{a}_L^{\mathbf{Mo}_{\alpha}}), \mathcal{E}_{\alpha} \phi_{r_L}
    (\mathfrak{a}) |\mathcal{E}_{\alpha} \phi_r
    (\mathfrak{a}_R^{\mathbf{Mo}_{\alpha}}), \mathcal{E}_{\alpha} \phi_{r_R}
    (\mathfrak{a}) \} .
  \end{eqnarray*}
  By (\ref{eq-rich-hyperlog}), Lemma~\ref{lem-Tr'-translation-invariant} and
  (\ref{eq-uniform-sum}), we have:
  \begin{eqnarray*}
    T_r L_{\alpha} \mathfrak{a} & = & \{ T_r L_{\alpha}
    \mathfrak{a}^{\mathbf{Mo}_{\alpha}}_L, T_{r_L} L_{\alpha}
    \mathfrak{a}|T_{r_R} L_{\alpha} \mathfrak{a}, T_r L_{\alpha}
    \mathfrak{a}^{\mathbf{Mo}_{\alpha}}_R, L_{< \alpha} \mathfrak{a}
    \}_{\mathbf{Tr}_{\alpha}} .
  \end{eqnarray*}
  We deduce that
  \begin{eqnarray*}
    \phi_r (\mathfrak{a}) & = & \left\{ E_{< \alpha} T_r L_{\alpha} \mathfrak{a},
    \mathcal{E}_{\alpha} \phi_r (\mathfrak{a}^{\mathbf{Mo}_{\alpha}}_L),
    \mathcal{E}_{\alpha} \phi_{r_L} (\mathfrak{a}) \right| \\ & & \left. \mathcal{E}_{\alpha}
    \phi_{r_R} (\mathfrak{a}), \mathcal{E}_{\alpha} \phi_r
    (\mathfrak{a}_R^{\mathbf{Mo}_{\alpha}}), \mathcal{E}_{\alpha} E_{\alpha}
    L_{< \alpha} \mathfrak{a} \right\}\\
    & = & \left\{ E_{< \alpha} L_{\alpha} \mathfrak{a}, \mathcal{E}_{\alpha}
    \phi_r (\mathfrak{a}^{\mathbf{Mo}_{\alpha}}_L), \mathcal{E}_{\alpha}
    \phi_{r_L} (\mathfrak{a}) \right| \\ & & \left. \mathcal{E}_{\alpha} \phi_{r_R} (\mathfrak{a}),
    \mathcal{E}_{\alpha} \phi_r (\mathfrak{a}_R^{\mathbf{Mo}_{\alpha}}),
    E_{\alpha} L_{< \alpha} \mathfrak{a} \right\} .
  \end{eqnarray*}
  It is enough to prove that $E_{< \alpha} L_{\alpha} \mathfrak{a}<
  \Xi_{\mathbf{Mo}_{\alpha}} (z + r) < E_{\alpha} L_{< \alpha} \mathfrak{a}$
  to conclude that $\phi_r (\mathfrak{a}) = \Xi_{\mathbf{Mo}_{\alpha}} (z +
  r)$. Towards this, fix an $n \in \mathbb{N}$ with $- n \leqslant r \leqslant
  n$. Lemma~\ref{lem-hyperexp-hypermonomial-identity} yields
  \begin{eqnarray*}
    \Xi_{\mathbf{Mo}_{\alpha}} (z + r) & \leqslant &
    \Xi_{\mathbf{Mo}_{\alpha}} (z + n) \hspace{1.2em} = \hspace{1.2em}
    E_{\alpha_{/ \omega^n}} \mathfrak{a} \hspace{1.2em} < \hspace{1.2em}
    E_{\alpha} L_{< \alpha} \mathfrak{a}\\
    \Xi_{\mathbf{Mo}_{\alpha}} (z + r) & \geqslant &
    \Xi_{\mathbf{Mo}_{\alpha}} (z - n) \hspace{1.2em} = \hspace{1.2em}
    L_{\alpha_{/ \omega^n}} \mathfrak{a} \hspace{1.2em} > \hspace{1.2em} E_{<
    \alpha} L_{\alpha} \mathfrak{a}.
  \end{eqnarray*}
  We conclude by induction that (\ref{eq-real-iterate}) holds.
\end{proof}

\begin{remark}
  For $r, s \in \mathbb{R}$, we have $\phi_{r + s} = \phi_r \circ \phi_s$, and
  $\phi_1 = E_{\alpha_{/ \omega}}$. Therefore we can see $(\phi_r)_{r \in
  \mathbb{R}}$ as a system of fractional and real iterates of the
  hyperexponential function $E_{\alpha_{/ \omega}}$ on~$\mathbf{No}^{>,
  \succ}$. The previous proposition shows that the action of those iterates on
  $L_{< \alpha}$-atomic numbers reduces to translations, modulo the
  parametrization $\Xi_{\mathbf{Mo}_{\alpha}}$. In particular, one can compute
  the functional square root of $\exp$ on $\mathbf{Mo}_{\omega}$ in terms of
  sign sequences using the material from {\cite{Bag:signseq}}.
\end{remark}

{\color[HTML]{800000}\tmcolor{black}{\begin{proposition}
  \label{prop-hyperexp-J-right-factor}If $\nu$ is a successor ordinal, then
  $\mathbf{Mo}^{\ast}_{\alpha} = \mathbf{Mo}_{\alpha} \mathbin{\Yleft}
  \mathbf{No}_{\succ}$.
\end{proposition}}

\tmcolor{black}{\begin{proof}
  For $\theta \in \mathbf{No}_{\succ}$, we have $\theta_L +\mathbb{N}< \theta
  < \theta_R -\mathbb{N}$. By Lemma~\ref{lem-hyperexp-hypermonomial-identity},
  it follows that $E_{\alpha_{/ \omega} \mathbb{N}} \Xi_{\mathbf{Mo}_{\alpha}}
  \theta_L < \Xi_{\mathbf{Mo}_{\alpha}} \theta < L_{\alpha_{/ \omega}
  \mathbb{N}} \Xi_{\mathbf{Mo}_{\alpha}} \theta_R$. This implies that
  $\mathcal{E}^{\ast}_{\alpha} \Xi_{\mathbf{Mo}_{\alpha}} \theta_L <
  \Xi_{\mathbf{Mo}_{\alpha}} \theta <\mathcal{E}^{\ast}_{\alpha}
  \Xi_{\mathbf{Mo}_{\alpha}} \theta_R$, so $\Xi_{\mathbf{Mo}_{\alpha}} \theta$
  is~$\mathcal{E}^{\ast}_{\alpha}$-simple.
  
  Conversely, consider $\theta \in \mathbf{No}^{>, \succ}$ such that
  $\Xi_{\mathbf{Mo}_{\alpha}} \theta$ is $\mathcal{E}^{\ast}_{\alpha}$-simple.
  We have $\Xi_{\mathbf{Mo}_{\alpha}} \theta_L \subseteq
  (\Xi_{\mathbf{Mo}_{\alpha}} \theta)_L$ and $\Xi_{\mathbf{Mo}_{\alpha}}
  \theta_R \subseteq (\Xi_{\mathbf{Mo}_{\alpha}} \theta)_R$, whence
  $E_{\alpha_{/ \omega} \mathbb{N}} \Xi_{\mathbf{Mo}_{\alpha}} \theta_L <
  \Xi_{\mathbf{Mo}_{\alpha}} \theta < L_{\alpha_{/ \omega} \mathbb{N}}
  \Xi_{\mathbf{Mo}_{\alpha}} \theta_R$. We obtain $\theta_L +\mathbb{N}<
  \theta < \theta_R -\mathbb{N}$,  which proves that $\theta \in
  \mathbf{No}_{\succ}$.
\end{proof}}}

\begin{proposition}
  \label{prop-hyperexp-star-relation}We have $E_{\alpha} 
  \mathbf{Tr}_{\alpha}^{\ast} = \mathbf{Mo}_{\alpha}^{\ast}$.
\end{proposition}

\begin{proof}
  Let $\varphi \in \mathbf{Tr}_{\alpha}^{\ast}$. So $\varphi \in
  \mathbf{Tr}_{\alpha}$. By Proposition~\ref{prop-nearly-extensive}, the
  number $E_{\alpha} \varphi$ is simplest in
  \begin{eqnarray*}
    E_{\alpha} (\mathcal{E}^{\ast}_{\alpha} [\varphi] \cap
    \mathbf{Tr}_{\alpha}) & = & \mathcal{E}_{\alpha}^{\ast} [E_{\alpha}
    \varphi] \cap \mathbf{Mo}_{\alpha} .
  \end{eqnarray*}
  Since $\mathbf{Mo}_{\alpha}^{\ast} \subseteq \mathbf{Mo}_{\alpha}$, we have
  $E_{\alpha} \varphi \sqsubseteq \mathcal{E}_{\alpha}^{\ast} [E_{\alpha}
  \varphi] \cap \mathbf{Mo}_{\alpha}^{\ast}$ so $E_{\alpha} \varphi
  \sqsubseteq \mathfrak{d}_{\alpha}^{\ast} (E_{\alpha} \varphi)$. We deduce
  that $E_{\alpha} \varphi =\mathfrak{d}_{\alpha}^{\ast} (E_{\alpha}
  \varphi)$, so $E_{\alpha} \varphi$ is $\mathcal{E}_{\alpha}^{\ast}$-simple.
  Conversely, let $\mathfrak{a} \in \mathbf{Mo}_{\alpha}^{\ast}$. By
  Proposition~\ref{prop-nearly-extensive} the number $L_{\alpha} \mathfrak{a}$
  is simplest in $L_{\alpha} (\mathcal{E}_{\alpha}^{\ast} [\mathfrak{a}] \cap
  \mathbf{Mo}_{\alpha}) =\mathcal{L}_{\alpha}^{\ast} [L_{\alpha} \mathfrak{a}]
  \cap \mathbf{No}_{\succ, \alpha}$. Since $\mathbf{Tr}_{\alpha}^{\ast}
  \subseteq \mathbf{No}_{\succ, \alpha}$, we have $L_{\alpha} \mathfrak{a}
  \sqsubseteq \mathcal{L}_{\alpha}^{\ast} [L_{\alpha} \mathfrak{a}] \cap
  \mathbf{Tr}_{\alpha}^{\ast}$ so $L_{\alpha} \mathfrak{a} \sqsubseteq
  \sharp_{\alpha}^{\ast} (L_{\alpha} \mathfrak{a})$. We deduce that
  $L_{\alpha} \mathfrak{a} \sqsubseteq \sharp_{\alpha}^{\ast} (L_{\alpha}
  \mathfrak{a})$ is $\mathcal{L}_{\alpha}^{\ast}$-simple.
\end{proof}

\begin{corollary}
  \label{cor-hyperexp-successor-purely-infinite}If $\nu$ is a successor
  ordinal, then $\mathbf{Mo}^{\ast}_{\alpha} = E_{\alpha} 
  \mathbf{No}_{\succ}^{>}$.
\end{corollary}

\

\

\section*{Glossary}

\begin{theglossary}{gly}
  \glossaryentry{$\{ L|R \}$}{simplest number between $L$ and
  $R$}{\pageref{autolab1}}
  
  \glossaryentry{$\mathbb{R} [[\mathfrak{M}]]$}{field of well-based series
  with real coefficients over $\mathfrak{M}$}{\pageref{autolab2}}
  
  \glossaryentry{$\tmop{supp} f$}{support of a series}{\pageref{autolab3}}
  
  \glossaryentry{$\mathfrak{d}_f$}{$\max \tmop{supp} f$}{\pageref{autolab4}}
  
  \glossaryentry{$f_{\succ \mathfrak{m}}$}{truncation $\sum_{\mathfrak{n}
  \succ \mathfrak{m}} f_{\mathfrak{n}} \mathfrak{n}$ of
  $f$}{\pageref{autolab5}}
  
  \glossaryentry{$f_{\succ}$}{$f_{\succ 1}$}{\pageref{autolab6}}
  
  \glossaryentry{$h = f \oplus g$}{$h = f + g$ and $\tmop{supp} f \succ
  g$}{\pageref{autolab7}}
  
  \glossaryentry{$f \trianglelefteqslant g$}{$\tmop{supp} f \succ g -
  f$}{\pageref{autolab8}}
  
  \glossaryentry{$f \prec g$}{$\mathbb{R}^{>}  | f | < | g
  |$}{\pageref{autolab9}}
  
  \glossaryentry{$f \preccurlyeq g$}{$\exists r \in \mathbb{R}^{>}, | f | < r
  | g |$}{\pageref{autolab10}}
  
  \glossaryentry{$f \asymp g$}{$f \preccurlyeq g$ and $g \preccurlyeq
  f$}{\pageref{autolab11}}
  
  \glossaryentry{$\mathbb{S}_{\succ}$}{series $f \in \mathbb{S}$ with
  $\tmop{supp} f \succ 1$}{\pageref{autolab12}}
  
  \glossaryentry{$\mathbb{S}^{\prec}$}{series $f \in \mathbb{S}$ with $f \prec
  1$}{\pageref{autolab13}}
  
  \glossaryentry{$\mathbb{S}^{>, \succ}$}{series $f \in \mathbb{S}$ with $f
  \geqslant 0$ and $f \succ 1$}{\pageref{autolab14}}
  
  \glossaryentry{${\text{\tmtextrm{\tmtextbf{\tmtextup{On}}}}}$}{class of
  ordinals}{\pageref{autolab15}}
  
  \glossaryentry{$\sqsubseteq$}{simplicity relation}{\pageref{autolab16}}
  
  \glossaryentry{$\dot{\omega}^{\gamma}$}{ordinal exponentiation with base
  $\omega$ at $\gamma$}{\pageref{autolab17}}
  
  \glossaryentry{$\mu_-$}{$\mu = \mu_- + 1$ if $\mu$ is a successor ordinal
  and $\mu_- = \mu$ if $\mu$ is a limit ordinal}{\pageref{autolab18}}
  
  \glossaryentry{$\alpha_{/ \omega}$}{$\dot{\omega}^{\mu_-}$ for $\alpha =
  \dot{\omega}^{\mu}$}{\pageref{autolab19}}
  
  \glossaryentry{$\mathbf{U} \Yleft \mathbf{V}$}{the surreal substructure
  $\Xi_{\mathbf{U}} \mathbf{V}$}{\pageref{autolab20}}
  
  \glossaryentry{${\text{\tmtextrm{\tmtextbf{\tmtextup{Smp}}}}}_{\tmmathbf{\Pi}}$}{class
  of $\tmmathbf{\Pi}$-simple elements}{\pageref{autolab21}}
  
  \glossaryentry{$\pi_{\tmmathbf{\Pi}}$}{projection
  ${\text{\tmtextrm{\tmtextbf{\tmtextup{S}}}}} \longrightarrow
  {\text{\tmtextrm{\tmtextbf{\tmtextup{Smp}}}}}_{\tmmathbf{\Pi}}$}{\pageref{autolab22}}
  
  \glossaryentry{$\mathcal{G} [a]$}{class of numbers $b$ with $\exists g, h
  \in \mathcal{G}, g a \leqslant b \leqslant h a$}{\pageref{autolab23}}
  
  \glossaryentry{$\leqangle$}{comparison between sets of strictly increasing
  bijections}{\pageref{autolab24}}
  
  \glossaryentry{$\langle X \rangle$}{function group generated by
  $X$}{\pageref{autolab25}}
  
  \glossaryentry{$X \legeangle Y$}{$X$ and $Y$ are mutually pointwise
  cofinal}{\pageref{autolab26}}
  
  \glossaryentry{$T_r$}{translation $a \longmapsto a +
  r$}{\pageref{autolab27}}
  
  \glossaryentry{$H_s$}{homothety $a \longmapsto sa$}{\pageref{autolab28}}
  
  \glossaryentry{$P_s$}{power function $a \longmapsto
  a^s$}{\pageref{autolab29}}
  
  \glossaryentry{$\mathcal{T}$}{function group $\{ T_r \suchthat r \in
  \mathbb{R} \}$}{\pageref{autolab30}}
  
  \glossaryentry{$\mathcal{H}$}{function group $\{ H_s \suchthat s \in
  \mathbb{R}^{>} \}$}{\pageref{autolab31}}
  
  \glossaryentry{$\mathcal{P}$}{function group $\{ P_s \suchthat s \in
  \mathbb{R}^{>} \}$}{\pageref{autolab32}}
  
  \glossaryentry{$\mathcal{E}'$}{function group $\langle E_n H_s L_n : n \in
  \mathbb{N}, s \in \mathbb{R}^{>} \rangle$}{\pageref{autolab33}}
  
  \glossaryentry{$\mathcal{E}^{\ast}$}{function group $\{ E_n, L_n \suchthat n
  \in \mathbb{N} \}$}{\pageref{autolab34}}
  
  \glossaryentry{$\mathbb{L}$}{field of logarithmic
  hyperseries}{\pageref{autolab35}}
  
  \glossaryentry{$\mathfrak{L}_{< \alpha}$}{group of logarithmic
  hypermonomials of strength $< \: \alpha$}{\pageref{autolab36}}
  
  \glossaryentry{$\mathbb{L}_{< \alpha}$}{field of logarithmic hyperseries of
  strength $< \: \alpha$}{\pageref{autolab37}}
  
  \glossaryentry{$g^{\uparrow \gamma}$}{unique series in $\mathbb{L}$ with $g
  = \left( g^{\uparrow \gamma} \right) \circ
  \ell_{\gamma}$}{\pageref{autolab38}}
  
  \glossaryentry{$L_{\beta}$}{hyperlogarithm function}{\pageref{autolab39}}
  
  \glossaryentry{$\mathfrak{M}_{\beta}$}{class of $L_{< \beta}$-atomic
  series}{\pageref{autolab40}}
  
  \glossaryentry{${\text{\tmtextrm{\tmtextbf{\tmtextup{FE}}}}}_{\mu}$}{functional
  equation}{\pageref{autolab41}}
  
  \glossaryentry{$\mathbf{A}_{\mu}$}{asymptotics axiom}{\pageref{autolab42}}
  
  \glossaryentry{$\mathbf{M}_{\mu}$}{monotonicity axiom}{\pageref{autolab43}}
  
  \glossaryentry{$\mathbf{R}_{\mu}$}{regularity axiom}{\pageref{autolab44}}
  
  \glossaryentry{${\text{\tmtextrm{\tmtextbf{\tmtextup{P}}}}}_{\mu}$}{infinite
  products axiom}{\pageref{autolab45}}
  
  \glossaryentry{$\mathcal{E}_{\beta} [s]$}{class of series $t$ with $\exists
  \gamma < \beta, L_{\gamma} t = L_{\gamma} s$}{\pageref{autolab46}}
  
  \glossaryentry{$\mathfrak{d}_{\beta} (s)$}{$L_{< \beta}$-atomic element of
  $\mathcal{E}_{\beta} [s]$}{\pageref{autolab47}}
  
  \glossaryentry{$E_{\beta}$}{hyperexponential function}{\pageref{autolab48}}
  
  \glossaryentry{$\mathbb{T}_{\succ, \beta}$}{class of $\beta$-truncated
  series}{\pageref{autolab49}}
  
  \glossaryentry{$\mathcal{L}_{\alpha} [s]$}{series $t$ with $\sharp_{\alpha}
  (t) = \sharp_{\alpha} (s)$}{\pageref{autolab50}}
  
  \glossaryentry{$\sharp_{\beta} (s)$}{$\trianglelefteqslant$-maximal
  $\beta$-truncated truncation of $s$}{\pageref{autolab51}}
  
  \glossaryentry{${\text{\tmtextrm{\tmtextbf{\tmtextup{T4}}}}}$}{axiom for
  transseries fields [\cite{Schm01},
  Definition~2.2.1]}{\pageref{autolab52}}
  
  \glossaryentry{$\mathcal{E}_{\alpha}'$}{function group $\langle E_{\gamma}
  \mathcal{H}L_{\gamma} : \gamma < \alpha \rangle$}{\pageref{autolab53}}
  
  \glossaryentry{$\mathcal{E}^{\ast}_{\alpha}$}{function group $\langle E_{<
  \alpha}, \mathcal{P} \rangle$}{\pageref{autolab54}}
  
  \glossaryentry{$\mathcal{L}_{\alpha}'$}{function group $L_{\alpha}
  \mathcal{E}_{\alpha}' E_{\alpha}$}{\pageref{autolab55}}
  
  \glossaryentry{$\mathcal{L}^{\ast}_{\alpha}$}{function group $L_{\alpha}
  \mathcal{E}^{\ast}_{\alpha} E_{\alpha}$}{\pageref{autolab56}}
  
  \glossaryentry{${\text{\tmtextrm{\tmtextbf{\tmtextup{Mo}}}}}_{\alpha}'$}{structure
  of $\mathcal{E}_{\alpha}'$-simple elements}{\pageref{autolab57}}
  
  \glossaryentry{${\text{\tmtextrm{\tmtextbf{\tmtextup{Mo}}}}}^{\ast}_{\alpha}$}{structure
  of $\mathcal{E}^{\ast}_{\alpha}$-simple elements}{\pageref{autolab58}}
  
  \glossaryentry{${\text{\tmtextrm{\tmtextbf{\tmtextup{Tr}}}}}_{\alpha}$}{structure
  of $\mathcal{L}_{\alpha}'$-simple elements}{\pageref{autolab59}}
  
  \glossaryentry{${\text{\tmtextrm{\tmtextbf{\tmtextup{Tr}}}}}^{\ast}_{\alpha}$}{structure
  of $\mathcal{L}^{\ast}_{\alpha}$-simple elements}{\pageref{autolab60}}
\end{theglossary}

\end{document}